\documentclass{article}
\UseRawInputEncoding

\usepackage{epsfig}
\usepackage{amsmath,amsthm,amsfonts,amssymb,euscript}
\usepackage{graphicx}
\usepackage{pgf}
\usepackage{caption}
\usepackage{subcaption}
\usepackage{tikz,tkz-tab}
\usepackage{array,multirow,makecell}
\usepackage{mathtools}
\usepackage{tabularx}
\usepackage{hyperref}
\usepackage{makeidx}
\usepackage{geometry}
\usepackage{cancel}
\usepackage{xcolor}
\usepackage{enumitem} 
\usepackage{ulem}
\usepackage{authblk}
\def\Xint#1{\mathchoice
  {\XXint\displaystyle\textstyle{#1}}%
  {\XXint\textstyle\scriptstyle{#1}}%
  {\XXint\scriptstyle\scriptscriptstyle{#1}}%
  {\XXint\scriptscriptstyle\scriptscriptstyle{#1}}%
  \!\int}
\def\XXint#1#2#3{{\setbox0=\hbox{$#1{#2#3}{\int}$}
    \vcenter{\hbox{$#2#3$}}\kern-.5\wd0}}

\def\fint{\Xint-}

\def\NN{\mathbb N}
\def\RR{\mathbb R}
\def\R{\mathbb R}
\def\I{\text R}
\def\ZZ{\mathbb Z}
\newcommand{\di}{\mathrm{div}}

\newcommand{\holder}{\mathcal{C}^{0,\alpha}}

\def\dint{\text{d}}
\def\I{\text{I}}
\newcommand{\per}{\mathrm{per}}
\def\div{\operatorname{div}}
\def\11{\mathbf{1}}
\begin{document}
\numberwithin{equation}{section}
\newtheorem{theoreme}{Theorem}[section]
\newtheorem{proposition}[theoreme]{Proposition}
\newtheorem{remarque}[theoreme]{Remark}
\newtheorem{lemme}[theoreme]{Lemma}
\newtheorem{corollaire}[theoreme]{Corollary}
\newtheorem{definition}[theoreme]{Definition}
\newtheorem{exemple}[theoreme]{Example}

\title{Homogenization of the $p-$Laplace equation in a periodic setting with a local defect}
\author[1]{S. Wolf}
\affil[1]{{\footnotesize Universit\'e de Paris-Cit\'e, Laboratoire Jacques-Louis Lions, F-75013 Paris}}

\maketitle

\begin{abstract}
In this paper, we consider the homogenization of the $p-$Laplace equation with a periodic coefficient that is perturbed by a local defect. This setting has been introduced in~\cite{BLLMilan,BLLcpde} in the linear setting $p=2$. We construct the correctors and we derive convergence results to the homogenized solution in the case $p > 2$ under the assumption that the periodic correctors are non degenerate. 
\end{abstract}

\tableofcontents

\section{Introduction}

This paper is concerned with the homogenization of non-linear degenerate elliptic equations in a periodic setting with defects. More precisely, we are interested in $p-$Laplacian type equations that are defined, for some $p \geq 2$, as
\begin{equation}
\begin{cases}
\begin{aligned}
- \div a\left(\frac{\cdot}{\varepsilon} \right)\nabla u_{\varepsilon} \left| \nabla u_{\varepsilon} \right|^{p-2} & = f \quad \text{in} \quad \Omega \\
u_{\varepsilon} & \in W^{1,p}_0(\Omega) 
\end{aligned}
\end{cases}
\label{eq:nonlin}
\end{equation} 
for a fixed bounded domain $\Omega \subset \RR^d$, $d \geq 1$ and $f \in L^{p'}(\Omega)$. For $p=2$, we recover the standard linear conductivity equation.
 In~\eqref{eq:nonlin}, the scalar-valued coefficient $a$ is assumed to be of the form
\begin{equation}
a = a^{\per} + \widetilde{a},
\label{eq:intro_formecoeff}
\end{equation} where $a^{\per}$ is a periodic coefficient with standard coercivity and boundedness condition and $\widetilde{a}$ is a perturbation of $a^{\per}$ such that $\widetilde{a} \in L^q(\RR^d)$ for some $1 \leq q \leq \frac{p}{p-1}$. We assume that the coefficient $a$ itself is coercive and bounded and we choose $\lambda > 0$ such that
\begin{equation}  
  \forall y \in \RR^d, \quad \lambda^{-1} < a^{\per}(y) < \lambda \quad \text{and} \quad  \lambda^{-1} < a(y) < \lambda.
  \label{eq:coercif}
  \end{equation}
For fixed $\varepsilon > 0$, Problem~\eqref{eq:nonlin} is well-posed and corresponds to the Euler-Lagrange equation of the minimization Problem
\begin{equation}
\min_{v \in W^{1,p}_0(\Omega)} \left\{ \frac{1}{p} \int_{\Omega} a \big( \frac{\cdot}{\varepsilon} \big) \left| \nabla v \right|^p - \int_{\Omega} f v \right\}.
\label{eq:min}
\end{equation}
The behaviour of~\eqref{eq:nonlin} when $\varepsilon \rightarrow 0$ has been studied in the absence of perturbation, \textit{i.e.} when $a = a^{\per}$. It corresponds to a particular case of the homogenization of the equation
\begin{equation}
- \mathrm{div} A\big( \frac{\cdot}{\varepsilon}, \nabla u_{\varepsilon} \big) = f
\label{eq:introgen}
\end{equation}
 under general growth and continuity conditions for the operator $A(y,\xi)$ (in our case, we have that $A(y) = a^{\per}(y)\xi |\xi|^{p-2}$). The homogenized limit of~\eqref{eq:introgen} is derived in \cite{fusco1985further,fusco1986homogenization}.
It is proved that $u_{\varepsilon}$ converges in the $W^{1,p}-$weak topology, when $\varepsilon \rightarrow 0$, to $u^*$ which is defined by the homogenized equation
\begin{equation}
\begin{cases}
\begin{aligned}
- \text{div} A^*(\nabla u^*) & = f \\
u^* & \in W^{1,p}_0(\Omega),
\end{aligned}
\end{cases}
\label{eq:homogeneisee}
\end{equation} 
where, for $\xi \in \RR^d$, the homogenized operator is 
$$A^*(\xi) := \int_Q A\big(y,\xi + \nabla w_{\xi}(y) \big)\text{d}y,$$
and the function $w_{\xi} \in W^{1,p}_{\per}(Q)$ is the corrector in the direction $\xi$ given as the periodic solution (up to an additive constant) to the equation
\begin{equation}
-\text{div} A(\cdot,\xi + \nabla w_{\xi}) = 0.
\label{eq:intro_nonlin_cor}
\end{equation}
The strong convergence of the gradient
\begin{equation}
\nabla u_{\varepsilon} - \nabla u^* - \nabla w_{\nabla u^*}(./\varepsilon) \underset{\varepsilon \rightarrow 0}{\longrightarrow} 0 \quad \text{in} \quad L^p(\Omega)
\label{eq:intro_cvforte}
\end{equation}
has been obtained in \cite{dal1990correctors} with $\nabla u^*$ replaced by its discretization at small scale $\varepsilon$, for measurability reasons, see Section~\ref{sect:results} below for the details. The periodic homogenization of the integral functionals corresponding to \eqref{eq:nonlin} is exposed in  e.g. \cite{braides2002gamma}. The stochastic case has been studied qualitatively in~\cite{dal1986nonlinear}. Recently, quantitative results for non-linear stochastic problems have been obtained in \cite{fischer2019optimal} with optimal convergence rates for non-degenerate non-linear operators with quadratic growth, see also~\cite{wang2018quantitative} for the deterministic case. The case of stochastic non-degenerate operators with $p-$growth, $p > 2$, is addressed in~\cite{clozeau2021quantitative}. 

\medskip

In this paper, we study Equation \eqref{eq:intro_nonlin_cor} when the perturbation $\widetilde{a}$ belongs to the space $L^q(\RR^d)$ for $1 \leq q \leq \frac{p}{p-1}$ and to some H\"{o}lder space (see Theorem~\ref{th:existencecor} below). We then derive the homogenized limit of the sequence $(u_{\varepsilon})_{\varepsilon > 0}$ and we study the convergence of the two-scale expansion~\eqref{eq:intro_cvforte} when we use, on the one hand, the periodic corrector and, on the other hand, the non-periodic corrector (corresponding respectively to the solutions of~\eqref{eq:intro_nonlin_cor} when $A(y,\xi) = a^{\per}(y)\xi|\xi|^{p-2}$ and $A(y,\xi)= a(y)\xi|\xi|^{p-2}$). We also illustrate the quantitative convergence of the two-scale expansion~\eqref{eq:intro_cvforte} in the one-dimensional setting and prove that, in this case, using the non-periodic corrector instead of the periodic corrector in fact improves the quality of convergence of~\eqref{eq:intro_cvforte}. The main difficulty of this work is that Equation~\eqref{eq:intro_nonlin_cor} is posed on the whole space $\RR^d$. One major tool to obtain the strong convergence \eqref{eq:intro_cvforte} in the non-periodic case is the continuity of the application $\xi \longmapsto \nabla w_{\xi}$ (see Theorem~\ref{th:th_nonlin} below). This will be proved under one of the two Assumptions \textbf{(A4)} or \textbf{(A4)'} below.

\medskip

Before stating our main results, we would like to comment on the special case $p=2$ for the homogenization of Problem~\eqref{eq:nonlin}. This problem is very standard since the 70's for a periodic coefficient $a$, see e.g. \cite{BLP} for qualitative results and~\cite{AL} for quantitative results. It is worth mentioning that, in this case, the homogenization objects such as correctors and homogenized limits are explicit and very easy to compute.  
The setting \eqref{eq:nonlin}-\eqref{eq:intro_formecoeff} has first been introduced in \cite{BLLMilan} for $q=2$. It models local defects that could appear, at the microscale, in a periodic background. The results obtained have been generalized to the case $1 \leq q < +\infty$ in \cite{BLLcpde,BLLfutur1} and convergence rates have been proved in \cite{BJL}. In \cite{goudey}, a new non-periodic setting has been introduced to model defects that are not local but rare at infinity. We stress that, in \cite{BLLMilan,BLLcpde,BLLfutur1,BJL,goudey}, as in the present work, the macroscopic behaviour of the oscillating solution remains the same as in the case of a periodic coefficient. This will be expressed, for the non-linear case, in Theorem~\ref{th:homog_qualititative} below.  

\medskip

The paper is organized as follows. The main results of the paper are presented in Section~\ref{sect:results}. We develop in Section~\ref{sect:1D} explicit calculations in the one dimensional setting and obtain convergence results. We then turn in Section~\ref{sect:cor} to the existence of the non-periodic correctors in any dimension. The properties of the non-periodic corrector are proved in Section~\ref{sect:propcor}. We then derive qualitiative homogenization results in Section~\ref{sect:qual}. We finally prove in Section~\ref{sect:weakcont} a weaker continuity result for the mapping $\xi \longmapsto \nabla w_{\xi}$ that is enough to derive qualitative homogenization. We recall in Appendix~\ref{annexe:proof_prop2.1} the proof of classical results in the periodic case. Technical inequalities are gathered in Appendix~\ref{sect:ineq}.

\section{Main results}
\label{sect:results}

\subsection{Notations} In the whole paper, $d \geq 1$ will be the dimension of the ambient space. The standard unit cube $\left(-\frac{1}{2},\frac{1}{2}\right)^d$ will be denoted by $Q$. The euclidian norm will be written $|\cdot|$ as well as the Lebesgue measure of a measurable subset of $\RR^d$. Let $\Omega$ be a bounded domain of $\RR^d$. If $1 < q < +\infty$ is an exponent, we define its conjugate by $q' := q/(q-1)$.  The euclidian open ball of $\RR^d$ centered in $x$ and of radius $r> 0$ will be written $B(x,r)$. If $x=0$, we write $B_r := B(0,r)$. We use similar notations for cubes, namely $Q(x,r)$ and $Q_r$. We define the mean-value operation for a measurable and integrable function $u$ by
$$\fint_{B(x,r)} u := \frac{1}{|B(x,r)|} \int_{B(x,r)} u.$$
The indicator function of a measurable set $A$ is denoted $1_A$.

\medskip The standard Lebesgue and Sobolev spaces are denoted by $L^q(\Omega)$ and $W^{1,q}(\Omega)$. The associated norms are
$$\|u\|_{L^q(\Omega)} := \left(\int_{\Omega} |u|^q \right)^{1/q} \quad \text{and} \quad \|u\|_{W^{1,q}(\Omega)} := \left(\int_{\Omega} |u|^q \right)^{1/q} + \left(\int_{\Omega} |\nabla u|^q \right)^{1/q}.$$ 
The space $L^{q}_{\per}$ (resp. $W^{1,q}_{\per}$) denotes the set of functions that are periodic and locally belong to $L^q$ (resp.~$W^{1,q}$). Theses two spaces are endowed with the norms
$$\|u\|_{L^q_{\per}} := \left(\int_{Q} |u|^q \right)^{1/q} \quad \text{and} \quad \|u\|_{W^{1,q}_{\per}} := \left(\int_Q |u|^q \right)^{1/q} + \left(\int_Q |\nabla u|^q \right)^{1/q}.$$ 
The space of uniformly $L^q$ (resp. $W^{1,q}$) functions is denoted by $L^q_{unif}$ (resp. $W^{1,q}_{unif}$). These spaces are endowed with the norms
$$\|u\|_{L^q_{unif}(\RR^d)} := \sup_{x \in \RR^d} \|u\|_{L^q(x+Q)} \quad \text{and} \quad \|u\|_{W^{1,q}_{unif}(\RR^d)} := \sup_{x \in \RR^d} \|u\|_{W^{1,q}(x+Q)}.$$
For $0 <\alpha < 1$, the space $\mathcal{C}^{0,\alpha}$ refers to the standard H\"{o}lder space endowed with the norm
$$\|u\|_{\mathcal{C}^{0,\alpha}} := \|u\|_{L^{\infty}} + \sup_{x \neq y} \frac{|u(x) - u(y)|}{|x-y|^{\alpha}}.$$
We define, for $\delta >0$, the discretization operator $M_{\delta} : L^q(\Omega) \longrightarrow L^q(\Omega)$ introduced in \cite{dal1990correctors,fusco1986homogenization}. If $\phi \in L^q(\Omega)$, we set
\begin{equation}
M_{\delta} \phi := \sum_{k \in \ZZ^d \ \text{s.t.} \ \delta(Q+k) \subset \Omega} \left( \fint_{\delta(Q + k)} \phi \right)1_{\delta(k+Q)}.
\label{eq:M_eps}
\end{equation}
It is clear that $M_{\delta}$ is linear and bounded over $L^q(\Omega)$ and that $M_{\delta} \phi \underset{\delta\rightarrow 0}{\longrightarrow} \phi$ in $L^q(\Omega)$.

\medskip

\subsection{The periodic case}

We assume in this paragraph that $\widetilde{a} = 0$ in~\eqref{eq:intro_formecoeff}.
In this case, the corrector equation is, according to~\eqref{eq:intro_nonlin_cor}:
\begin{equation}
- \text{div}\ a^{\per}(y)(\xi + \nabla w_{\xi}^{\per})|\xi + \nabla w_{\xi}^{\per}|^{p-2} =0.
\label{eq:cor_per_2}
\end{equation} The equation~\eqref{eq:cor_per_2} admits a unique solution $w_{\xi}^{\per}$ in the space $W^{1,p}_{\per}(Q)/\RR$. Indeed, the weak formulation of \eqref{eq:cor_per_2} is
\begin{equation}
\forall \phi \in W^{1,p}_{\per}(Q)/\RR, \quad \int_Q a^{\per}(y)(\xi + \nabla w_{\xi}^{\per})|\xi + \nabla w_{\xi}^{\per}|^{p-2} \cdot \nabla \phi = 0,
\label{eq:cor_per}
\end{equation}
which is exactly the Euler-Lagrange equation of the minimization Problem
\begin{equation}
\min_{v \in W^{1,p}_{\per}(Q)/\RR} \left\{ \frac{1}{p} \int_Q a^{\per}(y)\big|\xi + \nabla v \big|^p \text{d}y \right\}.
\label{eq:min_per}
\end{equation}
It is easy to see that the functional appearing in Problem~\eqref{eq:min_per} is strictly convex, coercive and continuous with respect to $\nabla v$. Thus,~\eqref{eq:min_per} admits a minimizer $w_{\xi}^{\per}$, the gradient of which is unique. We impose that $\fint_{Q} w_{\xi}^{\per}=0$ so that $w_{\xi}^{\per}$ is itself unique. Besides, we have the following Proposition (see \cite{fusco1986homogenization,fusco1985further,dal1990correctors} or Appendix~\ref{annexe:proof_prop2.1} below for a proof) gathering the main properties of the application $\xi \longmapsto \nabla w_{\xi}^{\per}$:

\begin{proposition} Let $a^{\per} : \RR^d \longrightarrow \RR$ be a periodic and Lipschitz continuous coefficient satisfying~\eqref{eq:coercif}.
\begin{enumerate}[label=(\roman*)]
\item The map $\xi \longmapsto \nabla w_{\xi}^{\per}$ is homogeneous in the sense that for all $\xi \in \RR^d$ and $t \in \RR$, 
\begin{equation}
 \nabla w_{t\xi}^{\per} = t \nabla w_{\xi}^{\per}.
\label{eq:homog_per}
\end{equation}
\item There exists an exponent $\alpha = \alpha(d,p,a^{\per}) > 0$ such that for all $\xi \in \RR^d$, $\nabla w_{\xi}^{\per} \in \mathcal{C}^{0,\alpha}(\RR^d)$. Moreover, there exists a constant $C = C(d,p,a^{\per}) > 0$ such that
\begin{equation}
\label{eq:estim_holder_per}
\| \nabla w_{\xi}^{\per} \|_{L^p_{unif}(\RR^d)} \leq C|\xi| \quad \text{and} \quad \| \nabla w_{\xi}^{\per} \|_{\mathcal{C}^{0,\alpha}(\RR^d)} \leq C|\xi|.  
\end{equation}
\item There exists a constant $C = C(d,p,a^{\per}) > 0$ such that for all $\xi, \eta \in \RR^d$,
\begin{equation}
\| \nabla w_{\xi}^{\per} - \nabla w_{\eta}^{\per} \|_{L^p_{unif}(\RR^d)} \leq C \left[ |\xi|^{1-\beta} + |\eta|^{1 - \beta} \right]|\xi - \eta|^{\beta}, \quad \beta := \frac{1}{p-1}.
\label{eq:continuité_holder}
\end{equation}
\item There exists a constant $C = C(d,p,a^{\per}) > 0$ such that for all $\xi, \eta \in \RR^d$, \begin{equation}
\big\| \nabla w_{\xi}^{\per} - \nabla w_{\eta}^{\per} \big\|_{L^{\infty}(\RR^d)} \leq C \big[|\xi|^{1-\gamma} + |\eta|^{1-\gamma} \big]| \xi - \eta|^{\gamma}, \quad \gamma := \frac{\beta p}{p + d/\alpha},
\label{eq:continuité_Linfty}
\end{equation}
where $\beta$ is defined in \eqref{eq:continuité_holder} and $\alpha$ is given by (ii).
\end{enumerate}
\label{prop:periodique}
\end{proposition}

It is proved in~\cite{fusco1986homogenization} that $u_{\varepsilon}$ converges weakly~in $W^{1,p}(\Omega)$ to $u^*$ which is defined by~\eqref{eq:homogeneisee}. Note that~\eqref{eq:homogeneisee} is well posed due to the monoticity of $A^*$ (see \cite{dal1990correctors} and \cite[Corollary 8.1]{le2018nonlinear}). Convergence in the $L^{\infty}-$norm may be obtained in the one-dimensional setting, see Section~\ref{sect:1D} below.

\subsection{Results in the non-periodic case}

The first result of this contribution concerns the corrector equation \eqref{eq:intro_cvforte} in the setting \eqref{eq:nonlin}-\eqref{eq:intro_formecoeff}. For a fixed direction $\xi \in \RR^d$, this equation, posed on the whole space~$\RR^d$, is
\begin{equation}
- \text{div}\ a(y)(\xi + \nabla w_{\xi})|\xi + \nabla w_{\xi}|^{p-2} = 0,
\label{eq:cornonlin}
\end{equation}
where the coefficient $a$ is of the form $a := a^{\per} + \widetilde{a}$ and $a^{\per}$ is a periodic coefficient. We assume that $a$ and $a^{\per}$ satisfy the following assumptions:

\medskip

\textbf{(A1) there exists $\lambda > 0$ such that \eqref{eq:coercif} is satisfied;}

\medskip

\textbf{(A2) the coefficients $a$ and $a^{\per}$ are Lipschitz-continuous;}

\medskip

\textbf{(A3) the perturbation $\widetilde{a}$ vanishes at infinity in the sense that $\widetilde{a} \in L^{p'}(\RR^d)$.}

\medskip

A few comments are in order. First, if $\widetilde{a}$ satisfies $\widetilde{a} \in \mathcal{C}^{0,1}(\RR^d)$ and $\widetilde{a} \in L^q(\RR^d)$ for some $q \leq p'$ then $\widetilde{a}$ satisfies \textbf{(A3)} by interpolation. Second, Assumption \textbf{(A2)} allows to ensure local regularity (see Proposition~\ref{prop:periodique} above) of the periodic and non-periodic correctors. Finally, the assumptions of~\cite{BLLMilan} in the linear setting correspond to the case $p=2$ in the assumptions~\textbf{(A1)}-\textbf{(A2)}-\textbf{(A3)} above.



\medskip

We now consider the equation \eqref{eq:cornonlin} when the coefficient $a$ has the non-periodic structure \eqref{eq:intro_formecoeff}. For $u \in L^{\infty}(\RR^d)$, we define the spaces
\begin{equation}
\mathcal{W}_u := \left\{v \in W^{1,1}_{loc}(\RR^d), \quad \int_{\RR^d} \big|\nabla v\big|^p + \int_{\RR^d } \big| u \big|^{p-2} \big|\nabla v\big|^2 < +\infty \right\} \quad \text{and} \quad W_u := \mathcal{W}_u / \RR.
\label{eq:Wu}
\end{equation}
The space $W_u$ is endowed with the norm
\begin{equation}
\| v \|_{W_u} := \| \nabla v \|_{L^p(\RR^d)} + \big\| |u|^{\frac{p-2}{2}} \nabla v \big\|_{L^2(\RR^d)}.
\label{eq:normeWu}
\end{equation}
In the sequel, we denote undifferently functions and equivalence classes for the relation: $f \sim g$ if and only if $f - g$ is almost everywhere constant.
Lemma~\ref{lem:banach} below gathers some properties satisfied by spaces of the form~\eqref{eq:Wu}. In order to solve~\eqref{eq:cornonlin}, we seek for $w_{\xi}$ of the form $w_{\xi} = w_{\xi}^{\per} + \widetilde{w_{\xi}}$ where $w_{\xi}^{\per}$ is the solution to~\eqref{eq:cor_per} such that $\fint_Q w_{\xi}^{\per} = 0$. We transform the equation~\eqref{eq:cornonlin} into 
\begin{equation}
-\text{div} a \left[ \big|\xi + \nabla w_{\xi}^{\per} + \nabla \widetilde{w_{\xi}} \big|^{p-2}(\xi + \nabla w_{\xi}^{\per} + \nabla \widetilde{w_{\xi}}) - \big| \xi  + \nabla w_{\xi}^{\per} \big|^{p-2}(\xi + \nabla w_{\xi}^{\per}) \right] = \text{div}(h),
\label{eq:th2.2_gen2}
\end{equation}
where
\begin{equation}
h := \widetilde{a} (\xi + \nabla w_{\xi}^{\per})|\xi + \nabla w_{\xi}^{\per}|^{p-2}.
\label{eq:f_2}
\end{equation}
Assumption~\textbf{(A3)} and Proposition~\ref{prop:periodique} (ii) ensure that $h \in L^{p'}(\R^d)^d$.
\begin{definition}
We say that $\widetilde{w_{\xi}} \in W_{\xi + \nabla w_{\xi}^{\per}}$ is a solution in the weak sense in $W_{\xi + \nabla w_{\xi}^{\per}}$ to~\eqref{eq:th2.2_gen2} if for all $w \in W_{\xi + \nabla w_{\xi}^{\per}}$,
$$\int_{\R^d} a \left[ \big|\xi + \nabla w_{\xi}^{\per} + \nabla \widetilde{w_{\xi}} \big|^{p-2}(\xi + \nabla w_{\xi}^{\per} + \nabla \widetilde{w_{\xi}}) - \big| \xi  + \nabla w_{\xi}^{\per} \big|^{p-2}(\xi + \nabla w_{\xi}^{\per}) \right]\cdot\nabla w =- \int_{\R^d} h \cdot \nabla w.$$
\label{def:def}
\end{definition}
We easily check using Appendix~\ref{sect:ineq} that each integral appearing in Definition~\ref{def:def} is convergent. Note that if $\widetilde{w_{\xi}}$ is a solution to~\eqref{eq:th2.2_gen2} in the sense of Definition~\ref{def:def}, then it is a solution to~\eqref{eq:th2.2_gen2} in the distribution sense but it is not clear that the converse holds true. This is true if the weight $\xi + \nabla w_{\xi}^{\per}$ satisfies Assumption \textbf{(A4)'} below (see also Remark~\ref{re:A4_prime}).
\begin{theoreme}[Existence of the non-periodic correctors] Assume that the coefficient $a = a^{\per} + \widetilde{a}$ satisfies Assumptions \textbf{(A1)-(A2)-(A3)}. Then, for all $\xi \in \RR^d$, there exists a unique solution $w_{\xi}$ to equation~\eqref{eq:cornonlin} such that $w_{\xi} \in W^{1,1}_{loc}(\RR^d)$, $w_{\xi} = w_{\xi}^{\per} + \widetilde{w_{\xi}}$, where $\widetilde{w_{\xi}} \in W_{\xi + \nabla w_{\xi}^{\per}}$ is solution in the weak sense in $W_{\xi + \nabla w_{\xi}^{\per}}$ to \eqref{eq:th2.2_gen2}-\eqref{eq:f_2}. 
\label{th:existencecor}
\end{theoreme}

In view of Theorem~\ref{th:existencecor}, we denote in the sequel $\widetilde{w_{\xi}} \in \mathcal{W}_{\xi + \nabla w_{\xi}^{\per}}$ the unique function such that $\displaystyle\int_{Q} \widetilde{w_{\xi}} = 0$. The function $w_{\xi}^{\per} + \widetilde{w_{\xi}}$ is a solution to~\eqref{eq:cornonlin} and $\widetilde{w_{\xi}}$ solves~\eqref{eq:th2.2_gen2}-\eqref{eq:f_2} in the sense of Definition~\ref{def:def}.
 We also define
\begin{equation}
w_{\xi} := w_{\xi}^{\per} + \widetilde{w_{\xi}} \in W^{1,p}_{\per}(Q) + \mathcal{W}_{\xi + \nabla w_{\xi}^{\per}}.
\label{eq:w_xi}
\end{equation}
The analogous properties of those given in Proposition~\ref{prop:periodique} are given in Theorem~\ref{th:th_nonlin} below for the non-linear correctors $w_{\xi}$, $\xi \in \R^d$. In order to obtain continuity results for the application $\xi \longmapsto \nabla w_{\xi}$, we need the following assumption: 

\medskip

\textbf{(A4) There exists $c > 0$ independent of $\xi \in \RR^d$ such that $|\xi + \nabla w_{\xi}^{\per}| \geq c|\xi|$ on $Q$.}
\medskip

\noindent We comment in Subsection~\ref{sect:comments} on this assumption. We are able to prove the following Theorem.

\begin{theoreme} Let $a := a^{\per} + \widetilde{a}$ be a non-periodic coefficient satisfying Assumptions \textbf{(A1)-(A2)-(A3)}. For $\xi \in \RR^d$, let $w_{\xi}$ be defined by~\eqref{eq:w_xi}. 
\begin{enumerate}[label=(\roman*)]
\item The map $\xi \longmapsto \nabla w_{\xi}$ is homogeneous in the sense that for all $\xi \in \RR^d$ and $t \in \RR$, 
\begin{equation}
\nabla w_{t\xi} = t \nabla w_{\xi}.
\label{eq:homog_nonper}
\end{equation}
\item There exists a constant $C = C(d,p,a) > 0$ and an exponent $\alpha = \alpha(d,p,a) > 0$ such that for all $\xi \in \RR^d$, $\nabla w_{\xi} \in L^p_{unif}(\RR^d)$, $\nabla w_{\xi} \in \mathcal{C}^{0,\alpha}(\RR^d)$ and, moreover, we have the estimates
\begin{equation}
\label{eq:estim_holder_per}
\| \nabla w_{\xi} \|_{L^p_{unif}(\RR^d)} \leq C|\xi| \quad \text{and} \quad \| \nabla w_{\xi} \|_{\mathcal{C}^{0,\alpha}(\RR^d)} \leq C|\xi|.  
\end{equation}
\item Assume that Assumption~\textbf{(A4)} is satisfied. Then there exists a constant $C = C(d,p,a,c) > 0$ independent of $\xi$ and $\eta$ such that, for all $\xi, \eta \in \RR^d$,
 \begin{equation}
\| \nabla \widetilde{w_{\xi}} - \nabla \widetilde{w_{\eta}} \|_{L^p(\RR^d)} \leq C \left( |\xi|^{1-\widetilde{\beta}} + |\eta|^{1 -\widetilde{\beta}} \right) | \xi - \eta|^{\widetilde{\beta} }, \quad \widetilde{\beta} = \frac{\gamma}{p-1}\min(1,p-2).
\label{eq:cont}
\end{equation}
where $\gamma$ is given by~\eqref{eq:continuité_Linfty}.
\item Assume that Assumption~\textbf{(A4)} is satisfied. Then there exists a constant $C = C(d,p,a,c) > 0$ and an exponent $\widetilde{\gamma} > 0$ both$  $ independent of $\xi$ and $\eta$ such that, for all $\xi, \eta \in \RR^d$, \begin{equation}
\| \nabla w_{\xi} - \nabla w_{\eta} \|_{L^{\infty}(\RR^d)} \leq C \left( |\xi|^{1-\widetilde{\gamma}} + |\eta|^{1 -\widetilde{\gamma}} \right) | \xi - \eta|^{\widetilde{\gamma}}.
\label{eq:contlinfinity}
\end{equation}
\end{enumerate}
\label{th:th_nonlin}
\end{theoreme}
An important tool to obtain~Theorem~\ref{th:th_nonlin} (iii) is the following Theorem:

\begin{theoreme} Let $a := a^{\per} + \widetilde{a}$ be a non-periodic coefficient such that \textbf{(A1)-(A2)-(A3)-(A4)} are satisfied. 
For all $\xi \in \RR^d$, we have that $\nabla \widetilde{w_{\xi}} \in L^{p'}(\RR^d)$ and the estimate 
\begin{equation}
\| \nabla \widetilde{w_{\xi}} \|_{L^{p'}(\RR^d)} \leq C|\xi|
\label{th:estimLp_prime}
\end{equation} holds true where $C = C(d,p,a,c) >0$ is a constant independent of $\xi$.
\label{th:coercivitypercor}
\end{theoreme}

\begin{remarque}
Note that, under Assumption \textbf{(A4)}, the non-periodic part $\nabla \widetilde{w_{\xi}}$ of the corrector has the same integrability as the defect $\widetilde{a}$ at infinity. This is reminiscent of the linear case $p=2$, see~\cite{BLLcpde}.
\label{rem:th2.5}
\end{remarque}


Using Theorem~\ref{th:th_nonlin}, we can prove qualitative results concerning the homogenization of~ \eqref{eq:nonlin} in the non-periodic setting. 

\begin{theoreme}
\label{th:homog_qualititative}
Let $\Omega$ be a bounded smooth domain, $f \in L^{p'}(\Omega)$, $a := a^{\per} + \widetilde{a}$ be a scalar-valued coefficient satisfying Assumptions \textbf{(A1)-(A2)-(A3)}. For $\varepsilon>0$, let $u_{\varepsilon} \in W^{1,p}_0(\Omega)$ be the solution to~\eqref{eq:nonlin}. 
\begin{enumerate}[label=(\roman*)]
\item We have that $u_{\varepsilon} \underset{\varepsilon \rightarrow 0}{\relbar\joinrel\rightharpoonup} u^*$ weakly in $W^{1,p}(\Omega)$ and $u_{\varepsilon} \underset{\varepsilon \rightarrow 0}{\longrightarrow} u^*$ strongly in $L^p(\Omega)$, where $u^*$ solves Problem~\eqref{eq:homogeneisee} and 
\begin{equation}
\forall \xi \in \RR^d, \quad a^*(\xi) = \int_{Q} a^{\per}(y)(\xi + \nabla w_{\xi}^{\per})|\xi +\nabla w_{\xi}^{\per}|^{p-2}\text{d}y.
\label{eq:coeffhomog}
\end{equation}
Besides, we have the $L^{p'}(\Omega)-$weak convergence $a(./\varepsilon)\nabla u_{\varepsilon}|\nabla u_{\varepsilon}|^{p-2} \underset{\varepsilon\rightarrow 0}{\relbar\joinrel\rightharpoonup} a^*(\nabla u^*)$.
\item Assume that \textbf{(A4)} is satisfied. Then, we have the strong convergence 
\begin{equation}
\nabla u_{\varepsilon} - \nabla u^* - \nabla w_{M_{\varepsilon} \nabla u^*} \big(\frac{\cdot}{\varepsilon} \big) \underset{\varepsilon \rightarrow 0}{\longrightarrow} 0 \quad \text{in} \quad L^p(\Omega),
\label{eq:cvforte}
\end{equation}
where $M_{\varepsilon}$ is defined by~\eqref{eq:M_eps}.
\item We have the strong convergence 
\begin{equation}
\nabla u_{\varepsilon} - \nabla u^* - \nabla w^{\per}_{M_{\varepsilon} \nabla u^*} \big(\frac{\cdot}{\varepsilon} \big) \underset{\varepsilon \rightarrow 0}{\longrightarrow} 0 \quad \text{in} \quad L^p(\Omega),
\label{eq:cvforte_per}
\end{equation}
where $M_{\varepsilon}$ is defined by~\eqref{eq:M_eps}.
\end{enumerate}
\end{theoreme}

\noindent We stress that, instead of assuming \textbf{(A4)}, Theorem~\ref{th:homog_qualititative} can be proved under the assumption that the mapping
\begin{equation}
\Phi_p:
\begin{cases}
\begin{aligned}
\RR^d & \longrightarrow L^p_{unif}(\RR^d) \\
\xi & \longmapsto \nabla w_{\xi}
\end{aligned}
\end{cases}
\label{eq:xi-w_xi}
\end{equation}
is continuous. This continuity can be obtained under the following Assumption \textbf{(A4)'} which is clearly weaker than Assumption \textbf{(A4)}:

\medskip

\textbf{(A4)' For $\xi \in \R^d$, there exists a constant $C > 0$ that may depend on $\xi$ and $\nabla w_{\xi}^{\per}$ such that the following weighted Poincar\'e-Wirtinger inequality holds true: there exists $r_{min} > 0$ such that for all $R > r_{min}$ and $w \in H^1\big(Q \setminus Q_{1/2}\big)$,
\begin{equation}
\left\||\xi + \nabla w_{\xi}^{\per}(R\cdot)|^{\frac{p-2}{2}} \left(w - \fint_{Q \setminus Q_{1/2}} w \right) \right\|_{L^2(Q \setminus Q_{1/2})} \leq C\big\||\xi + \nabla w_{\xi}^{\per}(R\cdot)|^{\frac{p-2}{2}} \nabla w \|_{L^2(Q \setminus Q_{1/2})}.
\label{eq:A4_prime}
\end{equation}}

\medskip

We comment in Subsection~\ref{sect:comments} on Assumption \textbf{(A4)'} and we will provide a sufficient condition on $\xi + \nabla w_{\xi}^{\per}$ so that~\eqref{eq:A4_prime} is satisfied. 
 We are able to prove the following Theorem:

\begin{theoreme}
Assume that \textbf{(A1)-(A2)-(A3)-(A4)'} are satisfied. Then the mapping $\Phi_p$ defined by~\eqref{eq:xi-w_xi} is continuous. Hence the conclusion of Theorem~\ref{th:homog_qualititative} holds true.
\label{th:cont}
\end{theoreme}

We close this section by mentioning that the results of Theorem~\ref{th:homog_qualititative} can be improved in the one-dimensional setting. We devote Section~\ref{sect:1D} to convergence results in this particular case.

\begin{remarque}
To prove Theorems~\ref{th:homog_qualititative} and~\ref{th:cont}, Assumption~\textbf{(A4)'} can further be weakened into the following one: the set of smooth functions with compact support over $\R^d$, denoted by $\mathcal{C}_0^{\infty}(\RR^d)$, is dense in $W_{\xi + \nabla w_{\xi}^{\per}}$. We show in Lemma~\ref{lem:dense} (see Appendix~\ref{sect:ineq}) that, as pointed out in~\cite{zhikov1998weighted}, the density result is implied by Assumption~\textbf{(A4)'}. Note that, under Assumption~\textbf{(A4)'}, we can easily prove (by density) that \eqref{eq:th2.2_gen2}--\eqref{eq:f_2} admits a unique solution in the distribution sense in $W_{\xi + \nabla w_{\xi}^{\per}}$. 
\end{remarque}

\begin{remarque}
The method of proof of this paper allows to build the non-periodic correctors for a defect $\widetilde{a}$ that belongs to the dual space of $W_{\xi+\nabla w_{\xi}^{\per}}$, see Lemma~\ref{lem:banach} (iii). This is in particular the case if $\widetilde{a} \in L^2(\R^d)$. We are however not able to show that the non-periodic corrector satisfies $\nabla \widetilde{w_{\xi}} \in L^2(\R^d)$ but only that $\nabla \widetilde{w_{\xi}} \in L^2(|\xi + \nabla w_{\xi}^{\per}|^{p-2}d\lambda)$, see Remark~\ref{rem:2.9} below. More generally, building the non-periodic correctors for a defect $\widetilde{a} \in L^{2+\delta} \cap \holder(\R^d)$ is a challenging problem that we are unable to address for now. In the linear setting $p=2$, this was achieved in~\cite{BLLfutur1} by studying the continuity from $L^q(\R^d)$ to $L^q(\R^d)$ for $q > 2$ of the Riesz operator associated to the coefficient $a$.
\end{remarque}

\begin{remarque}
The space $W_{\xi + \nabla w_{\xi}^{\per}}$ is in general different from the space $\overset{\circ}{W^{1,p}} \cap \overset{\circ}{H^1}(\R^d)$, where $\overset{\circ}{W^{1,p}}$ and $\overset{\circ}{H^1}(\R^d)$ are the standard homogeneous Sobolev spaces, unless $\xi + \nabla w_{\xi}^{\per}$ does not vanish. Assume that there exists $x_0 \in Q$ such that $\xi + \nabla w_{\xi}^{\per}(x_0) = 0$. We can assume by invariance translation that $x_0 = 0$. Owing to Proposition~\ref{prop:periodique} (ii), we have that $|\xi + \nabla w_{\xi}^{\per}(x)| \leq C|x|^{\alpha}$ in $Q$. Let $\phi \in \mathcal{D}(Q)$ be such that $\phi = 1$ on $B(0,1/4)$. We define $\Psi := \sum_{k \in \ZZ^d\setminus \{0\}} \frac{1}{|k|^{\delta + \nu}}\phi(|k|^{\nu}(\cdot - k))$, where $\delta, \nu > 0$ will be chosen later. We have 
\begin{equation}
\big\| \nabla \Psi \big\|_{L^p(\R^d)}^p = \sum_{k \in \ZZ^d\setminus \{0\}} \frac{1}{|k|^{p\delta}} \int_{|k|^{-\nu}Q} \big|\nabla \phi(|k|^{\nu}x)\big|^p \dint x = \sum_{k \in \ZZ^d\setminus \{0\}} \frac{\|\nabla\phi\|_{L^p(Q)}^p}{|k|^{ p\delta+d\nu}}.
\label{rk2.9_1}
\end{equation}
Besides, we have that 
\begin{equation}
\begin{aligned}
\big\| \nabla \Psi \big\|_{L^2(|\xi + \nabla w_{\xi}^{\per}|^{p-2}d\lambda)}^2 = \sum_{k \in \ZZ^d\setminus \{0\}} \frac{1}{|k|^{2\delta}} \int_{|k|^{-\nu}Q} \big|\nabla\phi(|k|^{\nu}x)\big|^2 & |\xi + \nabla w_{\xi}^{\per}(x)|^{p-2} \dint x \\&\leq \sum_{k \in \ZZ^d\setminus \{0\}} \frac{C\|\nabla\phi\|_{L^2(Q)}^2}{|k|^{\alpha\nu(p-2) + d\nu + 2\delta}}.
\end{aligned}
\label{rk2.9_2}
\end{equation}
Finally, 
\begin{equation}
\big\| \nabla \Psi \big\|_{L^2(\R^d)}^2 = \sum_{k \in \ZZ^d\setminus \{0\}} \frac{1}{|k|^{2\delta}} \int_{|k|^{-\nu}Q} \big|\nabla\phi(|k|^{\nu}x)\big|^2  = \sum_{k \in \ZZ^d\setminus \{0\}} \frac{\|\nabla \phi \|_{L^2(Q)}^2}{|k|^{d\nu + 2\delta}}.
\label{rk2.9_3}
\end{equation}
We fix $\nu \in (\frac{d}{d+2},1)$ and $\delta \in \big(\max\{\frac{d(1-\nu)}{p},\frac{d(1-\nu) - \alpha\nu(p-2)}{2} \} ,\frac{d(1-\nu)}{2} \big)$ so that $\nabla \Psi \in L^2(|\xi + \nabla w_{\xi}^{\per}|^{p-2} d \lambda) \cap L^p(\R^d)$ and $\nabla \Psi \notin L^2(\R^d)$. Note that $\nabla \Psi \in \mathcal{C}^{0,\delta/\nu}(\R^d)$ so that this counter-example is consistent with the result of Theorem~\ref{th:th_nonlin} (ii) since $\delta < \nu$.
\label{rem:2.9}
\end{remarque}


\subsection{Comments on the Assumptions}
\label{sect:comments}

\paragraph{On Assumption \textbf{(A4)}.} Assumption \textbf{(A4)} is quite restrictive but is known to be true in dimension~1. Besides, it is proved in~\cite[Lemma~2, p. 404]{cherednichenko2004full} that it is also satisfied in dimension $d=2$.

\medskip

 We show here that Assumption \textbf{(A4)} is satisfied for laminate materials (in any dimension). Suppose that $a^{\per}(x) = a_0(x_1)$ where $a_0: \R \longrightarrow \R$ is a periodic function. Let $\xi \neq 0$. In this case, the periodic corrector $w_{\xi}^{\per}$ is a function of the first variable \textit{i.e.} $w_{\xi}^{\per}(x) = w_{\xi}^0(x_1)$ and~\eqref{eq:cor_per_2} becomes
\begin{equation}
- \frac{d}{dx_1}\left(a_0(x_1)\left(\xi_1 + \frac{dw_{\xi}^0}{dx_1} \right)\left|\left(\xi_1+\frac{dw_{\xi}^0}{dx_1} \right)^2+ \xi_2^2 + \cdots + \xi_d^2\right|^{\frac{p-2}{2}} \right) = 0.
\label{eq:laminate}
\end{equation}
If there exists $i \geq 2$ such that $\xi_i \neq 0$, then $|\xi + \nabla w_{\xi}^{\per}| \geq |\xi_i| > 0$. In the other cases, $\xi_i = 0$ for $i \geq 2$, thus $\xi_1\neq 0$ and~\eqref{eq:laminate} reduces to:
\begin{equation}
- \frac{d}{dx_1}\left(a_0(x_1)\left(\xi_1 + \frac{dw_{\xi}^0}{dx_1} \right)\left|\xi_1+\frac{dw_{\xi}^0}{dx_1} \right|^{p-2} \right) = 0.
\label{eq:laminate2}
\end{equation}
There exists a constant $C(\xi)$ such that
 $\big(\xi_1 + \frac{dw_{\xi}^0}{dx_1}\big)^{p-1} = C(\xi)/a_0(x_1)$, where $z^{p-1}:=\text{sgn}(z)|z|^{p-1}$. If $C(\xi) = 0$, then $w_{\xi}^0(x_1) = - \xi_1 x_1$ which contradicts the periodicity of~$w_{\xi}^0$. In any cases, we have shown that $|\xi + \nabla w_{\xi}^{\per}| > 0$. We then prove easily that this implies~\textbf{(A4)}.

\paragraph{On Assumption \textbf{(A4)'.}} This Assumption is satisfied in dimension $d=1,2$ because \textbf{(A4)} is satisfied. For higher dimensions, we provide here a sufficient condition implying \textbf{(A4)'}:

\begin{lemme}[see \cite{cardone2013estimates} and~\cite{zhikov2008homogenization}]
Assume that $d \geq 2$ and that $|\xi + \nabla w_{\xi}^{\per}|^{2-p} \in L^{d/2}(Q)$, then \textbf{(A4)'} is satisfied. 
\label{lem:cardone}
\end{lemme} 

\begin{proof}
We refer to~\cite[Lemma 8]{cardone2013estimates}.
\end{proof} 

If we assume that $\{\xi + \nabla w_{\xi}^{\per } = 0 \}$ is a finite number of points (in the case $d>2$) and that all critical points have finite order, denoting by~$m$ the maximum order of the corresponding zero points, we have that $|\xi + \nabla w_{\xi}^{\per}|^{2-p} \in L^{d/2}(Q)$ if and only if $m \frac{d}{2}(p-2) < d$ i.e. $p < 2 + 2/m$. Thus, in this case, Assumption~\textbf{(A4)'} can be replaced by assuming that $p < 2+2/m$. Note also that if $\xi + \nabla w_{\xi}^{\per}$ vanishes at order $m$ along a line (or a curve) in dimension $d$, then $|\xi + \nabla w_{\xi}^{\per}|^{2-p} \sim |x|^{m(2-p)}$ which is $L^{d/2}(Q)$ if and only if $\frac{d}{2} m(p-2) < d-1$ \textit{i.e.} $p < 2 + \frac{2(d-1)}{dm}$.


\begin{remarque}
The Assumption~\textbf{(A4)'} is used in the proof of Lemma~\ref{lem:formfaible} which allows to pass from solutions in the distribution sense to solutions in the sense of Definition~\ref{def:def} for PDEs of the form~\eqref{eq:th2.2_gen2}. We then take advantage of Lemma~\ref{lem:formfaible} in the proof of Theorem~\ref{th:cont} by working locally in a concentration-compactness method. 
\label{re:A4_prime}
\end{remarque}


\subsection{Extension to other non-linear operators}
\label{subsect:extension}

We have limited the presentation of the results to the simplest operator~\eqref{eq:nonlin} in order to avoid some technicalities and the use of abstract existence Theorems for non-linear PDEs. However, the result of this paper extends to more general operators. We explain below the type of problems that we can address with the technique developed in this work.

\medskip

The first direct extension concerns the equivalent of~\eqref{eq:nonlin} when $a$ is a matrix-valued coefficient. This corresponds to the following non-linear operator: 
\begin{equation}
a(y,\xi) := \langle A(y)\xi,\xi \rangle^{\frac{p-2}{2}}A(y)\xi, \quad y \in \RR^d, \quad \xi \in \RR^d,
\label{eq:matrice}
\end{equation}
where $A$ is of the form $A = A^{\per} + \widetilde{A}$. We assume that the matrix $A^{\per}$ is periodic and that $A$ and $A^{\per}$ are symmetric and positive definite, that is,
$$\exists \lambda > 0,\quad \forall y \in \RR^d, \quad \lambda^{-1} |\xi|^2 \leq \langle A(y)\xi,\xi \rangle \leq \lambda |\xi|^2 \quad\text{and} \quad \lambda^{-1} |\xi|^2 \leq \langle A^{\per}(y)\xi,\xi \rangle \leq \lambda |\xi|^2.$$
The perturbation $\widetilde{A}$ satisfies $\widetilde{A} \in L^{p'}\cap \mathcal{C}^{0,1}(\RR^d)^{d\times d}$. The periodic correctors can be defined thanks to variational techniques by considering the minimization problem
$$\min_{w_{\xi}^{\per} \in H^{1,\per}(Q)} \left\{ \frac{1}{p} \int_Q \left\langle A(y)(\xi + \nabla w_{\xi}^{\per}),\xi + \nabla w_{\xi}^{\per} \right\rangle^{p/2} \right\}.$$
The non-periodic equation corresponding to~\eqref{eq:th2.2_gen2} is
\begin{equation}
- \di \left[ a(\cdot,\xi +\nabla w_{\xi}^{\per} + \nabla \widetilde{w_{\xi}}) - a(\cdot,\xi + \nabla w_{\xi}^{\per}) \right] = \di(h),
\end{equation}
where
\begin{equation}
h := a^{\per}(\cdot,\xi + \nabla w_{\xi}^{\per}) - a(\cdot,\xi + \nabla w_{\xi}^{\per}),
\label{eq:f_gen}
\end{equation}
where $a^{\per}(\cdot,\xi) := \left\langle A^{\per}(\cdot)\xi,\xi \right\rangle^{\frac{p-2}{2}}A^{\per}(\cdot)\xi$. It is easily proved that $f \in L^{p'}(\R^d)^{d}$ and that the method of proof of Section~\ref{sect:cor} extends to this case by studying the functional
$$\begin{aligned}
F_{\xi}(v) & := \frac{1}{p}\int_{\R^d} \Big\{\left\langle A(y)(\xi + \nabla w_{\xi}^{\per} + \nabla v),\xi + \nabla w_{\xi}^{\per} + \nabla v \right\rangle^{p/2} - \left\langle A(y)(\xi + \nabla w_{\xi}^{\per}),\xi + \nabla w_{\xi}^{\per} \right\rangle^{p/2} \\ & - p \langle A(y)(\xi + \nabla w_{\xi}^{\per}),\xi + \nabla w_{\xi}^{\per} \rangle^{\frac{p-2}{2}}A(y)(\xi + \nabla w_{\xi}^{\per})\cdot \nabla v \Big\}\dint y + \int_{\R^d} h\cdot \nabla v.
\end{aligned}$$ Note that the inequalities given in Appendix~\ref{sect:ineq} are valid for the matrix model~\eqref{eq:matrice}. Concerning the continuity results for the application $\xi \longmapsto \nabla\widetilde{w_{\xi}}$, the results proved in sections~\ref{sect:propcor}, \ref{sect:qual} and \ref{sect:weakcont} still hold true.

\medskip

The second less direct extension corresponds to non-variational operators, that is, PDEs that cannot be written as a minimization problem. We consider operators $a(y,\xi)$ that satisfy the following properties:
\begin{enumerate}[label=(\arabic*)]
\item \label{prop} for all $\xi \in \RR^d$, $a(\cdot,\xi)$ is a measurable function and $\xi \longmapsto a(y,\cdot)$ for fixed $y \in \R^d$ is of class $\mathcal{C}^1(\RR^d)$ and of class $\mathcal{C}^2(\RR^d \setminus \{0\})$.
\item \label{prop2} the application $\xi \longmapsto a(y,\xi)$ is homogeneous \textit{i.e.} $a(y,t\xi) = t^{p-1} a(y,\xi)$ for $t \in \R$ and $y,\xi \in \R^d$. We also assume that $a(\cdot,\xi)$ is a uniformly in $\xi$ Lipschitz continuous function: there exists $\lambda > 0$ such that
$$\forall y,y' \in \RR^d, \quad \forall \xi \in \RR^d, \quad \left|a(y,\xi) - a(y',\xi) \right| \leq \lambda |y-y'||\xi|^{p-1}.$$
$$\forall y,y' \in \RR^d, \quad \forall \xi \in \RR^d, \quad \left|\partial_{\xi} a(y,\xi) - \partial_{\xi} a(y',\xi) \right| \leq \lambda |y-y'||\xi|^{p-2}.$$
\item \label{prop3} we have that $a(y,\xi) = a^{\per}(y,\xi) + \widetilde{a}(y,\xi)$ where $a^{\per}(\cdot,\xi)$ is a periodic function satisfying the same homogeneity and regularity properties as $a$. We assume that the perturbation $\widetilde{a}$ satisfies:
$$\exists b \in L^{p'}\cap L^{\infty}(\R^d), \quad \forall \xi \in \RR^d,\quad \forall y \in \R^d,\quad  \big|\widetilde{a}(y,\xi)\big| \leq  b(y)|\xi|^{p-1} \quad \text{and} \quad \big|\partial_{\xi} \widetilde{a}(y,\xi)\big| \leq  b(y)|\xi|^{p-2}.$$
\item \label{prop4} There exists $\lambda > 0$ such that
$$\left\{a(y,\xi) - a(y,\xi')\right\}\cdot\left\{\xi - \xi'\right\} \geq \lambda^{-1}\left(|\xi|^{p-2} + |\xi'|^{p-2} \right)|\xi - \xi'|^2,$$
$$\left\{a^{\per}(y,\xi) - a^{\per}(y,\xi')\right\}\cdot\left\{\xi - \xi'\right\} \geq \lambda^{-1}\left(|\xi|^{p-2} + |\xi'|^{p-2} \right)|\xi - \xi'|^2,$$
and 
$$\left| a(y,\xi) - a(y,\xi') \right| \leq \lambda \left(|\xi|^{p-2} + |\xi'|^{p-2} \right)|\xi - \xi'|$$
$$\left| a^{\per}(y,\xi) - a^{\per}(y,\xi') \right| \leq \lambda \left(|\xi|^{p-2} + |\xi'|^{p-2} \right)|\xi - \xi'|.$$
We also assume that 
\begin{equation}
\sup_{y \in \RR^d}\sup_{|\xi|=1} \left|\partial^2_{\xi} a(y,\xi) \right| \leq \lambda.
\label{eq:sup_derivee_seconde}
\end{equation}
\end{enumerate}
We define the operator 
\begin{equation}
A : \begin{cases}
\begin{aligned}
W_{\xi +\nabla w_{\xi}^{\per}} & \longrightarrow \left( W_{\xi +\nabla w_{\xi}^{\per}} \right)' \\
\nabla v & \longmapsto \begin{cases} \begin{aligned}
W_{\xi + \nabla w_{\xi}^{\per}} & \longrightarrow \RR \\
\nabla h & \longmapsto \int_{\RR^d} \left[ a(\cdot,\xi +\nabla w_{\xi}^{\per} + \nabla v) - a(\cdot,\xi + \nabla w_{\xi}^{\per})\right]\cdot \nabla h.
 \end{aligned}
\end{cases}
\end{aligned}
\end{cases}
\label{eq:Abis}
\end{equation}
We can show that $A$ is hemicontinuous, bounded, coercive and strictly monotone. By~\cite[Corollary~8.1]{le2018nonlinear}, the PDE $A(\nabla v) = \mathcal{F}$, where $\mathcal{F} := \di\ \widetilde{a}(\cdot,\xi + \nabla w_{\xi}^{\per})$, admits a unique solution in $W_{\xi+\nabla w_{\xi}^{\per}}$. The results of Section~\ref{sect:weakcont}, which are sufficient to prove the qualititative homogenization of Section~\ref{sect:qual} (which is in fact the main result of this paper), only use the PDE and are thus directly generalized. The results of Section~\ref{sect:propcor} can be proved using the PDE instead of the minimization problem~\eqref{eq:min}. These extensions are detailed in~\cite[Chapter 5]{wolfphd}.

\begin{remarque}
A simple example of a non-variational operator satisfting the above assumptions is $a(y,\xi) = A(y)\xi \left|\xi \right|^{p-2}$, where $A$ is a positive definite and bounded symmetric matrix that can be written under the form $A = A^{\per} + \widetilde{A}$ where $\widetilde{A} \in L^{p'}\cap \mathcal{C}^{0,1}(\RR^d)^{d \times d}$. We check that $a$ is not variational: assume by contradiction that there exists a function $F : \RR^d \times \RR^d \rightarrow \RR$ such that $a(y,\xi) = \partial_{\xi} F(y,\xi)$. In particular, thanks to Schwartz Theorem, we should have that for all $i,j \in \{1,...,d\}$,
$$\partial_{\xi_j} \left[a(y,\xi)_i\right]=\partial_{\xi_i} \left[a(y,\xi)_j \right].$$
Expanding each term gives, for $\xi \neq 0$,
$$A(i,j)|\xi|^{p-2} + (p-2)\left[A(y)\xi \right]_i \xi_j |\xi|^{p-4} =  A(j,i)|\xi|^{p-2} + (p-2)\left[A(y)\xi \right]_j \xi_i |\xi|^{p-4}$$
In particular, for all $\xi \neq 0$ and $(i,j) \in \{ 1,...,d \}^2$,
$$ \left[A(y)\xi \right]_i \xi_j  = \left[A(y)\xi \right]_j \xi_i.$$
This shows that $A$ is a scalar matrix \textit{i.e.} proportional to the identity.  
\end{remarque}

\begin{remarque}
Assumption~\eqref{eq:sup_derivee_seconde} is only needed in the proofs of Theorem~\ref{th:th_nonlin} and Theorem~\ref{th:coercivitypercor}. Note also that, together with homogeneity, this Assumption implies that for all $\delta > 0$,
$$\sup_{y \in \RR^d} \sup_{|\xi|=\delta} \left| \partial^2_{\xi} a(y,\xi) \right| \leq \lambda \delta^{p-3}.$$
\end{remarque}

%










\section{The one-dimensional setting}
\label{sect:1D}


We consider the homogenization of~\eqref{eq:nonlin} in the one-dimensional case. This equation reads as:
\begin{equation}
\begin{cases}
- \left(a(./\varepsilon) u_{\varepsilon}' |u_{\varepsilon}'|^{p-2} \right)' = f \\
u_{\varepsilon}(-\frac{1}{2}) = u_{\varepsilon}(\frac{1}{2}) = 0,
\end{cases}
\end{equation}
where $a$ is of the form $a = a^{\per} + \widetilde{a}$ with $\widetilde{a} \in L^q \cap\mathcal{C}^{0,\alpha} (\R)$, $1 < q <+\infty$ and $a$ satisfies Assumption~\textbf{(A1)}. In this section, we assume that $f \in L^{p'}(-\frac{1}{2},\frac{1}{2})$.
Direct computations show that 
\begin{equation}
u_{\varepsilon}' = \left(\frac{-F + C_{\varepsilon}}{a(./\varepsilon)} \right)^{1/(p-1)}, \quad F(x) = \int_{-\frac{1}{2}}^x f,
\label{eq:1D_1}
\end{equation}
where $x^{\frac{1}{p-1}} := \text{sgn}(x) |x|^{\frac{1}{p-1}}$ for $x \in \R$. The constant $C_{\varepsilon}$ is such that 
\begin{equation}
\int_{-\frac{1}{2}}^{\frac{1}{2}}  \left(\frac{-F + C_{\varepsilon}}{a(./\varepsilon)} \right)^{1/(p-1)} = 0.
\label{eq:1D_2}
\end{equation}
We note that the function $F$ is bounded and thus the sequence $(C_{\varepsilon})_{\varepsilon> 0}$ is bounded. Passing to the limit $\varepsilon \longrightarrow 0$ in \eqref{eq:1D_1} and \eqref{eq:1D_2}, we get that $u_{\varepsilon} \underset{\varepsilon \rightarrow 0}{\relbar\joinrel\rightharpoonup} u^*$ in $W^{1,p}(-\frac{1}{2},\frac{1}{2})$ and $C_{\varepsilon} \underset{\varepsilon\rightarrow 0}{\longrightarrow} C^*$, where 
$$(u^*)' = \left( \frac{-F + C^*}{a^*} \right)^{1/(p-1)}, \quad \int_{-\frac{1}{2}}^{\frac{1}{2}} \left(-F + C^*\right)^{1/(p-1)} = 0.$$
and the homogenized coefficient is defined by
$$a^* := \left(L^{p}-\underset{\varepsilon \rightarrow 0}{\text{weaklim}} \  a\big(\frac{\cdot}{\varepsilon}\big)^{-\frac{1}{p-1}}\right)^{-(p-1)}.$$
We easily show with the ingredients used in Remark~\ref{re:sublinear} below that
$$a^* = \left(\int_{-\frac{1}{2}}^{\frac{1}{2}} \frac{1}{\big(a^{\per}\big)^{\frac{1}{p-1}}}\right)^{-(p-1)}.$$
The homogenized equation solved by $u^*$ is
$$
\begin{cases}
-\left( a^* (u^*)'|(u^*)'|^{p-2} \right)'  = f \\
u^*\big(-\frac{1}{2} \big) = u^*\big(\frac{1}{2}\big)  = 0.
\end{cases}
$$
The corrector equations~\eqref{eq:cornonlin} and~\eqref{eq:th2.2_gen2}-\eqref{eq:f_2} in the direction $\xi \in \R$
are easy to solve (see Remark~\ref{re:sublinear} below):
\begin{equation}
\xi + w'_{\xi} = \xi\left( \frac{a^*}{a} \right)^{\frac{1}{p-1}} \quad \text{and} \quad \xi + (w_{\xi}^{\per})' = \xi\left( \frac{a^*}{a^{\per}} \right)^{\frac{1}{p-1}}.
\label{eq:1Dcor}
\end{equation}
Let
$
R_{\varepsilon} := u_{\varepsilon} - (u^*)' - w_{(u^*)'}(./\varepsilon).
$ be the remainder between $u_{\varepsilon}$ and its two scale expansion.
When $u^*$ is regular enough, we have that
\begin{equation}
\begin{aligned}
R'_{\varepsilon} & = (u_{\varepsilon}^1)' - (u^*)'(1 + w'(./\varepsilon)) - \varepsilon w(./\varepsilon)(u^*)'' \\ & = \underbrace{\frac{(-F + C_{\varepsilon})^{1/(p-1)} - (-F + C^*)^{1/(p-1)}}{a(./\varepsilon)^{1/(p-1)}}}_{=: (u_{\varepsilon}^1)'} - \varepsilon w(./\varepsilon)(u^*)'',
\end{aligned}
\label{eq:two-scale1D}
\end{equation}
where $w := w_1$.
We concentrate in the sequel on the first term of~\eqref{eq:two-scale1D}, the second one being related to the regularity of $u^*$ on the one hand (which is not related to homogenization) and to the sublinearity of $w$ on the other hand. We prove briefly that $w$ is sublinear: indeed, we can write $w'= (w^{\per})' + \widetilde{w}'$ where, due to~Remark~\ref{re:optimal} below, $\widetilde{w}' \in L^q(\R^d)$. By H\"{o}lder (or Morrey) inequality, we get immediately that $\widetilde{w}$ is sublinear. Since $w^{\per}$ is periodic and bounded, it is in particular also sublinear. This proves that $w$ is sublinear.
We use Lemma~\ref{lem:1D_estim} stated in Appendix~\ref{sect:ineq} to obtain the bound
\begin{equation}
\left| u_{\varepsilon}' - (u^*)'(1 + w'(./\varepsilon)) \right| \leq \lambda |C_{\varepsilon} - C^*|^{1/(p-1)} \underset{\varepsilon \rightarrow 0}{\longrightarrow} 0 \quad \text{uniformly}.
\end{equation}
We have obtained the $L^{\infty}-$strong convergence of $(u_{\varepsilon}^1)'$ to zero when we use the non-periodic corrector. Let us now introduce the "periodic" remainder $R_{\varepsilon}^{\per}$ which is defined by $
R_{\varepsilon}^{\per} := u_{\varepsilon} - u^* - \varepsilon w_{(u^*)'}^{\per}(./\varepsilon).
$
We have that
\begin{equation}
\begin{aligned}
(R^{\per}_{\varepsilon})'& = \frac{(-F + C_{\varepsilon})^{1/(p-1)}}{a(./\varepsilon)^{1/(p-1)}} -  \frac{(-F + C^*)^{1/(p-1)}}{a^{\per}(./\varepsilon)^{1/(p-1)}} + \varepsilon w^{\per}_1(./\varepsilon) (u^*)'' \\
 &= (u_{\varepsilon}^1)' + (-F + C^*)^{1/(p-1)} \left[ \frac{1}{a(./\varepsilon)^{1/(p-1)}} - \frac{1}{a^{\per}(./\varepsilon)^{1/(p-1)}} \right] + \varepsilon w^{\per}_1(./\varepsilon) (u^*)''.
\end{aligned}
\label{eq:microscale}
\end{equation}
The first term tends uniformly to zero while the second one does not tend to zero in $L^{\infty}$ unless $\widetilde{a} = 0$ or $C^* = 0$. Indeed, testing~\eqref{eq:microscale} at the microscale gives:
$$
\left|(-F + C^*)^{1/(p-1)} \left[ \frac{1}{a(./\varepsilon)^{1/(p-1)}} - \frac{1}{a^{\per}(./\varepsilon)^{1/(p-1)}} \right](\varepsilon x)\right| \geq \underbrace{c(p,\lambda)|-F(\varepsilon x) + C^*|^{1/(p-1)} |\widetilde{a}(x)|}_{\underset{\varepsilon \rightarrow 0}{\longrightarrow} c(p,\lambda)|C^*|^{1/(p-1)}|\widetilde{a}(x)| \neq 0}.
$$
 This shows that the convergence of the remainder deteriorates when using $w_{\xi}^{\per}$ instead of $w_{\xi}$.
We close this section by commenting on the integrability of the correctors in the particular 1D setting. We show in Remark~\ref{re:optimal} that, in this case, the exponent given by Theorem~\ref{th:coercivitypercor} is optimal for $q=p'$, see also~Remark~\ref{rem:th2.5}.
\begin{remarque} Suppose that $\widetilde{a} \in L^q(\R^d) \cap \holder(\R^d)$, $1 < q < +\infty$.
An explicit calculation shows that 
\begin{equation}
\widetilde{w_{\xi}}' = - \left(\xi + (w_{\xi}^{\per})' \right)  + \left(\xi + (w_{\xi}^{\per})' \right)\left( 1 - \frac{\widetilde{a}}{a} \right)^{\frac{1}{p-1}},
\end{equation}
and
$|\xi +  (w_{\xi}^{\per})'| \geq c|\xi|$.
Since $\widetilde{a}(x) \underset{|x| \longrightarrow +\infty}{\longrightarrow} 0$, we have that
$$\widetilde{w_{\xi}}' \underset{x \rightarrow \pm \infty}{\sim} -\frac{1}{p-1} \frac{\widetilde{a}(\xi + (w_{\xi}^{\per})')}{a}.$$
Thus $\widetilde{w_{\xi}}' \in L^q(\R^d)$, that is $\widetilde{w_{\xi}}'$ has the same integrability as $\widetilde{a}$ and this exponent is optimal.
\label{re:optimal}
\end{remarque}

\begin{remarque}
We show below that there exists a unique solution $w_{\xi}$ to~\eqref{eq:cornonlin} that is sublinear at infinity. This justifies, in dimension one, to search $w_{\xi}$ under the form $w_{\xi}^{\per} + \widetilde{w_{\xi}}$ where $\widetilde{w_{\xi}}' \in L^p(\RR)$. 

\medskip

Assume that $w_{\xi}$ is a sublinear solution to~\eqref{eq:1Dcor}. Then, there exists a constant $C$ such that $\xi + w_{\xi}' = \left(C/a\right)^{1/(p-1)}$. We have by sublinearity that
$$\xi = \lim_{x \rightarrow +\infty} \fint_{0}^x \left(\xi + w_{\xi}'\right) =  \lim_{x \rightarrow +\infty} \fint_0^x \left(\frac{C}{a}\right)^{\frac{1}{p-1}} = C^{\frac{1}{p-1}}  \lim_{x \rightarrow +\infty} \fint_0^x \left(\frac{1}{a}\right)^{\frac{1}{p-1}}.$$
However, by Lemma~\ref{lem:1D_estim}, we have that 
$$\left|\fint_0^x \left(\frac{1}{a}\right)^{\frac{1}{p-1}} - \fint_0^x \left(\frac{1}{a^{\per}}\right)^{\frac{1}{p-1}} \right| \leq \fint_0^x \frac{\left| a^{\frac{1}{p-1}} - \left(a^{\per}\right)^{\frac{1}{p-1}}\right|}{a^{\frac{1}{p-1}}\left(a^{\per}\right)^{\frac{1}{p-1}}} \leq \mathrm{Cst.} \ \fint_0^x \left|\widetilde{a}\right|^{\frac{1}{p-1}},$$
where $\mathrm{Cst.}$ denotes a constant depending only on $p$ and $\lambda$. Since $\widetilde{a} \in L^{p'}(\RR^d)$, we get by H\"{o}lder inequality that $$\fint_0^x \left|\widetilde{a}\right|^{\frac{1}{p-1}} \underset{x\rightarrow +\infty}{\longrightarrow} 0.$$
This shows that 
$$\lim_{x \rightarrow +\infty} \fint_0^x \left(\frac{1}{a}\right)^{\frac{1}{p-1}} = \lim_{x \rightarrow +\infty} \fint_0^x \left(\frac{1}{a^{\per}}\right)^{\frac{1}{p-1}} = \left(\frac{1}{a^*} \right)^{\frac{1}{p-1}}$$
and gives that $C = \xi |\xi|^{p-2} a^*$. This shows that $\widetilde{w_{\xi}}$ is  necessarily of the form~\eqref{eq:1Dcor}.
\label{re:sublinear}
\end{remarque}

\paragraph*{Numerical experiments.} We have implemented for $p=3$ the solution to~\eqref{eq:nonlin} in the 1D setting for $f(x) = 2x$ and $$a(y) := a^{\per}(y) + \widetilde{a}(y) = 2+\cos(2\pi y) + 10e^{-|y|}$$ on the domain $\Omega:=(-\frac{1}{2},\frac{1}{2})$. The boundary conditions are homogeneous Dirichlet conditions \textit{i.e.} $u_{\varepsilon}(-\frac{1}{2}) = u_{\varepsilon}(\frac{1}{2}) = 0$. The coefficient $a$ satisfies of course Assumptions \textbf{(A1)-(A3)}. The results are plotted on Figure~\ref{fig}. We comment on these results. We have plotted for different values of $\varepsilon$ the function $u_{\varepsilon}'$ (which is labeled as '\texttt{exact solution}'), the periodic two scale approximation $(u^*)' + (w^{\per})'(./\varepsilon)(u^*)'$ (which is labeled as '\texttt{periodic two-scale approx.}') and the non-periodic two scale approximation $(u^*)' + w'(./\varepsilon)(u^*)'$ (which is labeled as '\texttt{non-periodic two-scale approx.}'). Tables~\ref{fig:tableaunum_1} and~\ref{fig:tableaunum_2} give numerical values for the periodic and non-periodic remainders in $L^2$ and $L^{\infty}-$norm for different values of $\varepsilon$. We see that on Figure~\ref{fig}, qualitatively, the non-periodic two-scale approximation fits efficiently the exact solution for each chosen value of $\varepsilon$. The periodic two-scale approximation corresponds to the exact solution far from the defect, which, as $\varepsilon \longrightarrow 0$, concentrates aroung the origin. We notice that the non-periodic corrector is useful to reconstruct the oscillations of the exact solution locally around the defect. Tables~\ref{fig:tableaunum_1} and~\ref{fig:tableaunum_2} express the same idea: the $L^{\infty}-$norms of the periodic remainders remain unchanged as~$\varepsilon$ decreases whereas those of the non-periodic remainder decrase with~$\varepsilon$. For the $L^2-$norm, which is weaker than the $L^{\infty}-$norm, both norms decrease as $\varepsilon$ gets closer to zero although the nonperiodic approximation is more accurate than the periodic approximation. This means that, depending on the precision we want (and also on the regularity on $f$ and $a$), we may use the periodic corrector, which is much easier to compute, or the non-periodic corrector, if we seek for a fine approximation of the exact solution. This can also be seen theoretically since $R_{\varepsilon}^{\per} - R_{\varepsilon} = \varepsilon \widetilde{w'}(./\varepsilon) (u^*)'$ and, for all $q \leq p$,
$$\big\|\varepsilon \widetilde{w}'(./\varepsilon)(u^*)' \big\|_{L^q(0,1)} \leq C\varepsilon^{d/p} \|(u^*)'\|_{L^{\infty}(0,1)} \|\widetilde{w}' \|_{L^p(\R)}.$$
In any case, we get that $R_{\varepsilon}^{\per} - R_{\varepsilon} \underset{\varepsilon\rightarrow 0}{\longrightarrow} 0$ in $L^q-$norm, $q \leq p$ but not in $L^{\infty}-$norm. Another way to reformulate the preceding remark is the following: the non-periodic corrector provides a better approximation at the microscale.
\begin{figure}[h!]
\centering
\begin{subfigure}{.5\textwidth}
\includegraphics[scale=0.5]{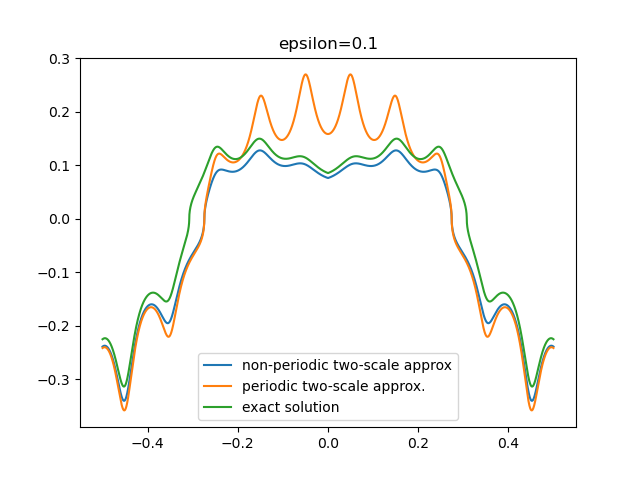}
\end{subfigure}%
\begin{subfigure}{.5\textwidth}
\includegraphics[scale=0.5]{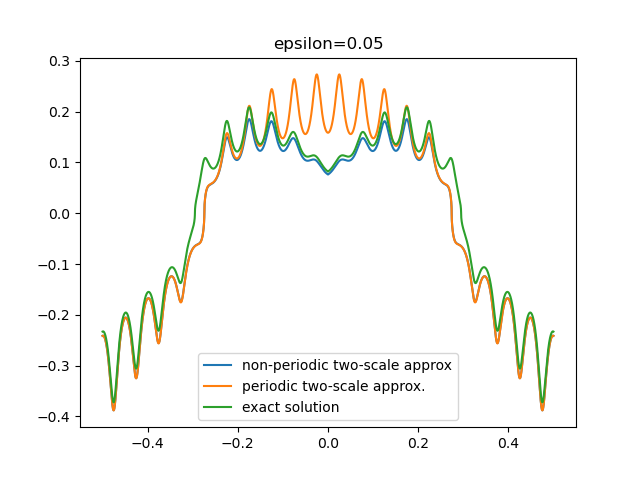}
\end{subfigure} 
\begin{subfigure}{.5\textwidth}
\includegraphics[scale=0.5]{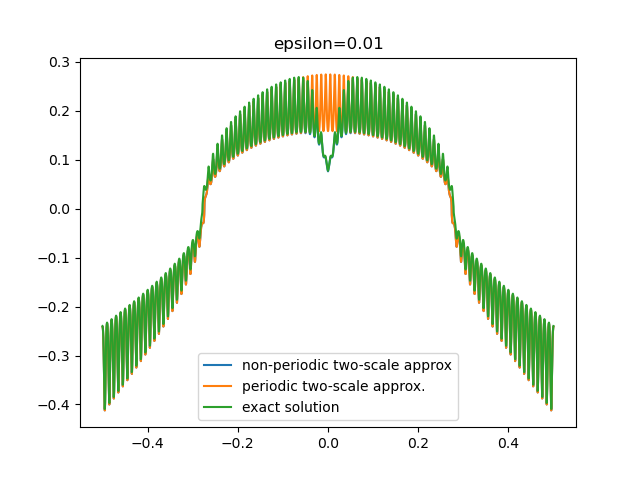}
\end{subfigure}%
\begin{subfigure}{.5\textwidth}
\includegraphics[scale=0.5]{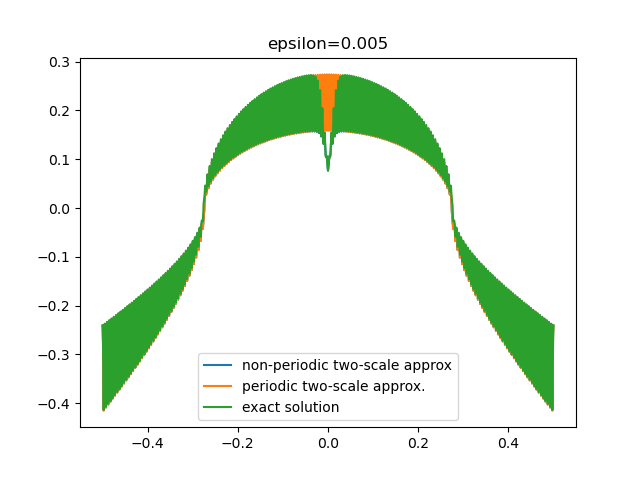}
\end{subfigure}
\includegraphics[scale=0.3]{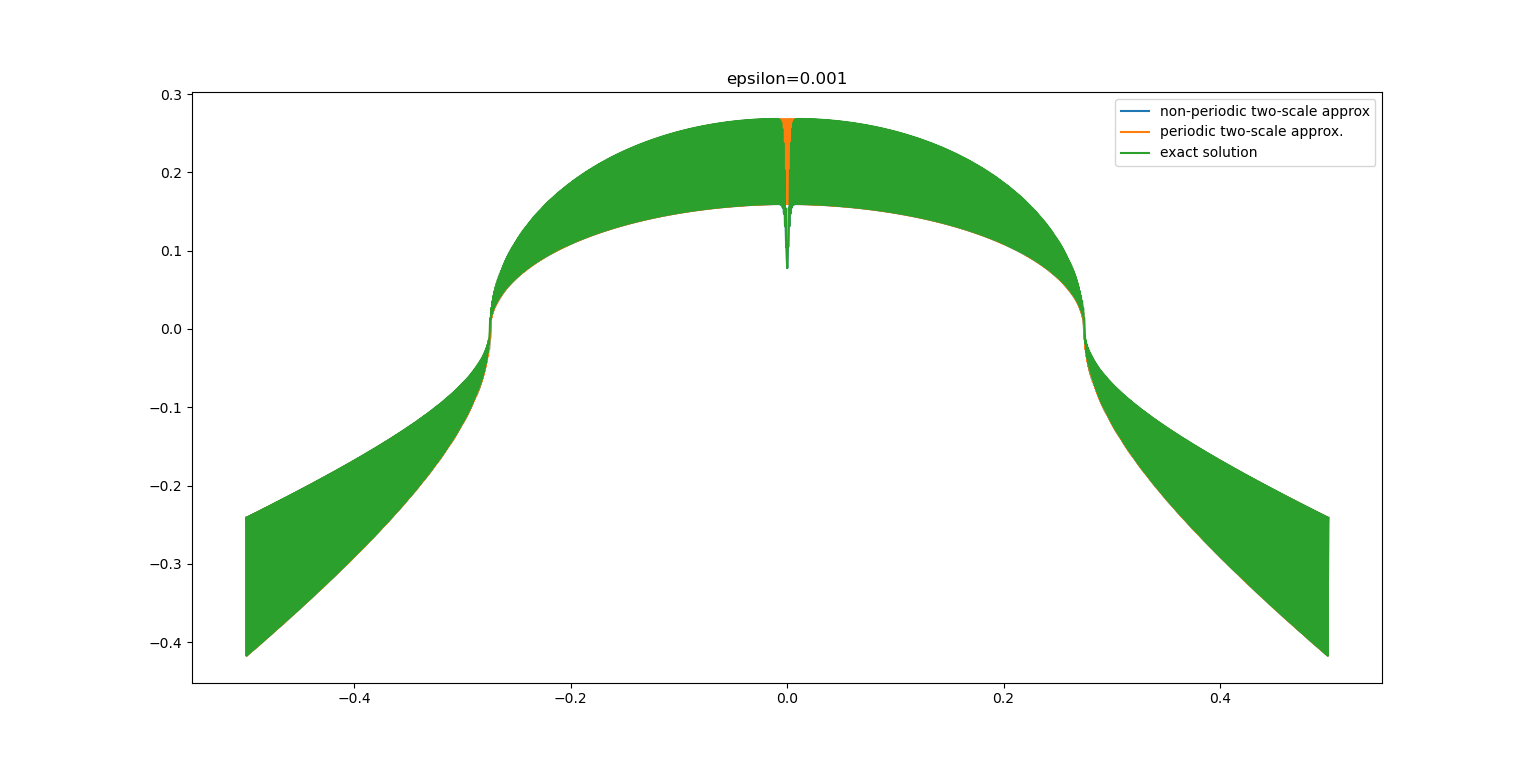}
\caption{Numerical simulation in the particular 1D case.}
\label{fig}
\end{figure}
\begin{table}[h!]
\centering
\begin{tabular}{|p{60pt}|p{60pt}|p{60pt}|}
\hline $\varepsilon$ & $\|R_{\varepsilon}^{\per}\|_{L^{\infty}}$ & $\|R_{\varepsilon}\|_{L^{\infty}}$\\
\hline 
0.1 & 0.156 & 0.109 \\
0.05 & 0.163 & 0.137 \\
0.01 & 0.170  & 0.0657 \\
0.005 & 0.170 & 0.0288 \\
0.001 & 0.170 & 0.0245 \\
0.0005 & 0.171 & 0.0136 \\
\hline
\end{tabular}
\caption{Numerical errors for different values of $\varepsilon$ in $L^{\infty}-$norm.}
\label{fig:tableaunum_1}
\end{table}
\begin{table}[h!]
\centering
\begin{tabular}{|p{60pt}|p{60pt}|p{60pt}|}
\hline $\varepsilon$ & $\|R_{\varepsilon}^{\per}\|_{L^2}$ & $\|R_{\varepsilon}\|_{L^2}$\\
\hline 
0.1 & 6.39 & 3.85 \\
0.05 & 5.01 & 3.16 \\
0.01 & 2.13  & 0.740 \\
0.005 & 1.47 & 0.331 \\
0.001& 0.654 & 0.108 \\
0.0005 & 0.46 & 0.0461 \\
\hline
\end{tabular}
\caption{Numerical errors for different values of $\varepsilon$ in $L^{2}-$norm.}
\label{fig:tableaunum_2}
\end{table}

\section[Existence of the correctors]{Existence of the non-periodic correctors: proof of Theorem~\ref{th:existencecor}}
\label{sect:cor}

We start this section with some preliminary results: 

\begin{lemme} Let $\xi \in \RR^d$ and $W_{\xi + \nabla w_{\xi}^{\per}}$ be defined by~\eqref{eq:Wu}.
\begin{enumerate}[label=(\roman*)]
\item The space $W_{\xi + \nabla w_{\xi}^{\per}}$ is a Banach space. 
\item Its topological dual space is
$$\left\{- \text{div}(g),\quad  g= g_1+g_2 |\xi + \nabla w_{\xi}^{\per}|^{p-2}, \quad g_1 \in L^{p'}(\R^d), \quad g_2 |\xi + \nabla w_{\xi}^{\per}|^{\frac{p-2}{2}} \in L^2(\R^d) \right\}.$$
\item Each bounded sequence in $W_{\xi + \nabla w_{\xi}^{\per}}$ admits a weakly converging subsequence.
\end{enumerate}
\label{lem:banach}
\end{lemme}

\begin{proof}We refer to~\cite[Chapter 5]{wolfphd} for the proof of this elementary Lemma. \end{proof}

 We now fix $\xi \in \R^d$, $h \in L^{p'}(\R^d)^d$, a coefficient $a$ satisfying Assumptions~\textbf{(A1)-(A2)-(A3)}. We introduce the functional $F_{\xi}$ defined by

\begin{equation}
F_{\xi}( v) := \frac{1}{p} \int_{\R^d} a g_{\xi + \nabla w_{\xi}^{\per}}(\nabla v) + \int_{\R^d} h \cdot \nabla v,
\label{eq:F_u}
\end{equation}
where the function $g_{\xi}$ is defined ny~\eqref{eq:g_(xi)}:
$$g_{\xi}(x) := |\xi + x|^p - |\xi|^p - p \xi|\xi|^{p-2}  \cdot x.$$ Since $g_{\xi}(x) \geq 0$ over $\R^d$, we immediately have that $F_{\xi}$ is defined over 
\begin{equation}
V := \left\{v \in W^{1,1}_{\text{loc}}(\R^d), \quad \nabla v \in L^p(\R^d)\right\}/\RR
\label{eq:V}
\end{equation}
 and takes its values in $\R \cup \{+ \infty\}$. Note that since $F_{\xi}(v)$ only depends on $\nabla v$, $F_{\xi}$ is well-defined on the space of equivalence classes $V$.
For $R > 0$, we define the mapping
\begin{equation}
F_{\xi}^R :
\begin{cases} \begin{aligned}
 V & \longrightarrow \R  \\
v & \longmapsto \frac{1}{p}\int_{B_R} a g_{\xi + \nabla w_{\xi}^{\per}}(\nabla v) + \int_{\R^d} h \cdot \nabla v.
\end{aligned} \end{cases}
\label{eq:Fr}
\end{equation} 
 We gather in Lemmas~\ref{lem:elementary_cor} and \ref{lem:diffF_u} below the key properties satisfied by the functional $F_{\xi}$.

\begin{lemme}
Let $\xi \in \R^d$, $h \in L^{p'}(\R^d)^d$, $F_{\xi}$ be defined by~\eqref{eq:F_u} over $V$ and the space $W_{\xi + \nabla w_{\xi}^{\per}}$ be defined by~\eqref{eq:Wu}.
\begin{enumerate}[label=(\roman*)]
\item There exist two constants $c, C > 0$ such that for all $v \in W_{\xi + \nabla w_{\xi}^{\per}}$, 
\begin{equation}
c \left[ - 1 + \|v \|^{2}_{W_{\xi + \nabla w_{\xi}^{\per}}} \right] \leq F_{\xi}(v) \leq C \left[1 + \| v \|^{p}_{W_{\xi + \nabla w_{\xi}^{\per}}} \right].
\label{eq:coercivit}
\end{equation}
In particular, $F_{\xi}(v)$ is finite if and only if $v \in W_{\xi + \nabla w_{\xi}^{\per}}$.
\item The function $F_{\xi}$ is convex over $V$ and strictly convex over $W_{\xi + \nabla w_{\xi}^{\per}}$.
\end{enumerate}
\label{lem:elementary_cor}
\end{lemme}

\begin{proof}
[Proof of Lemma~\ref{lem:elementary_cor}] The point (i) is a simple application of~\eqref{eq:useful_ineq_4}. Let $v \in W_{\xi + \nabla w_{\xi}^{\per}}$, we have thanks to~\eqref{eq:useful_ineq_4} together with H\"{o}lder inequality that 
\begin{equation}
\begin{aligned}
- \|h\|_{L^{p'}(\R^d)} \|\nabla v\|_{L^p(\R^d)} & + c \lambda^{-1} \int_{\R^d} | \nabla v|^p + |\xi + \nabla w_{\xi}^{\per}|^{p-2} |\nabla v|^2 \leq F_{\xi}( v) \\
& \leq \|h\|_{L^{p'}(\R^d)} \|\nabla v\|_{L^p(\R^d)} + C \lambda \int_{\R^d} | \nabla v|^p + |\xi + \nabla w_{\xi}^{\per}|^{p-2} |\nabla v|^2.
\end{aligned}
\end{equation}
Easy computations allow to deduce that
$$F_{\xi}( v) \leq \frac{1}{p'}\|h\|_{L^{p'}(\R^d)}^{p'} + (2 C \lambda + 1)\big\| v\big\|_{W_{\xi + \nabla w_{\xi}^{\per}}}^p.$$
This proves the right-most inequality of \eqref{eq:coercivit} after changing the constant $C$. For the left-most inequality, we write that, by Young inequality 
$$- \frac{\lambda}{2c} \|h\|_{L^{p'}(\R^d)}^{p'} + \frac{c \lambda^{-1}}{2}\int_{\R^d} |\nabla v|^p + c \lambda^{-1} \int_{\R^d} |\xi + \nabla w_{\xi}^{\per}|^{p-2} |\nabla v|^2 \leq F_{\xi}(v).$$
We deduce the lower bound
$$\frac{1}{2} \left\| v\right\|_{W_{\xi + \nabla w_{\xi}^{\per}}}^2 - 1 \leq \int_{\R^d} |\nabla v|^p + \int_{\R^d} |\xi + \nabla w_{\xi}^{\per}|^{p-2} |\nabla v|^2.$$
Thus,
$$- \left( \frac{\lambda}{2c}\|h\|_{L^{p'}(\R^d)}^{p'} + \frac{c \lambda^{-1}}{2} \right) + \frac{c \lambda^{-1}}{4} \| v\|_{W_{\xi + \nabla w_{\xi}^{\per}}}^2 \leq F_{\xi}( v).$$
After changing the constant $c$, we get~\eqref{eq:coercivit}. This proves (i).

\medskip

The point (ii) follows readily from the strict convexity of the application $z \mapsto |z|^p$.
\end{proof}

\begin{lemme} Let $\xi \in \R^d$, $h \in L^{p'}(\R^d)^d$, $F_{\xi}$ be defined by~\eqref{eq:F_u} over $V$ and the space $W_{\xi + \nabla w_{\xi}^{\per}}$ be defined by~\eqref{eq:Wu}. Then the application $F_{\xi}$ is Fr\'echet-differentiable over~$W_{\xi + \nabla w_{\xi}^{\per}}$. Its differential is given, for $v \in W_{\xi + \nabla w_{\xi}^{\per}}$, by
\begin{equation}
F'_{\xi}( v)\cdot u := \int_{\RR^d} \left\{a\left[(\xi + \nabla w_{\xi}^{\per} + \nabla v)\big| \xi + \nabla w_{\xi}^{\per} + \nabla v \big|^{p-2} - (\xi + \nabla w_{\xi}^{\per})\big| \xi + \nabla w_{\xi}^{\per}\big|^{p-2} \right]  + h\right\}\cdot \nabla u.
\label{eq:Euler_Lag}
\end{equation}
\label{lem:diffF_u}
\end{lemme}

\begin{proof}
[Proof of Lemma~\ref{lem:diffF_u}] We fix $v \in W_{\xi + \nabla w_{\xi}^{\per}}$ and $u\in W_{\xi + \nabla w_{\xi}^{\per}}$. We have that
\begin{equation}
F_{\xi}( v + u) - F_{\xi}( v) = \frac{1}{p} \int_{\R^d} a \left[ g_{\xi + \nabla w_{\xi}^{\per}}(\nabla v + \nabla u) - g_{\xi + \nabla w_{\xi}^{\per}}(\nabla v) \right] + \int_{\R^d} h \cdot \nabla u.
\label{eq:lemme_4.2}
\end{equation}
We note that
\begin{equation}
\begin{aligned}
& g_{\xi + \nabla w_{\xi}^{\per}}(\nabla v + \nabla u)  - g_{\xi + \nabla w_{\xi}^{\per}}(\nabla v) \\ & = \big|\xi + \nabla w_{\xi}^{\per} + \nabla v + \nabla u\big|^p - \big|\xi + \nabla w_{\xi}^{\per} + \nabla v \big|^p - p(\xi + \nabla w_{\xi}^{\per} + \nabla v)|\xi + \nabla w_{\xi}^{\per} + \nabla v|^{p-2} \cdot \nabla u \\
& \quad + p \left( (\xi + \nabla w_{\xi}^{\per} + \nabla v)|\xi + \nabla w_{\xi}^{\per} + \nabla v|^{p-2} - (\xi + \nabla w_{\xi}^{\per})|\xi + \nabla w_{\xi}^{\per}|^{p-2} \right) \cdot \nabla u \\
& = A + B,
\end{aligned}
\end{equation}
where
\begin{equation}
A :=\big|\xi + \nabla w_{\xi}^{\per} + \nabla v + \nabla u\big|^p - \big|\xi + \nabla w_{\xi}^{\per} + \nabla v \big|^p - p(\xi + \nabla w_{\xi}^{\per} + \nabla v)|\xi + \nabla w_{\xi}^{\per} + \nabla v|^{p-2} \cdot \nabla u
\label{eq:A_0}
\end{equation}
and
\begin{equation}
B := p \left( (\xi + \nabla w_{\xi}^{\per} + \nabla v)|\xi + \nabla w_{\xi}^{\per} + \nabla v|^{p-2} - (\xi + \nabla w_{\xi}^{\per})|\xi + \nabla w_{\xi}^{\per}|^{p-2} \right) \cdot \nabla u.
\label{eq:B}
\end{equation}
We note that, using the definition of $g_{\xi}$~\eqref{eq:g_(xi)}, 
$$
A =  g_{\xi + \nabla w_{\xi}^{\per} + \nabla v}(\nabla u) \leq C \left\{ \big| \nabla u \big|^p +  \big|\xi + \nabla w_{\xi}^{\per} + \nabla v|^{p-2} \big| \nabla u \big|^2 \right\},
$$
where we have used the right-most part of inequality~\eqref{eq:useful_ineq_4}. Thus, applying the inequality $(b_1 + b_2)^{p-2} \leq C(p)(b_1^{p-2} + b_2^{p-2})$ for $b_1,b_2 \geq 0$, we get that
\begin{equation}
|A| \leq C \left\{ |\nabla u|^p + |\xi + \nabla w_{\xi}^{\per}|^{p-2}|\nabla u|^2 + |\nabla v|^{p-2}|\nabla u|^2 \right\}.
\label{eq:estum_A}
\end{equation}
We now note that, due to H\"{o}lder inequality and the fact that
$$\left(\frac{p}{2} \right)' \left(\frac{p}{2} - 1 \right) = \frac{p}{2} \quad \Longrightarrow\quad \left(\frac{p}{2} \right)' \big(p-2 \big) = p,$$
we obtain
\begin{equation}
\int_{\R^d} |\nabla v|^{p-2}|\nabla u|^2 \leq \left( \int_{\R^d} |\nabla v|^p  \right)^{1 - 2/p} \left( \int_{\R^d} |\nabla u|^p \right)^{2/p}.
\label{eq:A}
\end{equation}
Gathering \eqref{eq:estum_A}, \eqref{eq:A} and recalling the definition~\eqref{eq:normeWu}, we have proved that $A \in L^1(\R^d)$ and that 
\begin{equation}
\int_{\R^d} |A| \leq C \left[ \|u \|_{W_{\xi + \nabla w_{\xi}^{\per}}}^2 + \|u \|_{W_{\xi + \nabla w_{\xi}^{\per}}}^p \right],
\label{eq:A_1}
\end{equation}
where the constant $C$ does not depend on $u$.
We now turn to estimating $B$, see~\eqref{eq:B}. Using~\eqref{eq:useful_ineq_3}, Cauchy-Schwarz inequality and Young inequality, we have that 
\begin{equation}
\begin{aligned}
|B| & \leq C \left[ |\xi + \nabla w_{\xi}^{\per} + \nabla v|^{p-2} + |\xi + \nabla w_{\xi}^{\per}|^{p-2} \right] |\nabla v| |\nabla u| \\
& \leq C \left[ |\xi + \nabla w_{\xi}^{\per}|^{p-2} |\nabla v|^2 + |\xi + \nabla w_{\xi}^{\per}|^{p-2} |\nabla u|^2 + |\nabla v|^{p} + |\nabla u|^p \right].
\end{aligned}
\end{equation}
This proves that $B \in L^1(\R^d)$ and that 
\begin{equation}
\int_{\R^d} |B| \leq C \left[ \| v \|_{W_{\xi + \nabla w_{\xi}^{\per}}}^2 + \| u \|_{W_{\xi + \nabla w_{\xi}^{\per}}}^2 +  \|v \|_{W_{\xi + \nabla w_{\xi}^{\per}}}^p  +\| u \|_{W_{\xi + \nabla w_{\xi}^{\per}}}^p  \right],
\label{eq:B_int}
\end{equation}
where the constant $C$ is independent of $v$ and $u$. We can now conclude the proof of Lemma~\ref{lem:diffF_u}: using~\eqref{eq:lemme_4.2} and the notations~\eqref{eq:A_0} and~\eqref{eq:B}, we have that 
\begin{equation}
F_{\xi}( v +  u) - F_{\xi}( v) - \left\{\frac{1}{p}\int_{\R^d} B + \int_{\R^d} h \cdot \nabla u \right\} = \frac{1}{p}\int_{\R^d} A.
\label{eq:lemme_4.2_1}
\end{equation}
Defining 
$$\begin{aligned}
L_v(u) &:=  \frac{1}{p}\int_{\R^d} B + \int_{\R^d} h \cdot \nabla u \\ &=\int_{\R^d} a \left( (\xi + \nabla w_{\xi}^{\per} + \nabla v)|\xi + \nabla w_{\xi}^{\per} + \nabla v|^{p-2} - (\xi + \nabla w_{\xi}^{\per})|\xi + \nabla w_{\xi}^{\per}|^{p-2} \right) \cdot \nabla u + \int_{\R^d} h \cdot \nabla u
\end{aligned}$$
and noting that, thanks to~\eqref{eq:B_int}, $L_v$ is a bounded linear form on $W_{\xi + \nabla w_{\xi}^{\per}}$, we have, gathering~\eqref{eq:lemme_4.2_1} and~\eqref{eq:A_1} together,
$$F_{\xi}(v +  u) - F_{\xi}(v) - L_v(u) = O_{u \rightarrow 0} \left(\|u\|^2_{W_{\xi + \nabla w_{\xi}^{\per}}} \right).$$
Lemma~\ref{lem:diffF_u} is proved.
\end{proof}

\begin{proof}[Proof of Theorem~\ref{th:existencecor}]
We prove below that, for $h \in L^{p'}(\R^d)^d$, the PDE
\begin{equation}
- \text{div} a \left[ \big|\xi + \nabla w_{\xi}^{\per} + \nabla w_{\xi} \big|^{p-2}(\xi + \nabla w_{\xi}^{\per} + \nabla \widetilde{w_{\xi}}) - \big|\xi  + \nabla w_{\xi}^{\per} \big|^{p-2}(\xi + \nabla w_{\xi}^{\per}) \right] = \text{div}(h),
\label{eq:th2.2_gen}
\end{equation}
admits a unique solution $\widetilde{w_{\xi}} \in W_{\xi + \nabla w_{\xi}^{\per}}$ in the weak sense (see Definition~\ref{def:def}). Theorem~\ref{th:existencecor} is then proved by defining
\begin{equation}
h:= \widetilde{a} (\xi + \nabla w_{\xi}^{\per})|\xi + \nabla w_{\xi}^{\per}|^{p-2}.
\label{eq:f}
\end{equation}
Because of Proposition~\ref{prop:periodique} (ii) and Assumptions~\textbf{(A2)-(A3)}, it is clear that $h \in L^{p'}(\RR^d)^d$. Since~\eqref{eq:th2.2_gen} is solvable for this choice of $h$, Theorem~\ref{th:existencecor} is proved. 

\medskip

We are thus left to study the PDE~\eqref{eq:th2.2_gen} for an abstract right-hand side $h \in L^{p'}(\R^d)^d$. 
  With Lemma~\ref{lem:elementary_cor}, Lemma~\ref{lem:diffF_u} and Lemma~\ref{lem:banach}, we prove in a standard way that Problem~\eqref{eq:th2.2_gen} admits a unique solution. Indeed, let us consider the minimization Problem:
 \begin{equation}
 \min_{v \in W_{\xi + \nabla w_{\xi}^{\per}}} \ F_{\xi}(v).
 \label{eq:min}
\end{equation}  
This Problem admits a unique solution. The existence is guaranteed by the following procedure: let $(v_n)_{n \in \NN} \subset W_{\xi + \nabla w_{\xi}^{\per}}$ be a minimizing sequence. Then, by the left-hand estimate of~\eqref{eq:coercivit}, we have that the sequence $\left(\| v_n \|_{W_{\xi + \nabla w_{\xi}^{\per}}} \right)_{n \in \NN}$ is bounded (see~\eqref{eq:normeWu} for the definition of $\|\cdot \|_{W_{\xi + \nabla w_{\xi}^{\per}}}$). By Lemma~\ref{lem:banach} (iii), we get that the sequence $(v_n)_{n \in \NN}$ weakly converges, up to a subsequence, to some $v$ in $W_{\xi + \nabla w_{\xi}^{\per}}$ when $n \longrightarrow +\infty$. Since by Lemma~\ref{lem:elementary_cor} (ii) and Lemma~\ref{lem:diffF_u}, $F_{\xi}$ is convex and continuous over $W_{\xi + \nabla w_{\xi}^{\per}}$, it is in particular weakly lower semi-continuous. Thus 
$$F_{\xi}(v) \leq \liminf_{n \rightarrow + \infty} F_{\xi}(v_n) = \inf_{W_{\xi + \nabla w_{\xi}^{\per}}} F_{\xi}.$$
This concludes the existence of a solution to \eqref{eq:min}. The uniqueness is given by the strict convexity of $F_{\xi}$, see Lemma~\ref{lem:elementary_cor} (ii). We finally note that the convexity of $F_{\xi}$ together with its differentiability ensure that being a solution to Problem \eqref{eq:min} is equivalent to solve the PDE \eqref{eq:th2.2_gen}, since \eqref{eq:Euler_Lag} is exactly the weak form of \eqref{eq:th2.2_gen} in the sense of Definition~\ref{def:def}. Theorem~\ref{th:existencecor} is proved.
\end{proof}

\section[Properties of the correctors]{Properties of the non-periodic correctors: proof of Theorem~\ref{th:th_nonlin}}
\label{sect:propcor}

\subsection{A useful Lemma}

We begin by introducing the following function: for all $\xi, \eta \in \R^d$, the function $G_{\xi,\eta}$ is defined over $\R^d \times \R^d$ by
\begin{equation}
G_{\xi,\eta}(X,Y) := |\xi + X|^p + |\eta + Y|^p - \left|\xi + \frac{X+Y}{2}\right|^p - \left|\eta + \frac{X+Y}{2} \right|^p - \frac{p}{2} \left(  \xi|\xi|^{p-2} - \eta |\eta|^{p-2} \right)\cdot (X - Y).
\label{eq:defcont}
\end{equation}

The following Lemma gives a lower bound for $G_{\xi,\eta}$ that will allow to prove Theorem~\ref{th:coercivitypercor} (iii).

\begin{lemme} Suppose that $2 \leq p < 3$. For all $\delta > 0$,
there exist constants $\gamma_p  = \gamma(p) > 0$ and $c_p = c(p)>0$ such that for all $X,Y \in \R^d$, for all $\xi \in \R^d \setminus B(0,\delta)$ and $\eta \in B(\xi,\delta/2)$, we have that 
\begin{equation}
G_{\xi,\eta}(X,Y) \geq \gamma_p |X - Y|^p - c_p\left\{ |\xi - \eta|^{p-2} |X-Y| + \delta^{p-3}|\xi - \eta||X+Y| \right\}|X - Y|.
\label{eq:lemcont}
\end{equation}
Suppose that $p \geq 3$. There exist constants $\gamma_p = \gamma(p) > 0$ and $c_p = c(p) > 0$ such that for all $X, Y \in \R^d$ and all $\xi, \eta \in \R^d$, 
\begin{equation}
\begin{aligned}
G_{\xi,\eta}&(X,Y) \geq \gamma_p |X - Y|^p \\
& - c_p\left\{ |\xi - \eta|^{p-2} |X-Y| + |\xi - \eta||X + Y|^{p-2} + (|\xi| + |\eta|)^{p-3}|\xi - \eta||X+Y| \right\}|X - Y|.
\end{aligned}
\label{eq:lemcont2}
\end{equation}
\label{lem:lemcont}
\end{lemme}

\begin{proof}
[Proof of Lemma~\ref{lem:lemcont}] We first give the proof of Estimate~\eqref{eq:lemcont}. We have that $\xi \neq 0$ and $\eta \neq 0$. For all $X,Y \in \R^d$, we define $Z := \frac{X-Y}{2}$ and $T := \frac{X+Y}{2}$. Inequality~\eqref{eq:lemcont} is equivalent to the following inequality: for any $Z, T \in \R^d$,
\begin{equation}
\begin{aligned}
|\xi + T + Z|^p + |\eta + T - Z|^p - &|\xi + T|^p - |\eta + T|^p - p (\xi|\xi|^{p-2} - \eta|\eta|^{p-2})\cdot Z \\ 
& \geq \gamma_p|Z|^p - c_{p}\left\{  |\xi-\eta|^{p-2}|Z| + \delta^{p-3}|\xi-\eta||T| \right\}|Z|.
\end{aligned}
\label{eq:eqZT}
\end{equation}
We prove~\eqref{eq:eqZT} for any $Z,T \in \R^d$.
We fix $T \in \R^d$ and we introduce the function
$$\Phi_{\gamma_p}(Z):= |\xi + T + Z|^p + |\eta + T - Z|^p - \gamma_p |Z|^p,$$
where $\gamma_p> 0$ is to be chosen later.
Since $p \geq 2$, the function $\Phi_{\gamma_p}$ is of class $\mathcal{C}^2$. Besides, denoting by~$\I$ the identity matrix, we have that 
$$ \begin{aligned}
\Phi''_{\gamma}(Z) & = p|\xi + T + Z|^{p-2} \text{I} + p(p-2)|\xi + Z + T|^{p-4}(\xi + Z + T) \otimes (\xi + Z + T) + p|\eta + T - Z|^{p-2} \text{I} \\
& + p(p-2)|\eta + T - Z|^{p-4}(\eta + T - Z) \otimes (\eta + T - Z) - \gamma_p p |Z|^{p-2} \text{I} - \gamma_p p(p-2) |Z|^{p-4} Z \otimes Z.
\end{aligned}$$
Thus, for all $h \in \R^d$,
\begin{equation}
\begin{aligned}
\Phi''_{\gamma}(Z)(h,h) & \geq p |\xi + T + Z|^{p-2}|h|^2 + p|\eta + T - Z|^{p-2} |h|^2 - \gamma_p p|Z|^{p-2}|h|^2 - \gamma_p p(p-2)|Z|^{p-4}(Z \cdot h)^2\\
& \geq  p \left[ |\xi + T + Z|^{p-2} + |\eta + T - Z|^{p-2} - \gamma_p(p-1)|Z|^{p-2} \right]|h|^2.
\end{aligned}
\label{eq:borne}
\end{equation}
We next note that
\begin{equation}
\begin{aligned}
|Z|^{p-2} & = \left|\frac{1}{2}(Z + \xi +T) + \frac{1}{2}(Z - \eta - T) + \frac{1}{2}(\eta - \xi) \right|^{p-2} \\
& \leq C(p) \left(|\xi + T + Z|^{p-2} + |\eta + T - Z|^{p-2} + |\xi - \eta|^{p-2} \right),
\end{aligned}
\label{eq:lemme_calc1}
\end{equation}
where we have used the triangle inequality together with the fact that for all $m \geq 1$ and $p\geq 2$, there exists a constant $C(p,m)$ such that 
$$\forall a_1,...,a_m \geq 0, \quad (a_1 + \cdots + a_m)^{p-2} \leq C(p,m) \left( a_1^{p-2} + \cdots + a_m^{p-2} \right).$$
Estimate~\eqref{eq:lemme_calc1} together with inequality~\eqref{eq:borne} give that 
$$\forall h \in \R^d, \quad \Phi''_{\gamma_p}(Z)(h,h) \geq - p |\xi - \eta|^{p-2} |h|^2$$
for $$\gamma_p := \frac{1}{C(p)(p-1)}.$$
The function $\Phi_{\gamma_p} + \frac{p}{2} |\xi - \eta|^{p-2} |\cdot|^2$ is convex, hence
$$\forall Z \in \R^d, \quad \Phi_{\gamma_p}(Z) + \frac{p}{2} |\xi - \eta|^{p-2} |Z|^2 \geq \Phi_{\gamma_p}(0) + \nabla \Phi_{\gamma_p}(0)\cdot Z.$$
We have thus proved that 
$$\Phi_{\gamma_p}(Z) \geq |\xi + T|^p + |\eta + T|^p + p\left[ (\xi + T)|\xi + T|^{p-2} - (\eta + T)|\eta + T|^{p-2} \right] \cdot Z - \frac{p}{2} |\xi - \eta|^{p-2} |Z|^2.$$
This proves estimate~\eqref{eq:eqZT} if $T = 0$. If $T \neq 0$, it remains to prove that
\begin{equation}
\big|(\xi + T)|\xi + T|^{p-2} - (\eta + T)|\eta + T|^{p-2} - \xi |\xi|^{p-2} + \eta |\eta|^{p-2} \big| \leq c_{p} \delta^{p-3}|\xi - \eta||T|.
\label{eq:accfinis}
\end{equation}
We want to apply the mean-value inequality to the function $\Psi_T$ defined by
$$\Psi_{T}(x) := |x+T|^{p-2}(x+T) - x|x|^{p-2}, \quad x \in [\xi, \eta] \subset \R^d \setminus B(0,\delta/2),$$
which is differentiable over $\R^d$.
We have that 
$$\Psi_T'(x) = \left( |x+T|^{p-2} - |x|^{p-2} \right)\I + \left((x + T) \otimes (x + T)|x+T|^{p-4} - x \otimes x |x|^{p-4} \right).$$
We now note that there exists a constant $C_p > 0$ such that for all $x \in \R^d \setminus B(0,\delta/2)$,
\begin{equation}
\left| |x+T|^{p-2} - |x|^{p-2} \right| \leq C_p \left\{\delta^{p-3}|T| + |x|^{p-3}|T| \right\}
\label{eq:estimaccfinis}
\end{equation}
and
\begin{equation}
\left|(x + T) \otimes (x + T)|x+T|^{p-4} - x \otimes x |x|^{p-4} \right| \leq C_p \left\{\delta^{p-3}|T| + |x|^{p-3}|T| \right\}.
\label{eq:estimaccfinis2}
\end{equation}
Noting that $|x|^{p-3} \leq (\frac{1}{2})^{p-3}|\delta|^{p-3}$ since $p \leq 3$, we have proved~\eqref{eq:accfinis}. The proof of Lemma~\ref{lem:lemcont} is completed up to the justification of~\eqref{eq:estimaccfinis}-\eqref{eq:estimaccfinis2}. 

\medskip

\noindent\textit{Proof of~\eqref{eq:estimaccfinis} and~\eqref{eq:estimaccfinis2}}. We concentrate on the first inequality: assume first that $|T| \geq \frac{1}{2}|x| \geq \frac{1}{4}|\delta|$, then 
\begin{equation}
\left| |x+T|^{p-2} - |x|^{p-2} \right| \leq C_p|T|^{p-2} \leq C_p \delta^{p-3} |T|.
\label{eq:borne1}
\end{equation}
We now treat the case $|T| \leq \frac{1}{2}|x|$. In particular $\big|\frac{T}{|x|}\big| \leq \frac{1}{2}$ and thus
\begin{equation}
\left| |x+T|^{p-2} - |x|^{p-2} \right| = |x|^{p-2} \left| \big|\frac{x}{|x|} + \frac{T}{|x|} \big|^{p-2} - \big|\frac{x}{|x|}\big|^{p-2} \right| \leq C_p|x|^{p-2} \big|\frac{T}{|x|}\big| = C_p|T| |x|^{p-3},
\label{eq:borne2}
\end{equation}
since the function $y \mapsto \big|\frac{x}{|x|} + y\big|^{p-2}$ is regular on $B(0,\frac{3}{4})$ with derivative uniformly bounded in $x$. Estimate~\eqref{eq:estimaccfinis2} is proved the same way. We have concluded the proof.

\medskip

\noindent\textit{Proof of~\eqref{eq:lemcont2}}. We assume that $p \geq 3$. With the above variables $T$ and $Z$, \eqref{eq:lemcont2} is equivalent to proving that for all $Z,T,\xi$ and $\eta \in \R^d$, the following inequality holds true:
\begin{equation}
\begin{aligned}
|\xi + T + Z|^p + &|\eta + T - Z|^p - |\xi + T|^p - |\eta + T|^p - p (\xi|\xi|^{p-2} - \eta|\eta|^{p-2})\cdot Z \\ 
& \geq \gamma_p|Z|^p - c_{p}\left\{  |\xi-\eta|^{p-2}|Z| + |\xi - \eta| |T|^{p-2} + (|\xi| + |\eta)^{p-3} |\xi - \eta||T| \right\}|Z|.
\end{aligned}
\label{eq:eqZT_2}
\end{equation}
Applying the same method as for the proof of~\eqref{eq:lemcont}, we only have to prove that
\begin{equation}
\begin{aligned}
\big|(\xi + T)|\xi + T|^{p-2} - (\eta + T)|\eta + T|^{p-2} & - \xi |\xi|^{p-2}  + \eta |\eta|^{p-2} \big|\\& \leq c_{p}\left\{|\xi - \eta||T|^{p-2} + (|\xi| + |\eta)^{p-3} |\xi - \eta||T|\right\}.
\end{aligned}
\label{eq:accfinis2}
\end{equation}
We once again appeal to the mean-value inequality on $\Psi_T$, noticing that, in this case, see~\eqref{eq:borne1} and~\eqref{eq:borne2}, we have for all $x \in \R^d$, 
\begin{equation}
|\Psi_T'(x)| \leq C_p \left\{|T|^{p-2} + |x|^{p-3}|T| \right\} \leq C_p \left\{|T|^{p-2} + (|\xi| + |\eta|)^{p-3}|T| \right\}, \quad x \in [\xi,\eta].
\label{eq:borne4}
\end{equation}
Note that, contrary to the case $p < 3$, estimate~\eqref{eq:borne4} does not depend on $\delta$. 
This gives~\eqref{eq:accfinis2} and finally~\eqref{eq:eqZT_2}.
\end{proof}

\subsection{Proof of Theorem~\ref{th:coercivitypercor}}

We start this section with a Remark:
\begin{remarque} The proofs of Theorem~\ref{th:th_nonlin} (i) and~\ref{th:th_nonlin} (ii) below do not use Theorem~\ref{th:coercivitypercor}. Consequently, we may use freely the results of   Theorem~\ref{th:th_nonlin} (i) and~\ref{th:th_nonlin} (ii) in the following proof.
\label{re:theorem}
\end{remarque}

\begin{proof}
[Proof of Theorem~\ref{th:coercivitypercor}]
By homogeneity, we can prove Theorem~\ref{th:coercivitypercor} for all $\xi \in \R^d$ such that $|\xi| = 1$. We fix such a $\xi \in \R^d$. By \textbf{(A4)}, there exists a constant $c > 0$ independent of $\xi$ such that $|\xi + \nabla w_{\xi}^{\per}| \geq c$. In the proof, we introduce the notations
\begin{equation}
C_{\infty}^{\per} := \sup_{|\xi|=1}\|\xi + \nabla w_{\xi}^{\per} \|_{L^{\infty}(Q)} \quad \text{and} \quad C_{\infty} := \sup_{|\xi|=1}\|\nabla \widetilde{w_{\xi}} \|_{L^{\infty}(\R^d)},
\label{eq:notation-taylor}
\end{equation}
where these quantities are well-defined owing to Proposition~\ref{prop:periodique} (ii) and Theorem~\ref{th:th_nonlin} (ii).
We use the following Taylor inequality~\eqref{eq:taylor} for the function $y \longmapsto \big(\xi + \nabla w_{\xi}^{\per} + y \big) \big| \xi + \nabla w_{\xi}^{\per} + y \big|^{p-2}$ which is of class $\mathcal{C}^2$ over $B(0,3c/4)$. For all $y \in \R^d$, we have, using also~\eqref{eq:useful_ineq_3} when $|y| \geq c/2$,
\begin{equation}
\begin{aligned}
\bigg|\big(\xi & + \nabla w_{\xi}^{\per} + y \big) \big| \xi + \nabla w_{\xi}^{\per} + y \big|^{p-2} - (\xi + \nabla w_{\xi}^{\per}) \big| \xi + \nabla w_{\xi}^{\per} \big|^{p-2} 
\\& - \big\{|\xi + \nabla w_{\xi}^{\per}|^{p-2} \text{I} + (p-2)|\xi + \nabla w_{\xi}^{\per}|^{p-4} (\xi + \nabla w_{\xi}^{\per}) \otimes (\xi + \nabla w_{\xi}^{\per}) \big\} y \bigg|
\\& \leq C(p,c) |y|^2 1_{\{|y| \leq c/2\}} + C(p)\bigg\{ |\xi + \nabla w_{\xi}^{\per}|^{p-2}|y| + |y|^{p-1} \bigg\} 1_{\{|y| \geq c/2\}} \\
& \leq C(p,c) |y|^2 1_{\{|y| \leq c/2\}}  + C(p,c) \bigg\{ (C_{\infty}^{\per})^{p-2}|y|^2 + |y|^{\max(2,p-1)} \bigg\} 1_{\{|y| \geq c/2\}} \\
& \leq C(p,c,C_{\infty}^{\per})\big(|y|^{2} + |y|^{\max(2,p-1)}\big).
\end{aligned}
\label{eq:taylor}
\end{equation}
By~\eqref{eq:taylor} applied with $y = \nabla \widetilde{w_{\xi}}$, we can write
\begin{equation}
\begin{aligned}
\big(\xi & + \nabla w_{\xi}^{\per} + \nabla\widetilde{w_{\xi}} \big) \big| \xi + \nabla w_{\xi}^{\per} + \nabla\widetilde{w_{\xi}} \big|^{p-2} - (\xi + \nabla w_{\xi}^{\per}) \big| \xi + \nabla w_{\xi}^{\per} \big|^{p-2} 
\\& = \big[|\xi + \nabla w_{\xi}^{\per}|^{p-2} \text{I} + (p-2)|\xi + \nabla w_{\xi}^{\per}|^{p-4} (\xi + \nabla w_{\xi}^{\per}) \otimes (\xi + \nabla w_{\xi}^{\per}) \big] \nabla \widetilde{w_{\xi}} + g_{\xi}(\nabla \widetilde{w_{\xi}}),
\label{eq:lem3.4}
\end{aligned}
\end{equation}
where, using~\eqref{eq:notation-taylor}, 
\begin{equation}
\left|g_{\xi}(\nabla\widetilde{w_{\xi}})\right| \leq C \big(p,c,C_{\infty}^{\per},C_{\infty}\big) | \nabla\widetilde{w_{\xi}}|^{2},
\label{eq:gz}
\end{equation}
Thus, collecting \eqref{eq:lem3.4} and~\eqref{eq:th2.2_gen}, we get that $\nabla \widetilde{w_{\xi}}$ solves
\begin{equation}\begin{aligned} 
-\di \ a \big[|\xi + \nabla w_{\xi}^{\per}|^{p-2} \text{I} + (p-2)|\xi + \nabla w_{\xi}^{\per}|^{p-4} (\xi + \nabla w_{\xi}^{\per}) & \otimes (\xi + \nabla w_{\xi}^{\per}) \big] \nabla \widetilde{w_{\xi}} \\ &= \di(h) + \di(a g_{\xi}(\nabla \widetilde{w_{\xi}}))
\end{aligned}
\label{eq:BLL}
\end{equation}
in the distribution sense. Equation \eqref{eq:BLL} is of the form
\begin{equation}
-\di \left(A_{\xi}\ \nabla \widetilde{w_{\xi}} \right) = \di(h) + \di(a g_{\xi}(\nabla \widetilde{w_{\xi}})), 
\label{eq:BLL2}
\end{equation}
where
\begin{equation} A_{\xi} := a \left(|\xi + \nabla w_{\xi}^{\per}|^{p-2} \text{I} + (p-2)|\xi + \nabla w_{\xi}^{\per}|^{p-4} (\xi + \nabla w_{\xi}^{\per})\otimes (\xi + \nabla w_{\xi}^{\per})\right).
\end{equation}
We may write that
$A_{\xi} = A_{\xi}^{\per} + \widetilde{A_{\xi}}$,
where 
$$A_{\xi}^{\per} := a^{\per } \left(|\xi + \nabla w_{\xi}^{\per}|^{p-2} \text{I} + (p-2)|\xi + \nabla w_{\xi}^{\per}|^{p-4} (\xi + \nabla w_{\xi}^{\per}) \otimes (\xi + \nabla w_{\xi}^{\per}) \right)$$
and
$$\widetilde{A_{\xi}} := \widetilde{a} \left(|\xi + \nabla w_{\xi}^{\per}|^{p-2} \text{I} + (p-2)|\xi + \nabla w_{\xi}^{\per}|^{p-4} (\xi + \nabla w_{\xi}^{\per}) \otimes (\xi + \nabla w_{\xi}^{\per}) \right).$$
The matrix $A_{\xi}^{\per}$ is symmetric, periodic, H\"{o}lder continuous, bounded and coercive while the matrix $\widetilde{A_{\xi}} \in L^{p'} \cap L^{\infty}(\R^d)^{d \times d}$ by Assumption \textbf{(A3)}, in particular $\widetilde{A_{\xi}} \nabla \widetilde{w_{\xi}} \in L^{p'} \cap L^{\infty}(\R^d)^d$ due to Proposition~\ref{prop:periodique} (ii) and Theorem~\ref{th:th_nonlin} (ii). We write equation \eqref{eq:BLL2} as
\begin{equation}
-\di \left(A_{\xi}^{\mathrm{per}}\ \nabla \widetilde{w_{\xi}} \right) = \di(h + a g_{\xi}(\nabla \widetilde{w_{\xi}}) + \widetilde{A_{\xi}} \nabla \widetilde{w_{\xi}}).
\label{eq:BLL1}
\end{equation} 
We have that $h \in L^{p'} \cap L^{\infty}(\R^d)$ and, thanks to the estimate~\eqref{eq:gz} and the fact that $\nabla \widetilde{w_{\xi}} \in L^p(\R^d)$, that $a g_{\xi}(\nabla \widetilde{w_{\xi}}) \in L^{p/2} \cap L^{\infty}(\R^d)$. Thus
$$h + a g_{\xi}(\nabla \widetilde{w_{\xi}}) + \widetilde{A_{\xi}} \nabla\widetilde{w_{\xi}} \in \left(L^{\max(p',p/2)} \cap L^{\infty}(\R^d)\right)^d.$$
Applying~\cite[Theorem p. 247]{ALliouville} and~\cite[Theorem A]{avellaneda1991lp} to~\eqref{eq:BLL1} gives $\nabla \widetilde{w_{\xi}} \in L^{\max(p',p/2)}(\R^d)$ with the estimate
\begin{equation}
\begin{aligned}
\|\nabla \widetilde{w_{\xi}}\|_{L^{\max(p',p/2)}(\R^d)} & \leq C(d,p,c,C_{\infty}^{\per},\alpha)\big\|h + a g_{\xi}(\nabla \widetilde{w_{\xi}}) + \widetilde{A_{\xi}} \nabla\widetilde{w_{\xi}}\big\|_{L^{\max(p',p/2)}(\R^d)}
\\ & \underset{\eqref{eq:gz}}{\leq}C \big(d,p,c,C_{\infty}^{\per},C_{\infty},\lambda\big)\big( \|\widetilde{a}\|_{L^{\max(p',p/2)}(\R^d)} (C^{\per}_{\infty})^{p-1} +  \big\||\nabla \widetilde{w_{\xi}}|^2 \big\|_{L^{\max(p',p/2)}(\R^d)} \\
& \quad \quad \quad +  \|\widetilde{a}\|_{L^{\max(p',p/2)}(\R^d)} (C^{\per}_{\infty})^{p-2} \|\nabla \widetilde{w_{\xi}}\|_{L^{\infty}(\R^d)} \big) \\
& \leq C\big(\|\widetilde{a}\|_{L^{p'}(\R^d)},\lambda,d,p,\alpha,c,C_{\infty}^{\per},C_{\infty},C_p \big),
\end{aligned} 
\label{eq:abaisserlareg}
\end{equation}
where $C_p = \sup_{|\xi|=1} \|\nabla \widetilde{w_{\xi}}\|_{L^p(\R^d)}$.
If $p' \geq p/2$, Theorem~\ref{th:coercivitypercor} is proved. 
Otherwise, $\nabla\widetilde{w_{\xi}} \in L^{p/2}(\R^d)^d$ and we iterate the argument. We have, thanks to~\eqref{eq:gz}, that
 $$h + a g_{\xi}(\nabla\widetilde{w_{\xi}}) + \widetilde{A_{\xi}} \nabla \widetilde{w_{\xi}} \in \left(L^{\max(p',p/4)}\cap L^{\infty}(\R^d)\right)^d,$$
 thus by \cite{ALliouville}, we get that $\nabla \widetilde{w_{\xi}} \in L^{\max(p',p/4)}(\R^d)^d$ and we can prove, similarly to~\eqref{eq:abaisserlareg} that
 \begin{equation}
 \|\nabla \widetilde{w_{\xi}}\|_{L^{\max(p',p/4)}(\R^d)} \leq C\big(\|\widetilde{a}\|_{L^{p'}(\R^d)},\lambda,d,p,\alpha,c,C_{\infty}^{\per},C_{\infty},C_p \big),
 \label{eq:abaisserlareg_2}
 \end{equation}
 where the constant on the right-hand side of~\eqref{eq:abaisserlareg_2} is potentially greater than the one on the right-hand side of~\eqref{eq:abaisserlareg} but the dependance on the data remains the same. If $p' \geq p/4$, the Theorem is proved. 
Otherwise, we iterate similarly. The procedure ends at step $k$ for which $p/2^k \leq p'$: we thus obtain that $\nabla \widetilde{w_{\xi}} \in L^{p'}(\R^d)^d$ and that there exists a constant $C_{final} := C\big(\widetilde{a},\lambda,d,p,\alpha,c,C_{\infty}^{\per},C_{\infty},C_p  \big)$ such that $$\|\nabla \widetilde{w_{\xi}}\|_{L^{p'}(\R^d)} \leq C_{final}.$$ Theorem~\ref{th:coercivitypercor} is proved. 
\end{proof}

\subsection{Proof of Theorem~\ref{th:th_nonlin}}
\label{subsect:th}

\paragraph*{Proof of \textit{(i)}.} This is due to Proposition~\ref{prop:periodique} (i), to the form of the PDE~\eqref{eq:th2.2_gen2}-\eqref{eq:f_2} defining $\nabla\widetilde{w_{\xi}}$ and the fact that this PDE is uniquely solvable in the sense of Definition~\ref{def:def}. Note that we use that for $t \neq 0$, $W_{\xi + \nabla w_{\xi}^{\per}} = W_{t\xi + \nabla w_{t\xi}^{\per}}$.

\paragraph*{Proof of \textit{(ii)}.} This result is proved in~\cite[Lemma 2.2]{wang2019convergence} but we reproduce the proof here for the sake of completeness. Let $\xi \in \RR^d$. By Definition~\ref{def:def} with $\nabla \phi = \nabla \widetilde{w_{\xi}}$, the inequality~\eqref{eq:useful_ineq_1},  H\"{o}lder inequality together with~\eqref{eq:f}, we have
\begin{equation}
c \int_{\R^d} |\nabla \widetilde{w_{\xi}}|^p \leq \|f\|_{L^{p'}(\R^d)} \|\nabla \widetilde{w_{\xi}}\|_{L^p(\R^d)} \leq \|\widetilde{a}\|_{L^{p'}(\R^d)}\|\xi + \nabla w_{\xi}^{\per}\|_{L^{\infty}(Q)}^{p-1}\|\nabla \widetilde{w_{\xi}}\|_{L^p(\R^d)}.
\label{eq:th2.3(ii)}
\end{equation}
Thus, by Proposition~\ref{prop:periodique} (ii) and \eqref{eq:th2.3(ii)}, we obtain the first estimate of~\eqref{eq:estim_holder_per}. 

\medskip

We show that there exists $\alpha > 0$ independent of $\xi$ such that $\nabla w_{\xi} \in \mathcal{C}^{0,\alpha}(\R^d)$. We introduce the function $\overline{w_{\xi}} := \xi \cdot x + w_{\xi}$, then $\nabla \overline{w_{\xi}}$ solves the standard homogeneous $p-$Laplace equation with varying coefficient $a$. Applying \cite[Theorem~1]{kuusi2014nonlinear}, we get that $\nabla \overline{w_{\xi}}$ is continuous over $\R^d$. Besides, by \cite[Theorem~4]{kuusi2014nonlinear}, there exists a constant $c \geq 1$ and a radius $r > 0$ depending only on $d$, $p$, $\lambda$ and the Lipschitz constant of $a$, denoted $a_{Lip}$ such that for all $x \in \R^d$, 
\begin{equation}
|\nabla \overline{w_{\xi}}(x)| \leq c \left(\fint_{B(x,r)} |\nabla \overline{w_{\xi}}|^{p'} \right)^{1/p'} \leq c\left(\fint_{B(x,r)} |\nabla \overline{w_{\xi}}|^{p} \right)^{1/p}.
\end{equation}
Due to the form of $\nabla \overline{w_{\xi}}$, see also~\eqref{eq:w_xi} and~the first estimate of~\eqref{eq:estim_holder_per}, we have that 
\begin{equation}
|\nabla \overline{w_{\xi}}(x)| \leq c |\xi| + c r^{-d/p} \|\nabla \widetilde{w_{\xi}} \|_{L^p(\R^d)} \leq C(d,p,\lambda,a_{Lip})|\xi|.
\label{eq:th2.3(ii)_2}
\end{equation}
In particular,~\eqref{eq:th2.3(ii)_2} proves that $\nabla w_{\xi}$ is bounded and that $\|\nabla w_{\xi}\|_{L^{\infty}(\R^d)} \leq C(d,p,\lambda,a_{Lip})|\xi|$. 
By Assumption~\textbf{(A2)}, the non-linear operator $a(y,z)=a(y)z|z|^{p-2}$ falls into the scope of~\cite{dibenedetto1982}. Let $x \in \R^d$, up to subtracting of $\overline{w_{\xi}}(x)$, we have by~\eqref{eq:th2.3(ii)_2} that $|\overline{w_{\xi}}| \leq C(d,p,\lambda,a_{Lip})|\xi|$ on $B(x,2)$. Thus, applying~\cite[Theorem~2]{dibenedetto1982}, there exist $\alpha > 0$ and $C_0 >0$ depending only on $\lambda, a_{Lip}, p, d, p,$ and $C(d,p,\lambda,a_{Lip})|\xi|$ such that $\nabla \overline{w_{\xi}} \in \mathcal{C}^{0,\alpha}(B(x,1))$ and 
\begin{equation}
[\nabla \overline{w_{\xi}}]_{\mathcal{C}^{0,\alpha}(B(x,1))} \leq C_0.
\label{eq:th2.3(ii)_3}
\end{equation}
To specify the dependence of $C_0$ in $\xi$, we first take $|\xi|=1$ and we then apply the homogeneity, Theorem~\ref{th:th_nonlin} (i). This gives that $C_0 = C_0(p,d,\lambda,a_{Lip})|\xi|$ and concludes the proof of (ii), gathering \eqref{eq:th2.3(ii)_2} and~\eqref{eq:th2.3(ii)_3} and the fact that $$\|\nabla \widetilde{w_{\xi}} \|_{\mathcal{C}^{0,\alpha}(\R^d)} \leq |\xi| + \|\nabla \overline{w_{\xi}} \|_{\mathcal{C}^{0,\alpha}(\R^d)}.$$

\paragraph*{Proof of \textit{(iii)}.}
We assume that $2 \leq p < 3$. Let us fix $\xi \in \R^d$ such that $|\xi| = 1$. In the proof, $c > 0$ will denote a universal constant given by~\textbf{(A4)}. We consider $\eta \in \R^d$ such that $\xi \neq \eta$. In the sequel, we fix $\delta_0 \in (0,1)$ such that 
$C ( 1 + 2^{1-\gamma})\delta_0^{\gamma} + \delta_0 \leq c/2$, where $C$ and $\gamma$ are given by~\eqref{eq:continuité_Linfty}.

\medskip

\noindent\underline{Case 1}. We assume that $|\xi - \eta| \geq \delta_0$. Then, thanks to Theorem~\ref{th:th_nonlin} (ii), we have that 
\begin{equation}
\|\nabla \widetilde{w_{\xi}} - \nabla \widetilde{w_{\eta}} \|_{L^p(\R^d)} \leq C_p + C_p|\eta|. 
\end{equation}
We now note that for all $0 < \widetilde{\beta} \leq 1$,
\begin{equation}
C_p + C_p|\eta| \leq \begin{cases}
\begin{aligned}
 \big( \frac{C_p}{\delta_0^{\widetilde{\beta}}}\big)|\xi - \eta|^{\widetilde{\beta}}(1+|\eta|) \leq 2^{\widetilde{\beta}}\big( \frac{C_p}{\delta_0^{\widetilde{\beta}}}\big)|\xi - \eta|^{\widetilde{\beta}}(1+|\eta|^{1-\widetilde{\beta}}) \quad &\text{if} \quad |\eta| \leq 2. \\
C_pC(\widetilde{\beta})\big||\eta| - 1\big|^{\widetilde{\beta}} (1 + |\eta|^{1 - \widetilde{\beta}}) \leq C_pC(\widetilde{\beta}) |\xi - \eta|^{\widetilde{\beta}}(1+|\eta|^{1-\widetilde{\beta}}) \quad &\text{if} \quad |\eta| > 2,
\end{aligned}
\end{cases}
\label{eq:homogeneite_3}
\end{equation}
where we used that the function $x \mapsto \frac{1+x}{|x - 1|^{\widetilde{\beta}} (1 + x^{1 - \widetilde{\beta}})}$ is bounded on $[2,+\infty[$.
Thus 
\begin{equation}
\|\nabla \widetilde{w_{\xi}} - \nabla \widetilde{w_{\eta}} \|_{L^p(\R^d)} \leq C(\delta_0,\widetilde{\beta},C_p) |\xi - \eta|^{\widetilde{\beta}}(1+|\eta|^{1-\widetilde{\beta}}).
\label{eq:final}
\end{equation}
This gives~\eqref{eq:cont}.

\medskip

\noindent\underline{Case 2}. We assume that $|\xi - \eta| < \delta_0$. Then, by the choice of $\delta_0$ and Proposition~\ref{prop:periodique} (iv), we have that
\begin{equation}
\big\|\xi + \nabla w_{\xi}^{\per} - \big\{\eta + \nabla w_{\eta}^{\per}  \big\} \big\|_{L^{\infty}(Q)} \leq \frac{c}{2} \quad \text{and} \quad |\xi + \nabla w_{\xi}^{\per}| \geq c.
\label{eq:lem4.8_(3)}
\end{equation}
Recalling the notation~\eqref{eq:F_u}, we have that
\begin{equation}
F_{\xi}(\nabla \widetilde{w_{\xi}}) + F_{\eta}(\nabla \widetilde{w_{\eta}}) < F_{\xi} \left(\frac{\nabla \widetilde{w_{\xi}} + \nabla\widetilde{w_{\eta}}}{2} \right) + F_{\eta} \left(\frac{\nabla \widetilde{w_{\xi}} + \nabla\widetilde{w_{\eta}}}{2} \right) < +\infty,
\label{eq:milieu}
\end{equation}
where we have used that $\xi \neq \eta$, $F_z$ admits a unique minimizer for $z \in \R^d$ and $\nabla \widetilde{w_{\xi}} \in L^2(\R^d)$, $\nabla \widetilde{w_{\eta}} \in L^2(\R^d)$.
We recall that 
\begin{equation}
F_{z}^R(\nabla v) := \int_{B_R} a g_{z + \nabla w_z^{\per}}(\nabla v) + \int_{\R^d} f_z \cdot \nabla v, \quad z \in \R^d, \quad \nabla v \in L^p(\R^d)
\label{eq:notationF_r}
\end{equation}
and that $R \longmapsto F^R_{z}(\nabla v)$ is a non-decreasing function.
Thus, for $R$ large enough, we have the inequality
\begin{equation}
F_{\xi}^R(\nabla \widetilde{w_{\xi}}) + F_{\eta}^R(\nabla \widetilde{w_{\eta}}) - F_{\xi}^R \left(\frac{\nabla \widetilde{w_{\xi}} + \nabla\widetilde{w_{\eta}}}{2} \right) - F_{\eta}^R \left(\frac{\nabla \widetilde{w_{\xi}} + \nabla\widetilde{w_{\eta}}}{2} \right) \leq 0.
\label{eq:milieu2}
\end{equation}
We now use Lemma~\ref{lem:lemcont} applied with $\delta = c$. Taking into account~\eqref{eq:lem4.8_(3)}, this gives
\begin{equation}
\begin{aligned}
G_{\xi+\nabla w_{\xi}^{\per},\eta + \nabla w_{\eta}^{\per}}(\nabla \widetilde{w_{\xi}},&\nabla\widetilde{w_{\eta}}) \geq \gamma_p |\nabla \widetilde{w_{\xi}}-\nabla \widetilde{w_{\eta}}|^p - c_p \bigg\{|\xi+\nabla w_{\xi}^{\per} - (\eta + \nabla w_{\eta}^{\per})|^{p-2} | |\nabla \widetilde{w_{\xi}}-\nabla \widetilde{w_{\eta}}| \\
&  \quad + c^{p-3} |\xi+\nabla w_{\xi}^{\per} - (\eta + \nabla w_{\eta}^{\per})||\nabla \widetilde{w_{\xi}}+\nabla \widetilde{w_{\eta}}|\bigg\}|\nabla \widetilde{w_{\xi}}-\nabla \widetilde{w_{\eta}}|.
\end{aligned}
\label{lem:lem4.8}
\end{equation}
For all $R > 0$, we can integrate~\eqref{lem:lem4.8} over the ball $B_R$. Using the notation~\eqref{eq:notationF_r} and the form of the map $G_{\xi,\eta}(X,Y)$, see~\eqref{eq:defcont}, this yields
\begin{equation}
\begin{aligned}
F_{\xi}^R(\nabla \widetilde{w_{\xi}}) & + F_{\eta}^R(\nabla \widetilde{w_{\eta}}) - F_{\xi}^R \left(\frac{\nabla \widetilde{w_{\xi}} + \nabla\widetilde{w_{\eta}}}{2} \right) - F_{\eta}^R \left(\frac{\nabla \widetilde{w_{\xi}} + \nabla\widetilde{w_{\eta}}}{2} \right) - \frac{1}{2} \int_{\R^d} (h_{\xi} - h_{\eta})\cdot(\nabla \widetilde{w_{\xi}} - \nabla \widetilde{w_{\eta}}) \\
& \geq \gamma_p \int_{B_R} a|\nabla \widetilde{w_{\xi}}-\nabla \widetilde{w_{\eta}}|^p - c_p \int_{B_R} a \bigg\{|\xi+\nabla w_{\xi}^{\per} - (\eta + \nabla w_{\eta}^{\per})|^{p-2} | |\nabla \widetilde{w_{\xi}}-\nabla \widetilde{w_{\eta}}| \\ 
& \quad \quad \ \ \ \ \ \ \ \ \ \ \ \ \ \ \ \  + c^{p-3} |\xi+\nabla w_{\xi}^{\per} - (\eta + \nabla w_{\eta}^{\per})||\nabla \widetilde{w_{\xi}}+\nabla \widetilde{w_{\eta}}|\bigg\}|\nabla \widetilde{w_{\xi}}-\nabla \widetilde{w_{\eta}}|,
\label{eq:lem4/8_(2)}
\end{aligned}
\end{equation}
where $h_z = \widetilde{a}(z + \nabla w_{z}^{\per})|z + \nabla w_z^{\per}|^{p-2}$ for $z \in \R^d$. 
For $R$ large enough, we get because of~\eqref{eq:milieu2} that
\begin{equation}
\begin{aligned}
-\frac{1}{2} &\int_{\R^d}  (h_{\xi} - h_{\eta})\cdot(\nabla \widetilde{w_{\xi}} - \nabla \widetilde{w_{\eta}}) \geq \gamma_p \int_{B_R} a|\nabla \widetilde{w_{\xi}}-\nabla \widetilde{w_{\eta}}|^p - c_p \int_{B_R} a \bigg\{|\xi+\nabla w_{\xi}^{\per} - (\eta + \nabla w_{\eta}^{\per})|^{p-2}\cdot \\&|\nabla \widetilde{w_{\xi}}-\nabla \widetilde{w_{\eta}}| +c^{p-3} |\xi+\nabla w_{\xi}^{\per} - (\eta + \nabla w_{\eta}^{\per})||\nabla \widetilde{w_{\xi}}+\nabla \widetilde{w_{\eta}}|\bigg\}|\nabla \widetilde{w_{\xi}}-\nabla \widetilde{w_{\eta}}|,
\label{eq:lem4/8_(2)}
\end{aligned}
\end{equation}
Letting $R\longrightarrow +\infty$ in~\eqref{eq:lem4/8_(2)} and using Theorem~\ref{th:coercivitypercor}, we get by the monotone convergence Theorem that
$$\begin{aligned}
-\frac{1}{2} \int_{\R^d} (f_{\xi} - f_{\eta})\cdot&(\nabla \widetilde{w_{\xi}} - \nabla \widetilde{w_{\eta}}) \geq \gamma_p \int_{\R^d} |\nabla \widetilde{w_{\xi}}-\nabla \widetilde{w_{\eta}}|^p  - c_p \int_{\R^d} \bigg\{|\xi+\nabla w_{\xi}^{\per} - (\eta + \nabla w_{\eta}^{\per})|^{p-2} \cdot \\& |\nabla \widetilde{w_{\xi}}-\nabla \widetilde{w_{\eta}}| + c^{p-3} |\xi+\nabla w_{\xi}^{\per} - (\eta + \nabla w_{\eta}^{\per})||\nabla \widetilde{w_{\xi}}+\nabla \widetilde{w_{\eta}}|\bigg\}|\nabla \widetilde{w_{\xi}}-\nabla \widetilde{w_{\eta}}|.
\end{aligned}$$
Thus, applying the H\"{o}lder inequality, Proposition~\ref{prop:periodique} (iv) and Theorem~\ref{th:coercivitypercor} under the form $$\|\nabla \widetilde{w_z} \|_{L^{p'}(\R^d)} \leq C|z|, \quad z \in \R^d,$$ we get
$$\begin{aligned}\int_{\R^d} |\nabla \widetilde{w_{\xi}}-\nabla \widetilde{w_{\eta}}|^p & \leq C\bigg(\|\widetilde{a}\|_{L^{p'}}|\xi - \eta|^{\gamma}\|\nabla \widetilde{w_{\xi}}- \nabla \widetilde{w_{\eta}}\|_{L^p(\R^d)} + \\ &c_p \big\{ C|\xi - \eta|^{\gamma(p-2)} + c^{p-3}C|\xi-\eta|^{\gamma} \big\}\|\nabla \widetilde{w_{\xi}}- \nabla \widetilde{w_{\eta}}\|_{L^p}\bigg).
\end{aligned}$$
Thus
$$\|\nabla \widetilde{w_{\xi}}- \nabla \widetilde{w_{\eta}}\|_{L^p(\R^d)}^{p-1} \leq C|\xi - \eta|^{\gamma(p-2)}.$$
This gives~\eqref{eq:cont} when $|\xi| = 1$. The case $|\xi| \neq 1$ is treated by homogeneity. 

\medskip

Gathering Case 1 and Case 2, we have proved Theorem~\ref{th:th_nonlin} (iii) for $p \in [2,3)$. The proof of the case $p\geq 3$ is performed using the same method and~\eqref{eq:lemcont2}.

\begin{remarque}
As suggested by~\eqref{eq:lemcont2}, the assumptions of Theorem~\ref{th:th_nonlin}(iii) may be weakened when $p \geq 3$. In this case, it is sufficient to assume, instead of \textbf{(A4)}, that $\nabla \widetilde{w_{\xi}} \in L^{p'}(\R^d)$ , that $\nabla w_{\eta} \in L^{p'}(\R^d)$ and that we have an estimate of the form~\eqref{th:estimLp_prime}.
\end{remarque}

\paragraph*{Proof of \textit{(iv)}.} It is analogous to the proof of Proposition~\ref{prop:periodique} (iv).

\section[Qualitative homogenization]{Qualitative Homogenization: proof of Theorem~\ref{th:homog_qualititative}}
\label{sect:qual}

The proof of Theorem~\ref{th:homog_qualititative} is an adaptation of \cite{fusco1986homogenization} and \cite[Theorem~2.1]{dal1990correctors} to the present setting. We start with the following central~Lemma:

\begin{lemme} For $\xi \in \R^d$, let us write $\nabla \widetilde{w_{\xi}}$ the solution to~\eqref{eq:th2.2_gen2}-\eqref{eq:f_2} given by Theorem~\ref{th:existencecor}. Assume that the application 
\begin{equation} 
\begin{cases} \begin{aligned}
\R^d & \longrightarrow L^p_{\text{unif}}(\R^d) \\
\xi & \longmapsto \nabla \widetilde{w_{\xi}} \\
\end{aligned} \end{cases}
\label{eq:continuite_lemmecvremainder}
\end{equation}
is continuous. Then for all $\Psi \in L^p(\Omega)^d$,
\begin{equation}
\limsup_{\varepsilon \rightarrow 0} \int_{\Omega} \left|\nabla \widetilde{w_{M_{\varepsilon} \Psi}} \big( \frac{\cdot}{\varepsilon} \big) \right|^p = 0.
\label{eq:cclcv_remainder}
\end{equation}
\label{lem:cv_remainder}
\end{lemme}

\begin{proof}[Proof of Lemma~\ref{lem:cv_remainder}] We first show the following assertion:
\begin{equation}
\forall\delta >0,\quad \exists A > 0, \quad \forall |x| \geq A, \quad \forall\xi \in \R^d, \quad \|\nabla \widetilde{w_{\xi}} \|_{L^p(x+Q)} \leq \delta |\xi|.
\label{eq:lem_cvremainder}
\end{equation} 
By contradiction, if~\eqref{eq:lem_cvremainder} does not hold, there exists $\delta>0$ and two sequences $(x_n)_{n \in \NN} \subset \R^d$ and $(\xi_n)_{n \in \NN} \subset \R^d$ such that $|x_n| \underset{n \rightarrow +\infty}{\longrightarrow} + \infty$ and $\|\nabla \widetilde{w_{\xi_n}} \|_{L^p(x_n+Q)} \geq \delta |\xi_n|$. By Theorem~\ref{th:th_nonlin} (i), we can assume that $|\xi_n| = 1$. Thus, up to a subsequence, $\xi_n \underset{n \rightarrow + \infty}{\longrightarrow} \xi$. However, by~\eqref{eq:continuite_lemmecvremainder}, for all $n$ large enough, we have that $\|\nabla \widetilde{w_{\xi_n}} - \nabla \widetilde{w_{\xi}}\|_{L^p(x_n + Q)} \leq \delta/2$. Thus, for $n$ large enough, we have that $\|\nabla \widetilde{w_{\xi}} \|_{L^p(x_n + Q)} \geq \delta/2$. Since $|x_n| \underset{n \rightarrow + \infty}{\longrightarrow} + \infty$, this contradicts that $\nabla \widetilde{w_{\xi}} \in L^p(\R^d)$. Thus~\eqref{eq:lem_cvremainder} is satisfied.

\medskip

We now turn to the proof of~\eqref{eq:cclcv_remainder}. By an immediate application of the Jensen inequality, we have that
\begin{equation}
\forall B \in \NN \cup \{+\infty\}, \quad \sum_{|k| < B, \ \varepsilon(Q+k) \subset \Omega} \varepsilon^d \big|\Psi_{\varepsilon}^k\big|^p \leq \int_{\Omega \cap B_{\infty}(0,\varepsilon B)} |\Psi|^p, \quad \Psi_{\varepsilon}^k := \fint_{\varepsilon(Q+k)} \Psi,
\label{eq:intermediaire_cvremaindeer}
\end{equation}
where $B_{\infty}(x,r)$ denotes the ball centered in $x$ and of radius $r > 0$ for the $|\cdot|_{\infty}-$norm on $\R^d$.
 Let $\delta > 0$ and $A$ be given by~\eqref{eq:lem_cvremainder}. We have that
\begin{equation}
\begin{aligned}
\int_{\Omega} \left|\nabla \widetilde{w_{M_{\varepsilon} \Psi}} \big( \frac{\cdot}{\varepsilon} \big) \right|^p & = \sum_{k \in \ZZ^d, \ \varepsilon(Q+k) \subset \Omega} \varepsilon^d \int_{Q+k} \big|\nabla \widetilde{w_{\Psi_{\varepsilon}^k}}\big|^p \\
& \leq \sum_{|k| < A, \ \varepsilon(Q+k) \subset \Omega} \varepsilon^d \int_{Q+k} \big|\nabla \widetilde{w_{\Psi_{\varepsilon}^k}}\big|^p + \sum_{|k| \geq A, \ \varepsilon(Q+k) \subset \Omega} \varepsilon^d \int_{Q+k} \big|\nabla \widetilde{w_{\Psi_{\varepsilon}^k}}\big|^p \\
& \underset{\eqref{eq:lem_cvremainder},\eqref{eq:estim_holder_per}}{\leq} C \sum_{|k| < A, \ \varepsilon(Q+k) \subset \Omega} \varepsilon^d |\Psi_{\varepsilon}^k|^p + \delta^p \sum_{|k| \geq A, \ \varepsilon(Q+k) \subset \Omega} \varepsilon^d |\Psi_{\varepsilon}^k|^p \\
& \underset{\eqref{eq:intermediaire_cvremaindeer}}{\leq} C \int_{B_{\infty}(0,\varepsilon A)\cap \Omega} |\Psi|^p + \delta^p \int_{\Omega} |\Psi|^p.
\end{aligned} 
\end{equation}
By the dominated convergence Theorem, we have that 
\begin{equation}
\limsup_{\varepsilon\rightarrow 0} \int_{\Omega} \left|\nabla \widetilde{w_{M_{\varepsilon} \Psi}} \big( \frac{\cdot}{\varepsilon} \big) \right|^p \leq \delta^p \int_{\Omega} |\Psi|^p.
\label{lem:cv_remainder_final}
\end{equation}
Since~\eqref{lem:cv_remainder_final} is true for all $\delta > 0$, we haved proved~\eqref{eq:cclcv_remainder}.
\end{proof}

We now state the analogous of~\cite[Lemma~3.5]{dal1990correctors} to the present non-periodic setting. Before that, we introduce for $\xi, y \in \R^d$ the notations 
\begin{equation}
p^{\per}(y,\xi) := \xi + \nabla w^{\per}_{\xi}(y) \quad \text{and} \quad p(y,\xi) := \xi + \nabla w_{\xi}(y) = p^{\per}(y,\xi) + \nabla \widetilde{w_{\xi}}(y).
\label{eq:p_(x,xi)}
\end{equation}

\begin{lemme}
Assume that the Assumptions of Lemma~\ref{lem:cv_remainder} are satisfied. Let $\Psi \in L^p(\Omega)$ and $\Phi \in L^p(\Omega)$ such that $\Phi = \sum_{j=1}^m \eta_j 1_{\Omega_j}$ where $\bigcup_{j=1}^m \Omega_j \subset \subset \Omega$, $\Omega_k \cap \Omega_{\ell} = \emptyset$ for $k \neq \ell$ and $|\partial \Omega_j| = 0$ for $j \in \{1,m\}$. Then there exists a constant $C > 0$ independent of $\varepsilon$, $\Psi$ and $\Phi$ such that 
\begin{equation}
\limsup_{\varepsilon \rightarrow 0} \big\| p(\cdot/\varepsilon,M_{\varepsilon}\Psi) - p(./\varepsilon,\Phi) \big\|_{L^p(\Omega)} \leq C \left\{ \|\Psi\|_{L^p(\Omega)}^{1-\beta} + \|\Phi \|_{L^p(\Omega)}^{1-\beta} \right\} \|\Psi - \Phi\|_{L^p(\Omega)}^{\beta},
\label{eq:equiv_lem_3.5}
\end{equation}
where $\beta$ is given by Proposition~\ref{prop:periodique} (iii).
\label{lem:equiv_lem_3.5}
\end{lemme} 

\begin{proof}
[Proof of Lemma~\ref{lem:equiv_lem_3.5}] We first notice that 
\begin{equation}
\int_{\Omega} \left|\nabla \widetilde{w}_{\Phi}\big(\frac{\cdot}{\varepsilon}\big) \right|^p = \sum_{j=1}^m \int_{\Omega_j} \left|\nabla \widetilde{w}_{\eta_j}\big(\frac{\cdot}{\varepsilon}\big) \right|^p \leq \varepsilon^d \sum_{j=1}^m \int_{\R^d} \left|\nabla \widetilde{w}_{\eta_j} \right|^p \underset{\varepsilon\rightarrow 0}{\longrightarrow} 0.
\label{eq:lem_6.2}
\end{equation}
 With the notations~\eqref{eq:p_(x,xi)}, we have, applying \cite[Lemma 3.5]{dal1990correctors}, that 
$$
\begin{aligned}\limsup_{\varepsilon \rightarrow 0} \big\| p(\cdot/\varepsilon,M_{\varepsilon}\Psi) - p(./\varepsilon,\Phi) \big\|_{L^p(\Omega)} & \leq \underbrace{\limsup_{\varepsilon \rightarrow 0} \big\| p^{\per}(\cdot/\varepsilon,M_{\varepsilon}\Psi) - p^{\per}(./\varepsilon,\Phi) \big\|_{L^p(\Omega)}}_{\leq \text{RHS} \ \text{of} \ \eqref{eq:equiv_lem_3.5}} \\ & \quad + \underbrace{\limsup_{\varepsilon\rightarrow 0} \left\{\big\| \nabla \widetilde{w_{M_{\varepsilon}\Psi}}(./\varepsilon)\big\|_{L^p(\Omega)} + \big\|\nabla \widetilde{w_{\Phi}}(./\varepsilon) \big\|_{L^p(\Omega)} \right\}}_{=0 \ \text{by} \ \text{Lemma}~\ref{lem:cv_remainder} \ \text{and} \ \eqref{eq:lem_6.2}}.
\end{aligned}$$
\end{proof}

With these tools, we can prove Theorem~\ref{th:homog_qualititative}. The first point (i) is not detailed here since it is mainly a rewriting of \cite{fusco1986homogenization}. Note that for this point, the continuity of $\xi \longmapsto \nabla \widetilde{w_{\xi}}$ is not neeeded. The only result on the non-periodic correctors $\nabla w_{\xi}$, $\xi \in \R^d$ that is used is Theorem~\ref{th:th_nonlin} (ii). The proof of Theorem~\ref{th:homog_qualititative} (ii) follows the proof of~\cite[Theorem 2.1]{dal1990correctors}. In the following, we sketch the proof of Theorem~\ref{th:homog_qualititative} (ii) by insisting on the points that differ from~\cite{dal1990correctors}. The proof of Theorem~\ref{th:homog_qualititative} (iii) follows from~Theorem~\ref{th:homog_qualititative}~(ii) together with Lemma~\ref{lem:cv_remainder}.

\paragraph*{Sketch of proof of Theorem~\ref{th:homog_qualititative} (ii).} Since $M_{\varepsilon} \nabla u^*$ converges to $\nabla u^*$ when $\varepsilon\rightarrow 0$ in~$L^p(\Omega)$, it is sufficient to show, using the notation~\eqref{eq:p_(x,xi)} that 
\begin{equation}
R_{\varepsilon} := \nabla u_{\varepsilon} - p(./\varepsilon,M_{\varepsilon}\nabla u^*) \underset{\varepsilon \rightarrow 0}{\longrightarrow} 0 \quad \text{in} \quad L^p(\Omega).
\label{eq:dalmaso}
\end{equation} 
During the proof, we introduce a step function $\Phi$ as in~Lemma~\ref{lem:equiv_lem_3.5} satisfying $\|\nabla u^* - \Phi \|_{L^p(\Omega)} \leq \delta$.
By monotonicity of the $p-$Laplace operator, see~\eqref{eq:useful_ineq_1}, and Assumption~\textbf{(A1)}, we have that
\begin{equation}
\begin{aligned}
\lambda^{-1}c&\| R_{\varepsilon} \|_{L^p(\Omega)}^p \\ & \leq \int_{\Omega}\left\langle a(./\varepsilon)|\nabla u_{\varepsilon}|^{p-2}\nabla u_{\varepsilon} - a(./\varepsilon)|p(./\varepsilon,M_{\varepsilon} \nabla u^*)|^{p-2}p(./\varepsilon,M_{\varepsilon} \nabla u^*), \nabla u_{\varepsilon} - p(./\varepsilon,M_{\varepsilon}\nabla u^*) \right\rangle  \\ & =A_{\varepsilon} - B_{\varepsilon} - C_{\varepsilon} + D_{\varepsilon},
\end{aligned}
\label{eq:R_(eps)}
\end{equation}
where 
$$\begin{aligned} A_{\varepsilon} & := \int_{\Omega} a(./\varepsilon)|\nabla u_{\varepsilon}|^p, \quad B_{\varepsilon} := \int_{\Omega} a(./\varepsilon)|\nabla u_{\varepsilon}|^{p-2}\nabla u_{\varepsilon}\cdot p(./\varepsilon,M_{\varepsilon}\nabla u^*)\\ 
 C_{\varepsilon} & := \int_{\Omega} a(./\varepsilon)|p(./\varepsilon,M_{\varepsilon} \nabla u^*)|^{p-2}p(./\varepsilon,M_{\varepsilon} \nabla u^*) \cdot \nabla u_{\varepsilon} \quad \text{and} \quad  D_{\varepsilon} := \int_{\Omega} a(./\varepsilon)|p(./\varepsilon,M_{\varepsilon} \nabla u^*)|^{p}. \end{aligned}$$
The term $A_{\varepsilon}$ is obviously treated by the $L^p-$weak convergence $u_{\varepsilon} \underset{\varepsilon\rightarrow 0}{\relbar\joinrel\rightharpoonup} u^*$:
\begin{equation}
A_{\varepsilon} = \int_{\Omega} f u_{\varepsilon} \underset{\varepsilon\rightarrow 0}{\longrightarrow} \int_{\Omega} f u^* = \int_{\Omega} a^*(\nabla u^*)\cdot \nabla u^*.
\label{eq:A_(eps)}
\end{equation}
We study the term $B_{\varepsilon}$ when $M_{\varepsilon}\nabla u^*$ is replaced by $\Phi$. This gives:
$$\int_{\Omega} a(./\varepsilon)|\nabla u_{\varepsilon}|^{p-2}\nabla u_{\varepsilon}\cdot p(./\varepsilon,\Phi) = \sum_{j=1}^m \int_{\Omega_j} a(./\varepsilon)|\nabla u_{\varepsilon}|^{p-2}\nabla u_{\varepsilon}\cdot p(./\varepsilon,\eta_j).$$
We then apply the standard div-curl Lemma, keeping in mind that $a(./\varepsilon)|\nabla u_{\varepsilon}|^{p-2}\nabla u_{\varepsilon}$ converges $L^{p'}(\Omega)-$weakly to $a^*(\nabla u^*)$ when $\varepsilon\rightarrow 0$, that $p(./\varepsilon,\eta_j)$ converges $L^p-$weakly to $\eta_j$ and that, thanks to Theorem~\ref{th:th_nonlin}~(ii), $a(./\varepsilon)|\nabla u_{\varepsilon}|^{p-2}\nabla u_{\varepsilon}\cdot p(./\varepsilon,\eta_j)$ is bounded in $L^{p'}(\Omega)$, uniformly with respect to $\varepsilon$. Thus,
$$\sum_{j=1}^m \int_{\Omega_j} a(./\varepsilon)|\nabla u_{\varepsilon}|^{p-2}\nabla u_{\varepsilon}\cdot p(./\varepsilon,\eta_j) \underset{\varepsilon\rightarrow 0}{\longrightarrow} \sum_{j=1}^m \int_{\Omega_j} a^*(\nabla u^*)\cdot \eta_j = \int_{\Omega} a^*(\nabla u^*)\cdot \Phi.
$$ In view of Lemma~\ref{lem:equiv_lem_3.5}, we obtain that 
\begin{equation}
\limsup_{\varepsilon \rightarrow 0} \left|B_{\varepsilon} - \int_{\Omega} a^*(\nabla u^*)\cdot \Phi \right| = O(\delta^{\beta}),
\label{eq:B_(eps)}
\end{equation} where the $O$ is independent of $\delta$. The term $C_{\varepsilon}$ is also treated by replacing $M_{\varepsilon} \nabla u^*$ by $\Phi$ and using the div-curl Lemma. Noticing that $a(./\varepsilon)p(./\varepsilon,\eta_j)|p(./\varepsilon,\eta_j)|^{p-2} \underset{\varepsilon\rightarrow 0}{\relbar\joinrel\rightharpoonup} a^*(\eta_j)$ in $L^{p'}(\Omega)$, we obtain that
\begin{equation}
\limsup_{\varepsilon \rightarrow 0} \left|C_{\varepsilon} - \int_{\Omega} a^*(\Phi) \cdot \nabla u^* \right| = O(\delta^{\beta}).
\label{eq:C_(eps)}
\end{equation} 
We introduce $D_{\varepsilon}^{\per} := \int_{\Omega} a^{\per}(./\varepsilon)\left|p^{\per}(./\varepsilon,M_{\varepsilon} \nabla u^*)\right|^{p}$. By \cite[Step~1, pp.1161-1162]{dal1990correctors}, we have that 
\begin{equation}
D_{\varepsilon}^{\per} \underset{\varepsilon \rightarrow 0}{\longrightarrow} \int_{\Omega}  a^*(\nabla u^*)\cdot\nabla u^*.
\label{eq:D_eps_0}
\end{equation}
Besides, since $\big| |x|^p - |y|^p \big| \leq C\big( |x|^{p-1} + |y|^{p-1} \big) |x-y|$ for all $x,y \in \RR^d$, we get 
\begin{equation}
\begin{aligned}
\big|D_{\varepsilon} - D_{\varepsilon}^{\per}\big| & \leq \int_{\Omega} \big|\widetilde{a}\big(\frac{\cdot}{\varepsilon}\big)\big|\big|p\big(\frac{\cdot}{\varepsilon},M_{\varepsilon} \nabla u^*\big)\big|^{p} \\ &+ C \int_{\Omega} \big|a^{\per}\big(\frac{\cdot}{\varepsilon}\big)\big|^p\left(\big|p\big(\frac{\cdot}{\varepsilon},M_{\varepsilon} \nabla u^* \big)\big|^{p-1} +\big|p^{\per}\big(\frac{\cdot}{\varepsilon},M_{\varepsilon} \nabla u^* \big)\big|^{p-1}\right)\big|\nabla \widetilde{w_{M_{\varepsilon}\nabla u^*}}\big(\frac{\cdot}{\varepsilon} \big) \big|.
\end{aligned}
\label{eq_D_(eps)}
\end{equation}
We show that each term of the RHS of~\eqref{eq_D_(eps)} vanishes as $\varepsilon \longrightarrow 0$. We use Theorem~\ref{th:th_nonlin} (ii) and~\eqref{eq:intermediaire_cvremaindeer} with $B =+\infty$, which imply that there exists a constant $C > 0$ independent of $\varepsilon$ such that
\begin{equation}
\int_{\Omega}\big|p\big(\frac{\cdot}{\varepsilon},M_{\varepsilon} \nabla u^* \big)\big|^{p} +\big|p^{\per}\big(\frac{\cdot}{\varepsilon},M_{\varepsilon} \nabla u^* \big)\big|^{p} \leq C \|\nabla u^*\|_{L^p(\Omega)}^p.
\label{eq:D_(eps)_2}
\end{equation}
With the H\"{o}lder inequality and Lemma~\ref{lem:cv_remainder}, we prove that the second term of the RHS of \eqref{eq_D_(eps)} tends to zero as $\varepsilon \longrightarrow 0$. As for the first term, we write that
\begin{equation}
\begin{aligned}
\int_{\Omega} & \big|\widetilde{a}\big(\frac{\cdot}{\varepsilon}\big)\big|\big|p\big(\frac{\cdot}{\varepsilon},M_{\varepsilon} \nabla u^*\big)\big|^{p}  \leq C\int_{\Omega} \big|\widetilde{a}\big(\frac{\cdot}{\varepsilon}\big)\big|\big|p\big(\frac{\cdot}{\varepsilon},\Phi\big)\big|^{p} + C\int_{\Omega} \big|\widetilde{a}\big(\frac{\cdot}{\varepsilon}\big)\big|\big|p\big(\frac{\cdot}{\varepsilon},M_{\varepsilon} \nabla u^*\big) - p\big(\frac{\cdot}{\varepsilon},\Phi\big)\big|^{p} \\
& \underset{\eqref{eq:estim_holder_per},\eqref{eq:D_(eps)_2}}{\leq} C\underbrace{\big\|\widetilde{a}\big(\frac{\cdot}{\varepsilon}\big)\|_{L^{1}(\Omega)}}_{= O(\varepsilon^{d/p'})} \|\Phi\|_{L^{\infty}(\R^d)}^p + C\|\widetilde{a}\|_{L^{\infty}(\R^d)} \big\|p\big(\frac{\cdot}{\varepsilon},M_{\varepsilon} \nabla u^*\big) - p\big(\frac{\cdot}{\varepsilon},\Phi\big)\big\|_{L^p(\R^d)}^p,
\end{aligned}
\label{eq:D_(eps)_3}
\end{equation}
where we used that $\widetilde{a} \in L^{\infty}(\R^d)$ and the bound $|p(y,\xi)| \leq C|\xi|$ where $C>0$ is independent of $y$ and $\xi$. Collecting \eqref{eq:D_eps_0}, \eqref{eq_D_(eps)}, \eqref{eq:D_(eps)_3} and Lemma~\ref{lem:equiv_lem_3.5}, we have proved that 
\begin{equation}
\limsup_{\varepsilon \rightarrow 0} \left|D_{\varepsilon} -  \int_{\Omega} a^*(\nabla u^*)\cdot\nabla u^* \right| = O(\delta^{p\beta}).
\label{eq:D_(eps)}
\end{equation}
Finally, collecting~\eqref{eq:A_(eps)}, \eqref{eq:B_(eps)}, \eqref{eq:C_(eps)},~\eqref{eq:D_(eps)} and~\eqref{eq:R_(eps)}, we obtain that
\begin{equation}
\limsup_{\varepsilon \rightarrow 0}\left\| R_{\varepsilon} \right\|_{L^p(\Omega)}^p \leq \int_{\Omega}\left| a^*(\nabla u^*)\cdot \left(\nabla u^* - \Phi\right) \right| + \int_{\Omega}\left|\left( a^*(\nabla u^*) - a^*(\Phi)\right)\cdot\nabla u^* \right| + O(\delta^{\beta}).
\end{equation} 
Using the following property of $a^*$, see \cite[Remark~1.3]{dal1990correctors},
$$\big\| a^*(\nabla u^*) - a^*(\Phi) \big\|_{L^{p'}(\Omega)} \leq C \big[ \|\nabla u^* \|_{L^p(\Omega)}^{p-2} + \|\Phi \|_{L^p(\Omega)}^{p-2} \big]\|\nabla u^* -\Phi\|_{L^p(\Omega)},$$
 we conclude that $\limsup_{\varepsilon \rightarrow 0}\| R_{\varepsilon} \|_{L^p(\Omega)}^p = O(\delta^{\beta})$ where the $O$ is independent of $\delta$. Since this is true for all $\delta > 0$, we conclude that $\| R_{\varepsilon} \|_{L^p(\Omega)} \underset{\varepsilon \rightarrow 0}{\longrightarrow} 0$.

\begin{remarque}  Using the same strategy as above, it is straightforward to show that Theorem~\ref{th:homog_qualititative} holds with the operator $M_{\varepsilon}$ replaced by $M_{\varepsilon^{\nu}}$, $0 <\nu < 1$. In this case, the continuity of the application $\xi \mapsto\nabla  \widetilde{w_{\xi}}$ is not needed and we only use that $\nabla \widetilde{w_{\xi}} \in L^p(\R^d)$.
\end{remarque}

\section[Continuity]{Continuity of $\xi \mapsto\nabla w_{\xi}$: proof of Theorem~\ref{th:cont}}
\label{sect:weakcont}

\subsection{Preliminary Lemmas}

We begin this section with the following lemma that defines weak solution of PDEs of the form~\eqref{eq:PDE_faible}:

\begin{lemme}
Let $\xi \in \R^d$, a coefficient $a$ satisfying Assumption \textbf{(A1)} and $h \in L^{p'}(\R^d)^d$. Assume that Assumption~\textbf{(A4)'} is satisfied. Let $ v \in W_{\xi+\nabla w_{\xi}^{\per}}$ be solution in the distribution sense to the following PDE:
\begin{equation}
-\div \ a \big[( \xi+\nabla w_{\xi}^{\per} + \nabla v)|\xi+\nabla w_{\xi}^{\per} +  \nabla v|^{p-2} - (\xi+\nabla w_{\xi}^{\per}) |\xi+\nabla w_{\xi}^{\per}|^{p-2} \big] = \mathrm{div}(h).
\label{eq:PDE_faible}
\end{equation}
Then $\nabla v$ solves~\eqref{eq:PDE_faible} in the weak sense of Definition~\ref{def:def}: for all $w \in W_{\xi+\nabla w_{\xi}^{\per}}$,
\begin{equation}
\int_{\R^d} a \big[(\xi+\nabla w_{\xi}^{\per} + \nabla v)|\xi+\nabla w_{\xi}^{\per} + \nabla v|^{p-2}   - (\xi+\nabla w_{\xi}^{\per})|\xi+\nabla w_{\xi}^{\per}|^{p-2} \big]\cdot \nabla w = - \int_{\R^d} h \cdot \nabla w.
\label{eq:formfaible}
\end{equation}
\label{lem:formfaible}
\end{lemme}

\begin{proof}
[Proof of Lemma~\ref{lem:formfaible}] We define $u := \xi + \nabla w_{\xi}^{\per}$ in the proof. We fix $w \in W_u$. In the following, $\chi$ will denote a smooth and compactly supported function with support in $Q(0,1)$ such that $\chi = 1$ in $Q(0,\frac{1}{2})$. We fix $R > 0$ and we introduce the function
$$\Phi_R := \left( w - \fint_{Q_{R}\setminus Q_{R/2}} w \right) \chi\left(\frac{\cdot}{R} \right).$$
By the Poincar\'e-Wirtinger inequality, we have that $\Phi_R\in W^{1,p}_0(Q_{R})$.
By \eqref{eq:useful_ineq_3}, we have the bound
\begin{equation}
\begin{aligned}
a \left|(u + \nabla v)|u + \nabla v|^{p-2} - u|u|^{p-2} \right| & \leq \lambda C \big[ |u+\nabla v|^{p-2} + |u|^{p-2} \big]|\nabla v| \\& \leq \lambda C(p) \left[ |u|^{p-2}|\nabla v| + |\nabla v|^{p-1} \right],
\end{aligned}
\label{eq:RHS2}
\end{equation}
where we have used the inequality $(b_1+b_2)^{p-2} \leq C(p)(b_1^{p-2} + b_2^{p-2})$ for $b_1,b_2 \geq 0$.
Thus, since $\nabla v \in L^p(Q_{R})$ and $p = p'(p-1)$,
$$\mathrm{div} \ a \left[(u + \nabla v)|u + \nabla v|^{p-2} - u|u|^{p-2} \right] \in W^{-1,p'}(Q_{R}) \quad \text{and} \quad \text{div}(h) \in W^{-1,p'}(Q_{R}).$$ 
Consequently, we can test \eqref{eq:PDE_faible} against $\Phi_R$ and obtain, after expanding $\nabla\Phi_R$,
\begin{equation}
\begin{aligned}
\int_{Q_{R}} & a \left[(u + \nabla v)|u + \nabla v|^{p-2} - u|u|^{p-2} \right]\cdot \chi\left(\frac{\cdot}{R} \right)\nabla w \\ & + \frac{1}{R}\int_{Q_{R}} a \left[(u + \nabla v)|u + \nabla v|^{p-2} - u|u|^{p-2} \right]\cdot \left(w - \fint_{Q_{R}\setminus Q_{R/2}} w \right)\nabla \chi \left(\frac{\cdot}{R} \right) = \\
&-  \int_{Q_{R}} h \cdot \chi\left(\frac{\cdot}{R} \right)\nabla w - \frac{1}{R}\int_{Q_{R}} h \cdot \left(w - \fint_{Q_{R} \setminus Q_{R/2}} w \right)\nabla \chi \left(\frac{\cdot}{R} \right).
\end{aligned}
\label{eq:Lemme_4.5}
\end{equation} 
We now recall the following Poincar\'e-Wirtinger inequality:
\begin{equation}
\begin{aligned}
&\left\| w - \fint_{Q_{R} \setminus Q_{R/2}} w \right\|_{L^p(Q_{R} \setminus Q_{R/2})} \leq CR \| \nabla w \|_{L^p(Q_{R } \setminus Q_{R/2})}
\end{aligned}
\label{eq:poinc}
\end{equation}
which is simply a rescaled version of the $L^p$ inequality on $Q\setminus Q_{1/2}$. Besides, thanks to Assumption~\textbf{(A4)'} (and its rescaled version), we have that
$$\left\||u|^{\frac{p-2}{2}} \left( w - \fint_{Q_{R}\setminus Q_{R/2}} w \right)\right\|_{L^2(Q_{R} \setminus Q_{R/2})} \leq CR \| |u|^{\frac{p-2}{2}}\nabla w \|_{L^2(Q_{R}\setminus Q_{R/2})}.$$
This yields, together with~\eqref{eq:RHS2}, H\"{o}lder inequality and the inclusion $\text{supp} (\nabla \chi) \subset Q_{R}\setminus Q_{R/2}$, that
\begin{equation}
\begin{aligned}
\bigg| \int_{Q_{R} \setminus Q_{R/2}} a &\left[(u + \nabla v)|u + \nabla v|^{p-2} - u|u|^{p-2} \right]\cdot \left(w - \fint_{Q_{R} \setminus Q_{R/2}} w \right)\nabla \chi \left(\frac{\cdot}{R} \right) \bigg| \\ & \leq \lambda CR \|\nabla \chi\|_{L^{\infty}} \bigg\{ \big\|\nabla v\big\|_{L^p(Q_{R} \setminus Q_{R/2})}^{p-1} \big\|\nabla w\big\|_{L^p(Q_{R} \setminus Q_{R/2})} \\& \hspace{4cm} + \big\| \left|u\right|^{\frac{p-2}{2}}\nabla v \big\|_{L^2(Q_{R} \setminus Q_{R/2})}  \left\| \left|u\right|^{\frac{p-2}{2}}\nabla w \right\|_{L^2(Q_{R} \setminus Q_{R/2})}  \bigg\}
\end{aligned}
\label{eq:RHS1}
\end{equation}
and 
\begin{equation}
\left|\int_{Q_R} h \cdot \left(w - \fint_{Q_{R}\setminus Q_{R/2}} w \right)\nabla \chi \left(\frac{\cdot}{R} \right) \right| \leq C R \| \nabla \chi \|_{L^{\infty}} \|h\|_{L^{p'}(Q_{R} \setminus Q_{R/2})} \| \nabla w \|_{L^{p}(Q_{R} \setminus Q_{R/2})}.
\label{eq:RHS3}
\end{equation}
Collecting~\eqref{eq:Lemme_4.5},~\eqref{eq:RHS1} and~\eqref{eq:RHS3} and recalling that $v,w \in W_u$, we have that
\begin{equation}
\left| \int_{\R^d} \chi\left(\frac{\cdot}{R}\right) \left\{ a \left[(u + \nabla v)|u + \nabla v|^{p-2} - u|u|^{p-2} \right] + h \right\}\cdot \nabla w \right| \underset{R \rightarrow + \infty}{\longrightarrow} 0.
\end{equation}
On the other hand, by the dominated convergence Theorem, again since $v,w \in W_u$, we have that
\begin{equation}
\begin{aligned}
\int_{\R^d} \chi\left(\frac{\cdot}{R}\right) & \left\{ a \left[(u + \nabla v)|u + \nabla v|^{p-2} - u|u|^{p-2} \right] + h \right\}\cdot \nabla w \\ & \underset{R \rightarrow + \infty}{\longrightarrow}  \int_{\R^d}\left\{ a \left[(u + \nabla v)|u + \nabla v|^{p-2} - u|u|^{p-2} \right] + h \right\}\cdot \nabla w.
\end{aligned}
\end{equation}
Thus~\eqref{eq:formfaible} is satisfied.
\end{proof}

The next lemma allows to pass to the limit in PDEs of the form~\eqref{eq:PDE_faible}.

\begin{lemme}
Let  $(\nabla\phi_n)_{n \in \NN} \subset L^{\infty}(\R^d)^d$, $(a_n)_{n \in \NN} \subset L^{\infty}(\R^d)$, $(h_n)_{n \in \NN} \subset L^{p'}(\R^d)^d$ and $(v_n)_{n \in \NN} \subset V$ (see~\eqref{eq:V}), such that $v_n \in W_{\nabla \phi_n}$ for all $n \in \NN$. We assume that, for all $n \in \NN$:
\begin{enumerate}
\item The coefficient $a_n$ satisfies Assumption \textbf{(A1)} with $\lambda$ uniformly bounded in $n$.
\item The function $\nabla v_n$ is solution, in the distribution sense, to 
\begin{equation}
- \mathrm{div}\ a_n \left[(\nabla \phi_n + \nabla v_n)|\nabla \phi_n + \nabla v_n|^{p-2} - \nabla \phi_n |\nabla \phi_n|^{p-2} \right] = \mathrm{div}(h_n).
\label{eq:passaglimite_eq_n}
\end{equation}
\end{enumerate}
We also assume the following convergences:
\begin{enumerate}
[label=(\roman*)]
\item $\nabla \phi_n \underset{n \rightarrow + \infty}{\longrightarrow} \nabla \phi$ in $L^{\infty}(\R^d)$;
\item $a_n \underset{n \rightarrow + \infty}{\longrightarrow} a$ in $L^{\infty}_{\text{loc}}(\R^d)$;
\item $\nabla v_n \underset{n \rightarrow + \infty}{\relbar\joinrel\rightharpoonup} \nabla v$ in~$L^p(\R^d)$ and $|\nabla\phi_n|^{\frac{p-2}{2}} \nabla v_n \underset{n \rightarrow + \infty}{\relbar\joinrel\rightharpoonup} |\nabla \phi|^{\frac{p-2}{2}} \nabla v$ in~$L^2(\R^d)$ which $v \in W_{\nabla \phi}$;
\item $h_n \underset{n \rightarrow + \infty}{\longrightarrow} h$ in $L^{p'}_{\text{loc}}(\R^d)$ with $h \in L^{p'}(\R^d)$.
\end{enumerate}
Then $\nabla v_n \underset{n \rightarrow + \infty}{\longrightarrow} \nabla v$ in $L^p_{\text{loc}}(\R^d)$ and $\nabla v$ is solution in the distribution sense to
\begin{equation}
-\mathrm{div} \ a \left[ (\nabla \phi + \nabla v)|\nabla \phi + \nabla v|^{p-2} - \nabla \phi |\nabla \phi|^{p-2} \right] = \mathrm{div}(h).
\label{eq:passaglimite_eq}
\end{equation} 
\label{lem:passaglimite}
\end{lemme}

\begin{remarque}
If all other assumptions are satisfied, the assumption $|\nabla\phi_n|^{\frac{p-2}{2}} \nabla v_n \relbar\joinrel\rightharpoonup|\nabla \phi|^{\frac{p-2}{2}}\nabla v$ in $L^2(\R^d)$ can be weakened in $|\nabla\phi_n|^{\frac{p-2}{2}} \nabla v_n$ is $L^2-$weakly convergent. Indeed, following the proof of Lemma~\ref{lem:elementary_cor}(iii), we can prove that if $\nabla v_n \relbar\joinrel\rightharpoonup \nabla v$ in $L^p(\R^d)$, $\nabla \phi_n \longrightarrow \nabla \phi$ in $L^{\infty}(\R^d)$ and $|\nabla\phi_n|^{\frac{p-2}{2}} \nabla v_n$ is $L^2-$weakly convergent, then $|\nabla\phi_n|^{\frac{p-2}{2}} \nabla v_n \relbar\joinrel\rightharpoonup |\nabla \phi|^{\frac{p-2}{2}} \nabla v$ in $L^2(\R^d)$. In particular, we have that $v \in W_{\nabla\phi}$.
\label{re:hyp}
\end{remarque}

\begin{proof}
[Proof of Lemma~\ref{lem:passaglimite}] We fix two bounded smooth domains $B, B'$ such that $B \subset \subset B'$. Let $\chi \in \mathcal{D}(B')$ such that $\chi = 1$~on~$B$ and $0 \leq \chi \leq 1$ in $B'$. We introduce the function
$$\Psi_n :=\underset{:= \Psi_n^1}{\underbrace{ \left\{ \phi_n + v_n - (v + \phi) - \fint_{B'} \big[ \phi_n + v_n - (v + \phi) \big] \right\}}}\chi \in W^{1,p}_0(B').$$
We immediately check that, up to extracting a subsequence, we have by the Rellich compactness Theorem and \textit{(i), (iii)} that 
\begin{equation}
\Psi_n^1 \underset{n \rightarrow +\infty}{\longrightarrow} 0 \ \text{in} \ L^p(B') \quad \text{and} \quad \nabla \Psi_n^1 \underset{n \rightarrow +\infty}{\relbar\joinrel\rightharpoonup} 0 \ \text{in} \ L^p(B').
\label{eq:limit_Psi_n_1}
\end{equation}
Thus, 
\begin{equation}
\Psi_n \underset{n \rightarrow +\infty}{\longrightarrow} 0 \ \text{in} \ L^p(B') \quad \text{and} \quad \nabla \Psi_n \underset{n \rightarrow +\infty}{\relbar\joinrel\rightharpoonup} 0 \ \text{in} \ L^p(B').
\label{eq:limit_Psi_n}
\end{equation}
Since~$\Psi_n \in W^{1,p}_0(B')$, we can test $\Psi_n$ against~\eqref{eq:passaglimite_eq_n}. We re-organize the terms and get that
\begin{equation}
\begin{aligned}
\int_{B'} & a (\nabla \phi_n + \nabla v_n)|\nabla \phi_n + \nabla v_n|^{p-2}\cdot(\nabla \phi_n + \nabla v_n - \nabla \phi - \nabla v ) \chi \\ 
& = -\int_{B'}  a (\nabla \phi_n + \nabla v_n)|\nabla \phi_n + \nabla v_n|^{p-2}\cdot \Psi_n^1 \nabla\chi + \int_{B'} (a - a_n) (\nabla \phi_n + \nabla v_n)|\nabla \phi_n + \nabla v_n|^{p-2}\cdot \nabla \Psi_n\\
& \quad - \int_{B'} h_n\cdot \nabla \Psi_n + \int_{B'} a_n |\nabla \phi_n|^{p-2}\nabla \phi_n \cdot \nabla \Psi_n \\
& = - A_n + B_n - C_n + D_n.
\end{aligned}
\label{eq:passag_limit_preuve}
\end{equation}
We study each term separetely. The term $A_n$ vanishes when $n \longrightarrow + \infty$ since
$$
\big| (\nabla \phi_n + \nabla v_n)|\nabla \phi_n + \nabla v_n|^{p-2}\big| \leq C_p \big[|\nabla \phi_n|^{p-1} + |\nabla v_n |^{p-1}\big],
$$
which is bounded in $L^{p'}(B')$, uniformly with respect to $n$ by \textit{(i)} and \textit{(iii)} and \eqref{eq:limit_Psi_n_1}.
The term $B_n$ vanishes as $n \longrightarrow + \infty$ by \eqref{eq:limit_Psi_n} and since, by \textit{(i)}, \textit{(ii)} and \textit{(iii)}:
$$(a - a_n) (\nabla \phi_n + \nabla v_n)|\nabla \phi_n + \nabla v_n|^{p-2} \underset{n \rightarrow +\infty}{\longrightarrow} 0 \ \text{in} \ L^{p'}(B').$$
The term $C_n$ vanishes by~\eqref{eq:limit_Psi_n} and the $L^p_{\text{loc}}-$strong convergence of the sequence $(h_n)_{n \in \NN}$. The term $D_n$ vanishes by~\eqref{eq:limit_Psi_n} and the convergence of $a_n |\nabla \phi_n|^{p-2}\nabla \phi_n$ to $a |\nabla \phi|^{p-2} \nabla \phi$ in $L^{\infty}(B')$. We have proved that 
\begin{equation}
\int_{B'}  a (\nabla \phi_n + \nabla v_n)|\nabla \phi_n + \nabla v_n|^{p-2}\cdot(\nabla \phi_n + \nabla v_n - \nabla \phi - \nabla v ) \chi \underset{n \rightarrow +\infty}{\longrightarrow} 0.
\label{eq:passage_limie_tepreuve_cor}
\end{equation}
However, since $(\nabla \phi + \nabla v)|\nabla \phi + \nabla v|^{p-2} \in L^{p'}(B')$ and because of~\eqref{eq:limit_Psi_n_1} and \textit{(i)}, we also have that
\begin{equation}
\int_{B'}  a (\nabla \phi + \nabla v)|\nabla \phi + \nabla v|^{p-2}\cdot(\nabla \phi_n + \nabla v_n - \nabla \phi - \nabla v ) \chi \underset{n \rightarrow +\infty}{\longrightarrow} 0.
\label{eq:passage_limie_tepreuve_cor2}
\end{equation}
The difference between~\eqref{eq:passage_limie_tepreuve_cor} and~\eqref{eq:passage_limie_tepreuve_cor2} gives that
\begin{equation}
\int_{B'}  a \left[(\nabla \phi_n + \nabla v_n)|\nabla \phi_n + \nabla v_n|^{p-2} - (\nabla \phi + \nabla v)|\nabla \phi + \nabla v|^{p-2}\right]\cdot(\nabla \phi_n + \nabla v_n - \nabla \phi - \nabla v ) \chi \underset{n \rightarrow +\infty}{\longrightarrow} 0.
\label{eq:passage_limie_tepreuve_cor3}
\end{equation}
Using~\eqref{eq:passage_limie_tepreuve_cor3}, \eqref{eq:useful_ineq_1}, that $\chi \geq 0$ and $\chi = 1$ in $B$, we get
\begin{equation}
\int_B \big| \nabla \phi_n + \nabla v_n - \nabla \phi  - \nabla v\big|^p \underset{n \rightarrow +\infty}{\longrightarrow} 0.
\label{eq:passag_limit_cvloc}
\end{equation}
By \textit{(i)}, we obtain that $\nabla v_n \underset{n \rightarrow +\infty}{\longrightarrow} \nabla v$ in $L^p(B)$ up to a subsequence. We easily show that the convergence in fact holds for the whole sequence. We consequently get the $L^p_{\text{loc}}(\R^d)-$convergence since $B$ is arbitrary. 

\medskip

We now pass to the limit $n \rightarrow + \infty$ in~\eqref{eq:passaglimite_eq_n}. Let $\Psi \in \mathcal{D}(\R^d)$, we test~\eqref{eq:passaglimite_eq_n} against $\Psi$. By \textit{(iv)}, it is clear that
\begin{equation}
\int_{\R^d} h_n \cdot \nabla \Psi \underset{n \rightarrow +\infty}{\longrightarrow} \int_{\R^d} h \cdot \nabla \Psi.
\label{eq:lem_passag_limit_ccl1}
\end{equation}
Besides, by~\eqref{eq:useful_ineq_3}, \textit{(i)}, \textit{(ii)} and the $L^p_{\text{loc}}(\R^d)-$convergence of $\nabla v_n$, we have that 
$$
\begin{aligned}
a_n \big\{(\nabla \phi_n + \nabla v_n&)|\nabla \phi_n + \nabla v_n|^{p-2} - \nabla \phi_n|\nabla \phi_n|^{p-2} \big\} \\ & \underset{n \rightarrow +\infty}{\longrightarrow} a \left\{ (\nabla \phi + \nabla v)|\nabla \phi + \nabla v|^{p-2}  - \nabla \phi |\nabla \phi|^{p-2} \right\}\quad \text{in} \quad L^{p'}_{\text{loc}}(\R^d).
\end{aligned}$$
This shows that
$$\int_{\R^d} a \left\{ (\nabla \phi + \nabla v)|\nabla \phi + \nabla v|^{p-2}  - \nabla \phi |\nabla \phi|^{p-2} \right\} \cdot \nabla \Psi= - \int_{\R^d} h \cdot \nabla \Psi,$$
and concludes the proof of the Lemma~\ref{lem:passaglimite}.
\end{proof}

\subsection{Proof of Theorem~\ref{th:cont}}

Let $(\xi_n)_{n \in \NN} \subset \R^d$ such that $\xi_n \underset{n \rightarrow + \infty}{\longrightarrow} \xi$ for $\xi \in \R^d$. We aim at showing that $\nabla w_{\xi_n} \underset{n \rightarrow +\infty}{\longrightarrow} \nabla w_{\xi}$ in $L^p_{\text{unif}}(\R^d)$. By~Proposition~\ref{prop:periodique} (iii), it is sufficient to show that $\nabla \widetilde{w_{\xi_n}} \underset{n \rightarrow +\infty}{\longrightarrow} \nabla \widetilde{w_{\xi}}$ in $L^p_{\text{unif}}(\R^d)$.

\medskip

\underline{Step 1.} We have that $\nabla \widetilde{w_{\xi_n}} \underset{n \rightarrow +\infty}{\relbar\joinrel\rightharpoonup} \nabla \widetilde{w_{\xi}}$ in $L^p(\R^d)$. Indeed, by~\eqref{eq:useful_ineq_1},~\eqref{eq:useful_ineq_2} and the form of $h$, see~\eqref{eq:f_2}, we have the following a priori estimate: there exists a constant $C=C(d,p,a,a^{\per}) > 0$ such that for all $n \in \NN$, 
\begin{equation}
\|\widetilde{w_{\xi_n}} \|_{W_{\xi_n + \nabla w_{\xi_n}^{\per}}} \leq C\big(|\xi_n|+|\xi_n|^{p/2} \big).
\label{eq:bound}
\end{equation} 
In particular, there exists $v \in V$ (see~\eqref{eq:V} for the definition of $V$) and $w \in L^2(\R^d)$ such that 
\begin{equation}
\nabla \widetilde{w_{\xi_n}} \underset{n \rightarrow +\infty}{\relbar\joinrel\rightharpoonup} \nabla v \ \text{in} \ L^p(\R^d) \quad \text{and} \quad |\xi_n + \nabla w_{\xi_n}^{\per}|^{\frac{p-2}{2}}\nabla \widetilde{w_{\xi_n}} \underset{n \rightarrow +\infty}{\relbar\joinrel\rightharpoonup} w \ \text{in} \ L^2(\R^d).
\label{eq:th2_8_conv}
\end{equation}
Taking into account Remark~\ref{re:hyp}, we have that $v \in W_{\xi + \nabla w_{\xi}^{\per}}$ and $$|\xi_n + \nabla w_{\xi_n}^{\per}|^{\frac{p-2}{2}}\nabla \widetilde{w_{\xi_n}} \underset{n \rightarrow +\infty}{\relbar\joinrel\rightharpoonup} |\xi + \nabla w_{\xi}^{\per}|^{\frac{p-2}{2}} \nabla v \quad \text{in} \quad L^2(\R^d).$$
We apply Lemma~\ref{lem:passaglimite} with 
$$\begin{cases}
\begin{aligned}
a_n & = a, \quad v_n = \widetilde{w_{\xi_n}}\\
\phi_n & = \xi_n \cdot x + w_{\xi_n}^{\per}, \quad \phi = \xi \cdot x + w_{\xi}^{\per} \\
h_n & = \widetilde{a}|\xi_n + \nabla w_{\xi_n}^{\per}|^{p-2}(\xi_n + \nabla w_{\xi_n}^{\per}), \quad h = \widetilde{a}|\xi + \nabla w_{\xi}^{\per}|^{p-2}(\xi + \nabla w_{\xi}^{\per}).
\end{aligned}
\end{cases}$$
The required convergences follow from~\eqref{eq:th2_8_conv} and Proposition~\ref{prop:periodique} (ii) and (iv).
We get that $\nabla \widetilde{w_{\xi_n}}$ converges when $n \longrightarrow +\infty$ to $\nabla v$ in $L^{p}_{\text{loc}}(\R^d)$ and that $\nabla v$ solves \eqref{eq:th2.2_gen2} in the distribution sense. Thus, $\nabla v$ solves~\eqref{eq:th2.2_gen2} in the weak sense in $W_{\xi+\nabla w_{\xi}^{\per}}$ (see Definition~\ref{def:def}) by Lemma~\ref{lem:formfaible}, with $h$ given by~\eqref{eq:f_2}. In addition, by Theorem~\ref{th:existencecor}, the solution of this PDE is unique in $W_{\xi + \nabla w_{\xi}^{\per}}$. Thus $\nabla v  = \nabla \widetilde{w_{\xi}}$ and this concludes Step 1 since the sequence $(\nabla \widetilde{w_{\xi_n}})_{n \in \NN}$ has one possible limit.

\medskip

\underline{Step 2.} Suppose by contradiction that $\nabla \widetilde{w_{\xi_n}}$ does not converge to $\nabla \widetilde{w_{\xi}}$ in $L^p_{\text{unif}}(\R^d)$ when $n \longrightarrow +\infty$. Then there exists $\delta > 0$, a subsequence of $(\xi_n)_{n \in \NN}$ that we de not relabel and a sequence $(x_n)_{n \in \NN}$ such that 
\begin{equation}
\forall n \in \NN, \quad \| \nabla \widetilde{w_{\xi_n}} - \nabla \widetilde{w_{\xi}} \|_{L^p(B(x_n,1))} \geq \delta.
\label{eq:absurde}
\end{equation}
Up to another extraction, we can suppose that $x_n \underset{n \rightarrow +\infty}{\longrightarrow} x$ in $\mathbb{T}^d$, where $\mathbb{T}^d$ denotes the $d-$dimensional torus.
Since by Step 1, we know that $\nabla \widetilde{w_{\xi_n}} \underset{n \rightarrow +\infty}{\longrightarrow} \nabla \widetilde{w_{\xi}}$ in $L^p_{\text{loc}}(\R^d)$, we necessarily have that, up to extracting a subsequence, $|x_n| \underset{n \rightarrow +\infty}{\longrightarrow} +\infty$ in $\RR^d$. We introduce the shifted functions
\begin{equation}
\widehat{w_{n}^1} := \widetilde{w_{\xi_n}}(\cdot + x_n) \quad \text{and} \quad \widehat{w_n^2} :=  \widetilde{w_{\xi}}(\cdot + x_n),
\label{eq:shifted}
\end{equation}
\begin{equation}
\widehat{\phi_{n}^1} := \xi_n \cdot x + w_{\xi_n}^{\per}(\cdot + x_n) \quad \text{and} \quad \widehat{\phi_n^2} :=  \xi \cdot x + w_{\xi}^{\per}(\cdot + x_n).
\label{eq:shifted2}
\end{equation}
In particular, \eqref{eq:absurde} gives
\begin{equation}
\forall n \in \NN, \quad \| \nabla \widehat{w_n^1} - \nabla\widehat{w_n^2} \|_{L^p(B(0,1))} \geq \delta.
\label{eq:absurde2}
\end{equation}
We show in the sequel that, up to a subsequence, for $i = 1,2$,
\begin{equation}
\nabla \widehat{w_n^i} \underset{n \rightarrow + \infty}{\longrightarrow} \nabla \widehat{v^i} \ \text{in} \ L^p_{\text{loc}}(\R^d)  \ \quad \text{and} \quad  \nabla \widehat{v^i} = 0 \ \text{a.e}. 
\label{eq:step3}
\end{equation} 
In particular, passing to the limit $n \longrightarrow +\infty$ in~\eqref{eq:absurde2} will lead to a contradiction.

\medskip

\underline{Step 3.} Proof of~\eqref{eq:step3}. We prove \eqref{eq:step3} for $i=1$, the proof is standard for $i=2$. We have that $\nabla\widehat{w_n^1}$ solves in the distribution sense the PDE
\begin{equation}
-\text{div} a(\cdot + x_n)\big[ (\nabla \widehat{\phi_n^1} + \nabla \widehat{w_n^1})|\nabla \widehat{\phi_n^1} + \nabla \widehat{w_n^1}|^{p-2} - \nabla \widehat{\phi_n^1}|\nabla \widehat{\phi_n^1}|^{p-2} \big] = \text{div} \left(\widetilde{a}(\cdot + x_n)\nabla \widehat{\phi_n^1}|\nabla \widehat{\phi_n^1}|^{p-2}\right)
\end{equation}
and that
$\widehat{w_n^1} \in W_{\nabla \widehat{\phi_n^1}}.$ Since $$\|\widehat{w_n^1}\|_{W_{\nabla \widehat{\phi_n^1}}} = \|\widetilde{w_{\xi_n}} \|_{W_{\xi_n + \nabla w_{\xi_n}^{\per}}},$$
we get because of~\eqref{eq:bound} that the sequences $\big(\|\nabla \widehat{w_{n}^1} \|_{L^p(\R^d)} \big)_{n \in \NN}$ and $\big( \| |\nabla\widehat{\phi_n^1} |^{\frac{p-2}{2}} \nabla \widehat{w_n^1} \|_{L^2(\R^d)} \big)_{n \in \NN}$ are uniformly bounded in $n$. Thus, up to extracting a subsequence,
\begin{equation}
\nabla \widehat{w_{n}^1} \underset{n \rightarrow +\infty}{\relbar\joinrel\rightharpoonup} \nabla \widehat{v^1} \ \text{in} \ L^p(\R^d) \quad \text{and} \quad \nabla\widehat{w_{n}^1} |\nabla \widehat{\phi_n^1}|^{\frac{p-2}{2}} \underset{n \rightarrow + \infty}{\relbar\joinrel\rightharpoonup} \widehat{w^1} \ \text{in} \ L^2(\R^d).
\label{eq:step2}
\end{equation}
We may apply Lemma~\ref{lem:passaglimite} with 
$$\begin{cases}
\begin{aligned}
a_n& = a(\cdot + x_n), \quad \nabla\phi_n = \nabla\widehat{\phi_n^1}, \\
\nabla v_n & = \nabla \widehat{w_n^1} \quad \text{and} \quad h_n = \widetilde{a}(\cdot + x_n)\nabla \widehat{\phi_n^1}|\nabla \widehat{\phi_n^1}|^{p-2}.
\end{aligned}
\end{cases}$$
We check the required convergences.
\begin{enumerate}[label=(\roman*)]
\item We have that $\nabla\widehat{\phi_n^1}$ converges in $L^{\infty}(\R^d)$ to $\nabla\widehat{\phi^1} : x \longmapsto \xi \cdot x + \nabla w_{\xi}^{\per}(\cdot + x)$. Indeed, by periodicity, it is enough to check that $\nabla \widehat{\phi_n^1} \underset{n \rightarrow +\infty}{\longrightarrow} \nabla \widehat{\phi^1}$ in $Q$. For all $y \in Q$,
$$
\begin{aligned}
|\nabla \widehat{\phi_n^1}(y) - \nabla \widehat{\phi^1}(y)|& \leq |\xi_n - \xi| + | \nabla w_{\xi_n}^{\per}(x_n + y) - \nabla w_{\xi_n}^{\per}(x + y)| + |\nabla w_{\xi_n}^{\per}(x + y) - \nabla w_{\xi}^{\per}(x + y)| \\
& \leq |\xi_n - \xi| + C|\xi_n||x_n - x|_{\mathbb{T}}^{\alpha} + \left\{|\xi_n|^{1-\gamma} + |\xi|^{1-\gamma} \right\}|\xi_n - \xi|^{\gamma} 
\end{aligned}$$
 where we used~Proposition~\ref{prop:periodique} (ii) together with~Proposition~\ref{prop:periodique} (iv) in the last inequality and $|\cdot|_{\mathbb{T}}$ denotes the euclidian norm on $\mathbb{T}^d$. This proves the result by convergence of the sequences $(\xi_n)_{n \in \NN}$ and $(x_n)_{n \in \NN}$.
\item We have $a_n = a^{\per}(\cdot + x_n) + \widetilde{a}(\cdot + x_n)$. Since $x_n \underset{n \rightarrow +\infty}{\longrightarrow} x$ in $\mathbb{T}^d$, we have by Assumption~\textbf{(A2)} that $a^{\per}(\cdot + x_n) \underset{n \rightarrow +\infty}{\longrightarrow} a^{\per}(\cdot + x)$ in $L^{\infty}(Q)$. Let $B$ be a bounded domain, then since $\widetilde{a} \in \mathcal{C}^{0,\alpha} \cap L^{p'}(\R^d)$, we have $\widetilde{a} \underset{|x| \rightarrow + \infty}{\longrightarrow} 0$. Thus $\widetilde{a}(\cdot + x_n) \underset{n \rightarrow +\infty}{\longrightarrow} 0$ in $L^{\infty}(B)$ and finally $a_n$ converges locally uniformly to $\widehat{a} := a^{\per}(\cdot + x)$ when $n \rightarrow + \infty$.
\item This is \eqref{eq:step2}.
\item By the same argument as in (ii), we have that $h_n \underset{n \rightarrow +\infty}{\longrightarrow} 0$ in $L^{p'}_{\text{loc}}(\R^d)$.
\end{enumerate}
We have proved that, up to exacting a subsequence, $\nabla \widehat{w_n^1} \underset{n \rightarrow +\infty}{\longrightarrow} \nabla \widehat{v^1}$ in $L^p_{\text{loc}}(\R^d)$ where $ \widehat{v^1} \in W_{\xi + \nabla w_{\xi}^{\per}(\cdot + x)}$ solves in the distribution sense the PDE
$$\begin{aligned}
- \text{div}\ a^{\per}(\cdot + x) \bigg[\big(\xi + \nabla w_{\xi}^{\per}(\cdot + x) + \nabla \widehat{v^1}\big)&\big|\xi + \nabla w_{\xi}^{\per}(\cdot + x) + \nabla \widehat{v^1}\big|^{p-2} \\ & - \big( \xi + \nabla w_{\xi}^{\per}(\cdot + x) \big) \big|\xi + \nabla w_{\xi}^{\per}(\cdot + x)\big|^{p-2}  \bigg]=0.
\end{aligned}$$
Introducing $\widehat{V^1} := \widehat{v^1}(\cdot - x)$, we get that $\widehat{V^1} \in W_{\xi+\nabla w_{\xi}^{\per}}$ and that $\nabla\widehat{V^1}$ solves in the distribution sense the PDE
\begin{equation}
-\text{div} a^{\per}\big[(\xi + \nabla w_{\xi}^{\per} + \nabla \widehat{V^1})|\xi + \nabla w_{\xi}^{\per} + \nabla \widehat{V^1}|^{p-2} - (\xi + \nabla w_{\xi}^{\per})|\xi + \nabla w_{\xi}^{\per}|^{p-2} \big] = 0.
\label{eq:lem7.1_2}
\end{equation}
 Applying Lemma~\ref{lem:formfaible} to~\eqref{eq:lem7.1_2} with $w := \widehat{V^1}$ gives now
$$\int_{\R^d} a^{\per}(\cdot + x) \left[(\xi + \nabla w_{\xi}^{\per} + \nabla \widehat{V^1})|\xi + \nabla w_{\xi}^{\per}  + \nabla \widehat{V^1}|^{p-2} - (\xi + \nabla w_{\xi}^{\per}) |\xi + \nabla w_{\xi}^{\per} |^{p-2}  \right] \cdot \nabla \widehat{V^1} = 0.$$
Applying~\eqref{eq:useful_ineq_1} allows to conclude that $\nabla \widehat{V^1} = 0$ in $\R^d$ and thus $\nabla \widehat{v^1} = 0$. This proves~\eqref{eq:step3} and concludes the proof of~Theorem~\ref{th:cont}.

\begin{remarque} From the above theorem, we can deduce that $$\begin{cases}
\begin{aligned}
\R^d &\longrightarrow L^{\infty}(\R^d)\\
\xi &\longmapsto \nabla \widetilde{w_{\xi}} \end{aligned} \end{cases}$$ is continuous. 

\medskip

 The continuity of $\xi \mapsto \nabla w_{\xi}^{\per}$ is due to Proposition~\ref{prop:periodique} (iv). We prove that $\xi \mapsto \nabla \widetilde{w_{\xi}}$ is continuous for the $L^{\infty}(\R^d)$ topology. By contradiction, suppose that there exists $\xi \in \R^d$, two sequences $(\xi_n)_{n \in \NN} \subset \R^d$ and $(x_n)_{n \in \NN} \subset \R^d$ and a $\delta > 0$ such that $\xi_n \underset{n \rightarrow +\infty}{\longrightarrow} \xi$ and
$$\forall n \in \NN, \quad \big| \nabla \widetilde{w_{\xi_n}}(x_n) - \nabla \widetilde{w_{\xi}}(x_n)\big| \geq \delta.$$
By~Theorem~\ref{th:th_nonlin} (ii), there exists $\eta$ independent of $n$ such that 
$$\forall n \in \NN, \quad \forall y \in B(x_n,\eta), \quad \big| \nabla \widetilde{w_{\xi_n}}(y) - \nabla \widetilde{w_{\xi}}(y)\big| \geq \frac{\delta}{2}.$$
Thus, for all $n \in \NN$, $$\|\nabla \widetilde{w_{\xi_n}} - \nabla \widetilde{w_{\xi}} \|_{L^p_{\text{unif}}(\R^d)} \geq |B(0,1)|^{1/p}\delta\eta^{d/p}>0.$$
which is a contradiction with Theorem~\ref{th:cont}.
\end{remarque}

\section*{Acknowledgments}

The author warmly thanks his PhD advisor Xavier Blanc for fruitful discussions and for reading many versions of this manuscript.

\appendix 

\section{Proof of Proposition~\ref{prop:periodique}}
\label{annexe:proof_prop2.1}

\paragraph*{Proof of \textit{(i)}.} This point is obvious by the form of~\eqref{eq:cornonlin} and the uniqueness of $w_{\xi}^{\per}$ in $W^{1,p}_{\per}(Q)/\R$. 

\paragraph*{Proof of \textit{(ii)}.} The first estimate follows for example from~\eqref{eq:min_per} and in particular:
$$\frac{1}{p} \int_Q a^{\per}(y)|\xi + \nabla w_{\xi}^{\per}|^p \leq \frac{1}{p} \int_Q a^{\per}|\xi|^p,$$
together with \textbf{(A1)}. The proof of the second estimate is exactly the same as the one of~Theorem~\ref{th:th_nonlin} (ii) (see Subsection~\ref{subsect:th}) with $a$ replaced by $a^{\per}$ and $w_{\xi}$ replaced by $w_{\xi}^{\per}$.

\paragraph*{Proof of \textit{(iii)}.} Let $\xi_1, \xi_2 \in \R^d$. Applying~\eqref{eq:cor_per} with $\phi = w_{\xi_1}^{\per} - w_{\xi_2}^{\per}$ with $\xi = \xi_i, i=1,2$ and making the difference between the two expressions gives:
\begin{equation}
\int_{Q} a^{\per}\big\{ (\xi_1 + \nabla w_{\xi_1}^{\per})|\xi_1 + \nabla w_{\xi_1}^{\per}|^{p-2}  - (\xi_2 + \nabla w_{\xi_2}^{\per})|\xi_2 + \nabla w_{\xi_2}^{\per}|^{p-2} \big\}\cdot \big\{\nabla w_{\xi_1}^{\per} - \nabla w_{\xi_2}^{\per} \big\} = 0.
\label{eq:prop2.2_(iii)} 
\end{equation} 
Thus, adding the term
\begin{equation}
\int_{Q} a^{\per}\big\{ (\xi_1 + \nabla w_{\xi_1}^{\per})|\xi_1 + \nabla w_{\xi_1}^{\per}|^{p-2}  - (\xi_2 + \nabla w_{\xi_2}^{\per})|\xi_2 + \nabla w_{\xi_2}^{\per}|^{p-2} \big\}\cdot \big\{\xi_1 - \xi_2 \big\} = 0
\label{eq:prop2.2_(iii)_2} 
\end{equation}
in the left and right-hand side of~\eqref{eq:prop2.2_(iii)}, applying~\eqref{eq:useful_ineq_1} on the left-hand side and~\eqref{eq:useful_ineq_3} on the right-hand side provides
$$
\begin{aligned}
c \int_Q \big|\xi_1 + \nabla w_{\xi_1}^{\per} & - (\xi_2 + \nabla w_{\xi_2}^{\per})\big|^p \\ &\leq C \int_Q \big(\big|\xi_1 + \nabla w_{\xi_1}^{\per}\big|^{p-2} + \big|\xi_2 + \nabla w_{\xi_2}^{\per}\big|^{p-2} \big) \big|\xi_1 + \nabla w_{\xi_1}^{\per} - (\xi_2 + \nabla w_{\xi_2}^{\per})\big| \big|\xi_1 - \xi_2 \big|.
\end{aligned}
$$
We apply the H\"{o}lder inequality on the right-hand side with exponents $p/(p-2)$, $p$ and $p$ find, using~\eqref{eq:estim_holder_per}:
$$c \left(\int_Q \big|\xi_1 + \nabla w_{\xi_1}^{\per} - (\xi_2 + \nabla w_{\xi_2}^{\per})\big|^p \right)^{1 - 1/p} \leq C \left[ |\xi_1|^{p-2} + |\xi_1|^{p-2} \right]|\xi_1-\xi_2|.$$
This yields~\eqref{eq:estim_holder_per} by taking the $1/(p-1)-$th power of the above inequality.

\paragraph*{Proof of \textit{(iv)}.} We argue by contradiction. Suppose that there exist three sequences $(x_n)_{n \in \NN} \subset Q$, $(\xi_n)_{n \in \NN} \subset \RR^d$ and $(\eta_n)_{n \in \NN} \subset \RR^d$ such that for all $n \in \NN$, $|\xi_n| = 1$, $0 < |\xi_n - \eta_n| \leq \frac{1}{2}$ and
$$\left| \nabla w_{\xi_n}^{\per}(x_n) - \nabla w_{\eta_n}^{\per}(x_n) \right| \geq n |\xi_n - \eta_n |^{\gamma}.$$
By~Proposition~\ref{prop:periodique} (ii), we have for $n$ large enough that
\begin{equation}
\forall y \in B(x_n,\delta_n), \quad \big| \nabla w_{\xi_n}^{\per}(y) - \nabla w_{\eta_n}^{\per}(y) \big| \geq \frac{n}{2} |\xi_n - \eta_n |^{\gamma},
\label{eq:rk45_2}
\end{equation}
where
$\delta_n := |\xi_n - \eta_n|^{\gamma/\alpha}$. Integrating \eqref{eq:rk45_2} over $B(x_n,\delta_n)$, we get that 
\begin{equation}
 \big\| \nabla w_{\xi_n}^{\per} - \nabla w_{\eta_n}^{\per} \big\|_{L^p(Q)}^p \geq |B_1| \left(\frac{n}{2}\right)^p |\xi_n - \eta_n |^{p\gamma} \delta_n^d = Cn^p|\xi_n - \xi_n|^{\gamma (p + d/\alpha)} = C n^p |\xi_n - \eta_n |^{\beta p}.
\label{eq:rk45_3}
\end{equation}
However, by~\eqref{eq:estim_holder_per},
$\big\| \nabla w_{\xi_n}^{\per} - \nabla w_{\eta_n}^{\per} \big\|_{L^p(Q)}^p \leq C|\xi_n - \eta_n|^{\beta p}.$
This is a contradiction with \eqref{eq:rk45_3} when taking $n \longrightarrow +\infty$. We have proved~\eqref{eq:continuité_Linfty} for $|\xi| = 1$ and $|\xi - \eta| \leq \frac{1}{2}$. The other cases are treated by homogeneity and with the help of \eqref{eq:estim_holder_per} as in~\eqref{eq:homogeneite_3}. 

\section{Some technical inequalities}
\label{sect:ineq}

We gather in this subsection some useful inequalities. We first have
\begin{equation}
(x|x|^{p-2} - y|y|^{p-2})\cdot(x-y) \geq c|x-y|^p,
\label{eq:useful_ineq_1}
\end{equation}
\begin{equation}
(x|x|^{p-2} - y|y|^{p-2})\cdot(x-y) \geq c\left[|x|^{p-2} + |y|^{p-2} \right]|x-y|^2, 
\label{eq:useful_ineq_2}
\end{equation}
\begin{equation}
\left|x|x|^{p-2} - y|y|^{p-2} \right| \leq C\left[|x|^{p-2} + |y|^{p-2} \right]|x-y|.
\label{eq:useful_ineq_3}
\end{equation}
In the above inequalities~\eqref{eq:useful_ineq_1}-\eqref{eq:useful_ineq_3}, $c$ and $C$ refer to universal constants that only depend on $p$. For a proof of these inequalities, we refer to \cite{iwaniec1983projections}. For $\xi, x \in \R^d$, we introduce the function
\begin{equation}
g_{\xi}(x) := |\xi + x|^p - |\xi|^p - p \xi |\xi|^{p-2} \cdot x.
\label{eq:g_(xi)}
\end{equation}
We have the following lemma:
\begin{lemme} There exist two constants $c, C > 0$ depending only on $p$ such that
\begin{equation}
\forall \xi, x \in \R^d, \quad c \left[ |x|^2|\xi|^{p-2} + |x|^p \right] \leq g_{\xi}(x) \leq C \left[ |x|^2|\xi|^{p-2} + |x|^p \right].
\label{eq:useful_ineq_4}
\end{equation}
\label{lem:B.1}
\end{lemme}

\begin{proof} This proof is elementary and will be omitted here.
\end{proof}

\begin{lemme} There exists a constant $C > 0$ such that for all $x, h \in \R$, 
\begin{equation}
\left| (x + h)^{1/(p-1)} - x^{1/(p-1)} \right| \leq C |h|^{1/(p-1)}.
\label{eq:1D_estim}
\end{equation}
\label{lem:1D_estim}
\end{lemme}

\begin{proof} This proof is elementary and will be omitted here.
\end{proof}

\begin{lemme}
Assume that Hypothesis~\textbf{(A4)'} is satisfied. Then $\mathcal{C}_0^{\infty}(\RR^d)$ is dense in $W_{\xi + \nabla w_{\xi}^{\per}}$.
\label{lem:dense}
\end{lemme}

\begin{proof}[Proof of Lemma~\ref{lem:dense}]
Let $v \in W_{\xi + \nabla w_{\xi}^{\mathrm{per}}}$ and $\varepsilon > 0$. There exists $R = R(\varepsilon) >1$ such that 
\begin{equation}
\left\||\xi + \nabla w_{\xi}^{\per}|^{\frac{p-2}{2}}\nabla v \right\|_{L^2(Q_{R/2}^c)} + \left\|\nabla v \right\|_{L^p(Q_{R/2}^c)} < \varepsilon.
\label{eq:lemdense_4}
\end{equation}
Let $\chi_R$ be a cut-off function such that $\chi_R = 1$ in $Q_{R/2}$ and $\chi = 0$ in $Q_{3R/4}^c$. We have that $|\chi_R | + R|\nabla \chi_R| \leq C$ where $C$ depends only on the dimension $d$ (and in partiuclar not on $R$). We introduce 
$$w_R := \left( v - \fint_{Q_{R}\setminus Q_{R/2} } v \right)\chi_R,$$
We have immediately that $w_R$ is compactly supported in $Q_R$ and that $w_R \in W^{1,p}_0(Q_R)$. Thus there exists a function $\Phi \in \mathcal{C}^{\infty}_{0}(Q_R)$ such that  
\begin{equation}
\left\|  w_R - \Phi \right\|_{W^{1,p}(Q_R)} \leq \varepsilon R^{- \frac{dp}{p-2}}  \quad \left(\leq \varepsilon \right).
\label{eq:lemdense_2}
\end{equation}
We extend $\Phi$ by zero outside $Q_R$.
By H\"{o}lder inequality, we have that \begin{equation}
\left\|  w_R - \Phi \right\|_{H^1_0(Q_R)} \leq \varepsilon.
\label{eq:lemdense_2bis}
\end{equation}
We next show that
\begin{equation}
\left\| v - \Phi \right\|_{W_{\xi}+\nabla w_{\xi}^{\per}} \leq C(\xi,d,p,a^{\per},C_{\text{poinc}})\varepsilon,
\label{eq:lemdense_ccl}
\end{equation}
where $C_{\text{poinc}}$ denotes the maxmimum between the $L^p$ Poincar\'e-Wirtinger constant on $Q\setminus Q_{1/2}$ and the weighted $L^2$ Poincar\'e-Wirtinger constant, given by Assumption \textbf{(A4)'}, on $Q\setminus Q_{1/2}$.
By the triangle inequality, we have that  
\begin{equation}
\left\| v - \Phi \right\|_{W_{\xi}+\nabla w_{\xi}^{\per}} \leq \left\| v - w_R \right\|_{W_{\xi}+\nabla w_{\xi}^{\per}} + \left\| w_R - \Phi \right\|_{W_{\xi}+\nabla w_{\xi}^{\per}}.
\label{eq:lemdense_1}
\end{equation}
We study separately each term of~\eqref{eq:lemdense_1}. By Proposition \ref{prop:periodique} (iv), \eqref{eq:lemdense_2} and~\eqref{eq:lemdense_2bis}, we have that 
$$\begin{aligned}
\left\| w_R - \Phi \right\|&_{W_{\xi}+\nabla w_{\xi}^{\per}} = \left\|\nabla w_R - \nabla\Phi \right\|_{L^p(Q_R)} + \left\| |\xi +\nabla w_{\xi}^{\per}|^{\frac{p-2}{2}}\left(\nabla w_R - \nabla \Phi\right) \right\|_{L^2(Q_R)}
\\ & \leq \left\|w_R - \Phi \right\|_{W^{1,p}(Q_R)} + C|\xi|^{\frac{p-2}{2}}\left\| w_R - \Phi \right\|_{H^1(Q_R)} \leq \left(C(d,p,a^{\per})|\xi|^{\frac{p-2}{2}} + 1\right)\varepsilon.
\end{aligned}$$
As for the first term of~\eqref{eq:lemdense_1}, we write that 
$$\begin{aligned}
\nabla v - \nabla w_R = \nabla v (1-\chi_R) + \frac{1}{R}\left( v - \fint_{Q_{R}\setminus Q_{R/2} } v \right)\nabla\chi(./R)
\end{aligned}$$
Thus, applying the $L^p$ Poincar\'e-Wirtinger inequality, we have that 
\begin{equation}
\left\|\nabla v - \nabla w_R \right\|_{L^p(\RR^d)} \leq  \left\|\nabla v \right\|_{L^p(Q_{R/2}^c)} + C_{\text{poinc}}  \left\|\nabla v \right\|_{L^p(Q_R \setminus Q_{R/2})} \underset{\eqref{eq:lemdense_4}}{\leq} (1+C_{\text{poinc}})\varepsilon.
\label{eq:lemdense_5}
\end{equation}
As for the $L^2 \left(|\xi + \nabla w_{\xi}^{\per}|^{\frac{p-2}{2}} \dint \lambda\right)$ norm, we use Assumption \textbf{(A4)'} to obtain that
\begin{equation}
\begin{aligned}
\left\|\nabla v - \nabla w_R \right\|_{L^2 \left(|\xi + \nabla w_{\xi}^{\per}|^{\frac{p-2}{2}} \dint \lambda\right)} &\leq  \left\||\xi + \nabla w_{\xi}^{\per}|^{\frac{p-2}{2}}\nabla v \right\|_{L^2(Q_{R/2}^c)} \\ & + C_{\text{poinc}}  \left\||\xi + \nabla w_{\xi}^{\per}|^{\frac{p-2}{2}}\nabla v \right\|_{L^2(Q_R \setminus Q_{R/2})} \leq \left(1+C_{\text{poinc}}\right)\varepsilon.
\end{aligned}
\label{eq:lemdense_6}
\end{equation}
Gathering together~\eqref{eq:lemdense_1},~\eqref{eq:lemdense_5} and~\eqref{eq:lemdense_6}, we get~\eqref{eq:lemdense_ccl} and conclude the proof of the Lemma.
\end{proof}

\bibliographystyle{plain}
\bibliography{nonlin}
\end{document}